\documentclass[11pt,a4paper]{article}

\usepackage[utf8]{inputenc}
\usepackage[T1]{fontenc}
\usepackage{amsmath,amssymb,amsthm}
\numberwithin{equation}{section}
\usepackage{mathtools}
\usepackage[margin=2.5cm]{geometry}
\usepackage{hyperref}
\hypersetup{colorlinks=true,linkcolor=blue,citecolor=blue,urlcolor=blue}
\usepackage{booktabs}
\usepackage{array}
\newcolumntype{L}[1]{>{\raggedright\arraybackslash}p{#1}}
\usepackage{graphicx}
\usepackage{xcolor}
\usepackage{enumitem}
\usepackage{microtype}
\usepackage{natbib}
\newtheorem{theorem}{Theorem}[section]
\newtheorem{proposition}[theorem]{Proposition}
\newtheorem{lemma}[theorem]{Lemma}

\newtheorem{definition}[theorem]{Definition}
\newtheorem{remark}[theorem]{Remark}

\newcommand{\R}{\mathbb{R}}
\newcommand{\C}{\mathbb{C}}
\newcommand{\eps}{\varepsilon}
\newcommand{\bfbeta}{\boldsymbol{\beta}}
\newcommand{\bfxi}{\boldsymbol{\xi}}
\newcommand{\norm}[1]{\left\|#1\right\|}
\newcommand{\abs}[1]{\left|#1\right|}
\newcommand{\Lstar}{L^{*}_{h}}
\newcommand{\Kh}{K_{h}}
\newcommand{\Sh}{S_{h}}
\newcommand{\Icd}{I^{\mathrm{cd}}_{h}}
\newcommand{\Pe}{\mathrm{Pe}}
\newcommand{\Peh}{\mathrm{Pe}_{h}}
\newcommand{\apq}{a_{pq}}
\newcommand{\bpq}{b_{pq}}
\newcommand{\lhpq}{\lambda_{h}(\theta_{p},\theta_{q})}
\newcommand{\rpq}{\rho_{pq}}

\begin{document}

\title{\textbf{Modal-Rectification-Based Directional Edge Diffusion\\
for Cartesian Convection--Diffusion Problems}}

\author{%
	Gossrin Jean-Marc Bomisso\thanks{%
		Universit\'e Nangui Abrogoua, UFR Sciences Fondamentales et Appliqu\'ees,
		02~BP~801 Abidjan~02, C\^ote d'Ivoire.
		Corresponding emails: \texttt{bogojm@yahoo.fr},
		\texttt{benzeus96@gmail.com}.}
	\and
	Ali Ouattara Kouma\footnotemark[1]}

\date{}

\maketitle

\begin{abstract}
Centered finite-difference discretizations of convection--diffusion equations
may oscillate when convection dominates at the mesh scale. For homogeneous
Dirichlet problems with constant coefficients on uniform Cartesian grids, we
derive ADSC (Adaptive Directional Sparse Correction), a local directional
edge-diffusion correction guided by modal rectification of the centered-stencil
Fourier symbol. The ideal modal reference damps modes independently, but its
exact nodal action is nonlocal; ADSC replaces it by a nearest-neighbor
positive semidefinite correction.

For a regularized operator with activation fixed by an auxiliary sequence, we
prove consistency, fixed-\(\eps\) energy stability, and conditional discrete
\(H^1\)-seminorm convergence. The implemented iteration instead uses
activation generated by the computed solution. For that fully coupled
nonlinear problem we prove existence and qualitative \(L^2\)
compactness/convergence only; uniqueness, convergence of activation updates,
and energy-norm rates remain open. Numerical tests show selective extrema
control, reduced modal-dominance indicators, and a low-cost few-shot variant.
Comparisons with upwinding, SUPG, and AFC-inspired strategies are diagnostic
rather than claims of uniform superiority.
\end{abstract}



\section{Introduction}


Convection--diffusion problems in convection-dominated regimes are
central to the numerical simulation of transport phenomena, notably
in aerodynamics, meteorology, and chemical engineering.
This work considers the two-dimensional model problem
\begin{equation}\label{eq:strong}
	\begin{cases}
		-\eps\,\Delta u + \bfbeta\cdot\nabla u = f, & \text{in }
		\Omega = (0,1)^{2},                                                     \\
		u = 0,                                      & \text{on }\partial\Omega,
	\end{cases}
\end{equation}
where $\bfbeta=(\beta_1,\beta_2)\in\R^2$ is a constant convection vector and
$|\bfbeta|=\sqrt{\beta_1^2+\beta_2^2}$. We focus on regimes in which the
mesh P\'eclet number $\Peh=|\bfbeta|h/(2\eps)$ can be larger than one; this
number measures the ratio between convective and diffusive effects at the
mesh scale.
When $\Peh>1$, standard centered Galerkin discretizations can develop
spurious oscillations, especially near internal layers and sharp convective
fronts~\citep{Roos2008}. What remains less explicit in many local
stabilizations is the modal geometry of the dissipative footprint in two
space dimensions: which frequency regions are damped, in which directions,
and by what local edge mechanism.

\subsection*{Related work}

Classical stabilizations for~\eqref{eq:strong} include
SUPG~\citep{Brooks1982}, artificial
viscosity~\citep{Johnson1981}, upwind
schemes~\citep{Ilin1969}, discontinuous Galerkin
methods with upwind numerical fluxes~\citep{Cockburn2001}, and SOLD-type
methods~\citep{Mizukami1985,John2007}. Related frameworks include
Galerkin/least-squares formulations~\citep{Hughes1989},
continuous interior penalty (CIP)
methods~\citep{Burman2004,Burman2007}, local projection
stabilization (LPS)~\citep{Braack2006}, and algebraic flux
correction (AFC)~\citep{Kuzmin2005}.
Further algorithmic developments include variants of algebraic flux
correction with discrete maximum-principle analysis~\citep{Barrenechea2016},
entropy-viscosity stabilization~\citep{Guermond2011}, and stabilized
finite element formulations
for advective--diffusive models~\citep{Franca1992,ErnGuermond2004},
robust analysis of singularly perturbed convection--diffusion
problems~\citep{Stynes2005}, a posteriori error estimation and adaptive
refinement~\citep{Verfurth1996}, and local Fourier analysis for multigrid
discretizations~\citep{MacLachlan2011}.
Recent surveys emphasize
that finite-element stabilization for convection-dominated transport
remains active, especially when monotonicity, anisotropy, and
unstructured meshes are combined~\citep{Tobiska1996,John2018,Barrenechea2024}.
These surveys do not construct a nearest-neighbor edge correction from a
constructed rectified modal reference operator; this modal-design viewpoint
is the specific angle pursued here.

The Fourier-symbol approach used here is standard in local Fourier
analysis (LFA)~\citep{Trottenberg2001,Bolten2018} and in von~Neumann
stability analysis~\citep{LeVeque2007,Strikwerda2004}. The present paper
uses the local Fourier symbol for a specific purpose: to construct a two-dimensional
convection-dominance map, a directional modal rectification principle,
a nonlocal rectified reference operator, and a nearest-neighbor
edge-diffusion surrogate designed to reproduce selected dissipative
effects of that reference operator.

Unlike spectral viscosity methods~\citep{Tadmor1989} or global modal
filters used in spectral discretizations~\citep{Boyd2001}, the correction
considered here is local in physical space and uses a nearest-neighbor
sparse stencil. Unlike SOLD or AFC techniques, its local detector is
motivated by a directional modal rectification mechanism rather than being
primarily formulated as a flux-limiting construction. The purpose of the
paper is not to replace established monotonicity-preserving or
streamline-diffusion methods, but to derive a local directional edge
correction from a modal rectification viewpoint and to quantify what this
design achieves in the Cartesian constant-coefficient setting.

\subsection*{Main contributions}

The paper makes six contributions.
\begin{enumerate}[label=(\roman*)]
	\item We build a two-dimensional modal map from the local Fourier symbol of
	      the centered stencil, including the convection-dominance indicator
	      \(\rho_{pq}\), the dominant set \(\Icd\), and the anisotropic modal
	      footprint (Section~\ref{sec:symbol}).

	\item We introduce a constructed rectified modal reference operator and
	      prove that exact modal rectification is nonlocal in the nodal basis
	      (Proposition~\ref{prop:nonlocality}). The Frobenius discussion is
	      retained only as a short structural diagnostic of information loss
	      under forced sparsification, not as a design rule for ADSC
	      (Section~\ref{sec:rectification}).

	\item We interpret classical stabilizations through their structured-grid
	      modal footprints and position ADSC relative to upwinding, SUPG,
	      CIP, LPS, and AFC-type mechanisms (Section~\ref{sec:stabilizations}).

	\item We formulate ADSC as one specified nearest-neighbor,
	      \(\bfbeta\)-directional, symmetric positive semidefinite edge
	      correction (Section~\ref{sec:adsc}).

	\item We distinguish the sharp active-set detector from a regularized
	      detector. For the fixed-activation operator selected from an
	      auxiliary activation sequence, the regularized detector yields the
	      coefficient-variation estimate used in the consistency and
	      energy-stability analysis. For the fully coupled activation used in
	      computations, only existence and qualitative convergence are proved;
	      uniqueness, activation-iteration convergence, and energy-norm rates
	      remain open.

	\item We provide numerical evidence for this modal-design construction,
	      including the main two-dimensional benchmark, parameter and
	      direction sensitivity, few-shot activation, manufactured active and
	      inactive regimes, and NIST-inspired exponential boundary-layer tests (standard sharp exponential boundary-layer
	      benchmarks) on uniform and Shishkin meshes (Section~\ref{sec:numerics}).
\end{enumerate}

\subsection*{Scope}

The analysis is restricted to uniform Cartesian meshes and constant
coefficients. This is the setting in which the interior stencil has an
explicit two-frequency Fourier symbol and in which the nonlocality of the
rectified modal reference operator can be stated cleanly. Nonuniform meshes,
variable coefficients, and finite-element discretizations are not routine
extensions of the present proof, because the tensor-product modal basis and
translation-invariant symbol are then lost.
Unless explicitly stated as a numerical robustness test, all analytical
claims below concern this uniform Cartesian, constant-coefficient,
homogeneous Dirichlet setting.

In particular, even the structured criss-cross \(P_1\) finite-element case
would require a mass-matrix-based modal symbol and an edge-directional
balance involving the triangular edge directions. The present paper should
therefore be read as a modal mechanism study and local edge-diffusion design
on Cartesian grids, with finite-element and variable-coefficient versions
left as separate developments.

\section{Discrete modal symbol and convection-dominance indicator}
\label{sec:symbol}

\subsection{Weak formulation}

The weak formulation of~\eqref{eq:strong} reads: find $u\in V=H^1_0(\Omega)$
such that $a(u,v)=\ell(v)$ for all $v\in V$, where
\[
	a(w,v)=\eps\int_\Omega\nabla w\cdot\nabla v\,dx
	+\int_\Omega(\bfbeta\cdot\nabla w)\,v\,dx,
	\qquad
	\ell(v)=\int_\Omega f\,v\,dx.
\]
This weak form is recalled only to make explicit the skew-symmetry that is
also mirrored by the centered finite-difference convection matrix used
below. Since $\nabla\cdot\bfbeta=0$ and $u=v=0$ on $\partial\Omega$, the
convective part is skew-symmetric. Indeed, for $u,v\in H^1_0(\Omega)$,
\[
	\int_\Omega (\bfbeta\cdot\nabla u)v\,dx
	= -\int_\Omega u(\bfbeta\cdot\nabla v)\,dx.
\]

\subsection{Interior stencil}

On a uniform Cartesian mesh of step $h=1/N_e$, with $N=N_e-1$ interior
nodes per direction, the centered stencil is
\begin{equation}\label{eq:stencil}
	(\Kh U)_{ij} =
	\frac{\eps}{h^2}(4U_{ij}-U_{i+1,j}-U_{i-1,j}-U_{i,j+1}-U_{i,j-1})
	+\frac{\beta_1}{2h}(U_{i+1,j}-U_{i-1,j})
	+\frac{\beta_2}{2h}(U_{i,j+1}-U_{i,j-1}).
\end{equation}
This stencil describes the local action of the centered discrete operator
for all interior nodes \((i,j)\in\{1,\ldots,N\}^{2}\). It is not intended
as a global diagonalization of the full finite matrix, and boundary effects
are not captured by the symbol analysis below.
In the finite Dirichlet computations, \(\Kh\) denotes the
\(N^2\times N^2\) interior matrix obtained from this stencil after setting
all boundary values to zero.

\subsection{Interior LFA symbol}

\begin{remark}[Relation to LFA]\label{rem:lfa}
	The symbol below is the local Fourier analysis symbol of the interior
	stencil~\eqref{eq:stencil}, not the spectrum of the full finite Dirichlet
	matrix. It is used to define a modal convection-dominance indicator, a
	rectified reference operator, and a direction-aware sparse correction.
	Boundary layers and Dirichlet effects are assessed in the numerical
	Dirichlet solves; the modal quantities reported later are interior-stencil
	diagnostics.
\end{remark}

\begin{definition}[Discrete modal symbol]\label{def:symbol}
	Throughout Section~\ref{sec:symbol}, let \(K_h^{\mathrm{int}}\) denote the translation-invariant interior
	stencil operator used in local Fourier analysis, distinct from the finite
	Dirichlet matrix \(\Kh\). For
	$\theta=(\theta_1,\theta_2)\in(0,\pi)^2$, the discrete modal symbol of
	the stencil~\eqref{eq:stencil} is the complex number
	$\lambda_h(\theta_1,\theta_2)\in\C$ defined by
	$K_h^{\mathrm{int}}\psi_\theta =
	\lambda_h(\theta_1,\theta_2)\psi_\theta$,
	where $\psi_\theta(r,s)=\exp\!\left(\mathrm{i}(r\theta_1+s\theta_2)\right)$.
\end{definition}

\begin{proposition}[Expression of the discrete modal symbol]
	\label{prop:symbol}
	The discrete modal symbol of~\eqref{eq:stencil} is
	\begin{equation}\label{eq:symbol}
		\lambda_h(\theta_1,\theta_2)
		= a_h(\theta_1,\theta_2)+\mathrm{i}\,b_h(\theta_1,\theta_2),
	\end{equation}
	where
	\begin{equation}\label{eq:ab}
		a_h(\theta_1,\theta_2)=\frac{2\eps}{h^2}(2-\cos\theta_1-\cos\theta_2),
		\qquad
		b_h(\theta_1,\theta_2)=\frac{1}{h}(\beta_1\sin\theta_1+\beta_2\sin\theta_2).
	\end{equation}
\end{proposition}

\begin{proof}
	For \(\psi_\theta(r,s)=\exp(\mathrm{i}(r\theta_1+s\theta_2))\), the shifts
	multiply the mode by \(e^{\pm\mathrm{i}\theta_1}\) and
	\(e^{\pm\mathrm{i}\theta_2}\). Substituting these factors into
	\eqref{eq:stencil} and using
	\(e^{\mathrm{i}\theta}+e^{-\mathrm{i}\theta}=2\cos\theta\) and
	\(e^{\mathrm{i}\theta}-e^{-\mathrm{i}\theta}=2\mathrm{i}\sin\theta\)
	gives exactly \eqref{eq:symbol}--\eqref{eq:ab}.
\end{proof}

The preceding LFA remark, symbol definition, and symbol formula delimit the
local-symbol framework used in the rest of this section.

\subsection{Modified equation}

Before introducing the modal dominance indicator, we record the
modified equation associated with the same real/imaginary splitting of the
symbol.

Modified equation analysis is used only as a formal diagnostic of the
dispersive and dissipative terms introduced by the centered stencil, in
the classical sense of~\citet{Warming1974}.

\begin{lemma}[Formal modified equation]\label{lem:modified}
	For \(u\in C^6(\overline\Omega)\), applying stencil~\eqref{eq:stencil} to the
	exact solution yields
	\begin{equation}\label{eq:modified}
		-\eps\Delta u+\bfbeta\cdot\nabla u
		+\frac{h^2}{6}(\beta_1\partial_{xxx}u+\beta_2\partial_{yyy}u)
		-\frac{\eps h^2}{12}(\partial_{xxxx}u+\partial_{yyyy}u)
		= f+O(h^4).
	\end{equation}
\end{lemma}

\begin{proof}
	The centered first and second differences satisfy
	\[
	D^0_xu=\partial_xu+\frac{h^2}{6}\partial_{xxx}u+O(h^4),\qquad
	-D^2_xu=-\partial_{xx}u-\frac{h^2}{12}\partial_{xxxx}u+O(h^4),
	\]
	and analogously in the \(y\)-direction. Substitution into
	\eqref{eq:stencil} and use of
	\(-\eps\Delta u+\bfbeta\cdot\nabla u=f\) gives \eqref{eq:modified}.
	No mixed derivative appears in this expansion because the stencil is the
	sum of one-dimensional centered differences in the coordinate directions.
\end{proof}

The odd third-order terms $(\beta_1\partial_{xxx}u+\beta_2\partial_{yyy}u)$
	are dispersive and directional: they modify the phase of discrete modes
	without adding symmetric dissipation. This is the physical-space
	counterpart of the imaginary part $b_h$ of the modal symbol. It explains
	why centered convection primarily introduces phase errors rather than
	additional dissipation, consistent with the purely imaginary convective
	contribution \(b_h\) in~\eqref{eq:symbol}. Lemma~\ref{lem:modified} is
	therefore used only as a diagnostic motivation for selective symmetric
	damping.


\subsection{Modal convection-dominance indicator}

Set $\theta_p=p\pi/(N+1)$, $\theta_q=q\pi/(N+1)$ for $p,q=1,\ldots,N$.

\begin{definition}[Modal convection-dominance indicator]
	\label{def:indicator}
	The modal convection-dominance indicator for the pair $(p,q)$ is
	\begin{equation}\label{eq:indicator}
		\rpq
		= \frac{\abs{b_h(\theta_p,\theta_q)}}{a_h(\theta_p,\theta_q)}
		= \frac{\abs{\operatorname{Im}(\lhpq)}}{\operatorname{Re}(\lhpq)}.
	\end{equation}
	The pair $(p,q)$ is \emph{convectively dominant} when $\rpq>1$,
	and the convectively dominant modal set is
	$\Icd=\{(p,q)\in\{1,\ldots,N\}^2:\rpq>1\}$.
	For these interior modal pairs, \(a_h(\theta_p,\theta_q)>0\), because
	\(\theta_p,\theta_q\in(0,\pi)\); hence the denominator in
	\eqref{eq:indicator} is never zero.
	The threshold \(\rpq>1\) means that the skew convective contribution of
	the local symbol is larger than the symmetric diffusive contribution for
	that mode.
\end{definition}

Definition~\ref{def:indicator} supplies the modal diagnostic used in the parameter rule and in the numerical tables.

\paragraph{Notation used for modal averages.}
Throughout the paper, $(p,q)$ always denotes an element of
	$\{1,\ldots,N\}^2$, where \(N=N_e-1\) is the number of interior nodes per
	coordinate direction. Averages over \(\Icd\), for example
	\(\bar\rho_{\mathrm{Gal}}=|\Icd|^{-1}\sum_{(p,q)\in\Icd}\rho_{pq}\), are therefore
	averages over modal pairs, not over physical nodes. The vector
	\(\bfxi_{pq}=(\theta_p/h,\theta_q/h)^T\) denotes the associated wave
	vector, while \(\phi_{pq}\) denotes the tensor-product sine mode.

\begin{remark}[Interpretation: modal P\'eclet number]
	\label{rem:peclet_modal}
	For the one-dimensional slice $\theta_2=0$ with $\bfbeta=(\beta,0)$:
	\[
		\rho_k
		=\frac{|\beta|\sin\theta_k/h}{2\eps(1-\cos\theta_k)/h^2}
		=\Peh\cdot\frac{\sin\theta_k}{1-\cos\theta_k}.
	\]
	As $\theta_k\to0^+$,
	\[
		\rho_k
		=\Peh\frac{\sin\theta_k}{1-\cos\theta_k}
		=\frac{2\Peh}{\theta_k}+O(\Peh\,\theta_k),
	\]
	so that $\rho_k$ becomes large for low-frequency modes when the
	symmetric diffusive contribution is small compared with the skew
	convective part. For $\theta_k=\pi/2$, one has $\rho_k=\Peh$.
	Thus $\rpq$ generalizes the mesh P\'eclet number to individual modal
	pairs, providing a mode-resolved convection-dominance measure. The
	threshold $\rpq>1$ identifies pairs for which the oscillatory convective
	contribution exceeds the symmetric diffusive contribution of the symbol
	-- a diagnostic indicator of phase-dominant modal behavior, meaning that
	the imaginary convective part of the symbol exceeds the real diffusive part,
	not a spectral instability criterion in the eigenvalue sense. The singular
	low-frequency behavior is not used as a proof that all long waves must
	be artificially damped. It records the fact that the symmetric diffusive
	symbol vanishes quadratically while the skew convective symbol vanishes
	linearly. In ADSC, this modal information is filtered through a local
	directional non-monotonicity detector and a bounded edge-diffusion
	scale; \(\rho_{pq}\) is therefore a design diagnostic, not an error estimator.
\end{remark}

Remark~\ref{rem:peclet_modal} explains why the indicator may become large even for low frequencies, so the subsequent ADSC design does not use \(\rho_{pq}\) alone but couples it with a spatial directional detector.

\subsection{Two-dimensional modal footprint}
\label{sec:footprint}

The \emph{modal footprint} is the image of the frequency domain under
the symbol map:
\[
	\mathcal{F}_h
	= \{\lambda_h(\theta_1,\theta_2):(\theta_1,\theta_2)\in(0,\pi)^2\}
	\subset\C.
\]
The real part satisfies $0<a_h<8\eps/h^2$ on $(0,\pi)^2$, and
$\abs{b_h}\le(|\beta_1|+|\beta_2|)/h$.
Ellipse-shaped curves appear only on one-dimensional slices; the full
two-dimensional footprint is a two-dimensional region in the complex plane whose shape
depends on the direction of $\bfbeta$. This is the geometric object that
distinguishes the two-dimensional analysis from its one-dimensional
reduction. Proposition~\ref{prop:footprint_interior} formalizes this two-dimensional feature.

\begin{proposition}[Nonempty interior of the modal footprint]
\label{prop:footprint_interior}
If \(\bfbeta\ne0\), then the footprint \(\mathcal F_h\) has nonempty interior
in \(\mathbb C\).
\end{proposition}

\begin{proof}
Write the symbol map in real and imaginary parts as
\((a_h,b_h)\). Its Jacobian determinant is
\[
	J(\theta_1,\theta_2)
	=
	\frac{2\eps}{h^3}
	\bigl(\beta_2\sin\theta_1\cos\theta_2
	-
	\beta_1\sin\theta_2\cos\theta_1\bigr).
\]
If \(\beta_2\ne0\), choose \(\theta_2\in(0,\pi/2)\) strictly positive
and sufficiently small; then \((\pi/2,\theta_2)\in(0,\pi)^2\) and
\(J(\pi/2,\theta_2)\ne0\). If \(\beta_2=0\), then \(\beta_1\ne0\), and one
chooses \(\theta_1\in(0,\pi/2)\) strictly positive and sufficiently small so
that \((\theta_1,\pi/2)\in(0,\pi)^2\) and \(J(\theta_1,\pi/2)\ne0\). By the
inverse function theorem, the symbol map is locally open at the selected
interior point, and \(\mathcal F_h\) has nonempty interior. Since
\((0,\pi)^2\) is connected and the symbol map is continuous, the footprint
is connected as the continuous image of a connected set.
\end{proof}

\section{Reference modal operator and nonlocality}
\label{sec:rectification}

\subsection{Dirichlet modal basis}

Introduce the tensor-product sine modes
\[
	\phi_{pq}(i,j)=\frac{2}{N+1}
	\sin\!\Bigl(\frac{ip\pi}{N+1}\Bigr)\sin\!\Bigl(\frac{jq\pi}{N+1}\Bigr),
	\quad p,q=1,\ldots,N.
\]
These form an orthonormal basis of $\R^{N^2}$ after the interior grid
points are ordered lexicographically row-by-row, and are compatible with
homogeneous Dirichlet boundary
conditions. The complex exponential modes $\psi_{pq}$ are used only
for the interior stencil symbol and footprint; they do not satisfy
homogeneous Dirichlet conditions. The constructions below are therefore
reference modal constructions associated with the interior symbol, not
exact diagonalizations of the full finite matrix.

\subsection{Modal rectification and reference operator}

Set $\apq=\operatorname{Re}(\lhpq)$, $\bpq=\operatorname{Im}(\lhpq)$.
\emph{Directional modal rectification} replaces the complex symbol
$\lhpq=\apq+\mathrm{i}\bpq$ by the real positive symbol
$\abs{\lhpq}=\sqrt{\apq^2+\bpq^2}$. Since
\(|\lhpq|=\sqrt{\apq^2+\bpq^2}\ge\apq\), with strict inequality whenever
\(\bpq\ne0\), the modulus removes the modal phase associated with the skew
convective part while increasing the effective real part relative to the
purely symmetric contribution, without acting on the right-hand side. The purpose of this rectification is not to replace the physical convection operator by a different equation. It is to identify which modal components would require additional symmetric damping if one wants to reduce phase-driven extrema violations while keeping a directional interpretation. ADSC uses this rectified symbol only as a design target: the actual method remains the centered convection--diffusion discretization augmented by a local positive edge correction.

\begin{definition}[Discrete reference modal operator]\label{def:lstar}
	The discrete reference modal operator $\Lstar$ is defined in the
	tensor-product modal basis by
	$\Lstar\phi_{pq}=\abs{\lhpq}\phi_{pq}$, $p,q=1,\ldots,N$.
\end{definition}

\paragraph{Constructed reference operator and symbolic consistency.}
The matrix \(\Phi\) is used here only as a Dirichlet-compatible
	orthonormal modal basis in which the sampled interior-symbol values
	\(|\lambda_h(\theta_p,\theta_q)|\) are assigned to the corresponding
	sine modes. It does not diagonalize the full nonsymmetric Dirichlet
	matrix \(\Kh\); it diagonalizes only the standard symmetric tensor-product
	Dirichlet Laplacian. Thus \(\Lstar=\Phi D^*\Phi^T\) is a constructed
	reference modal operator, not a spectral factorization or functional
	calculus of \(\Kh\). Its role is to describe the ideal dissipative effect
	of modal rectification and to serve as a theoretical target for the
	design of sparse local surrogates.

	As clarified below, this operator is not introduced as a consistent discretization of the
	original convection--diffusion operator. The original interior symbol is
	\(a_h+\mathrm{i}b_h\), whereas \(\Lstar\) uses the positive symbol
	\(|a_h+\mathrm{i}b_h|\). Its consistency is therefore symbolic: for
	low frequencies \(\theta=h\xi\), it approximates the modulus of the
	principal convection--diffusion symbol, not the original skew convective
	operator itself. This distinction is essential: \(\Lstar\) is a
	rectified reference used to motivate local damping, while the ADSC
	operator below is the actual local stabilization added to the centered
	discretization.

\subsection{Formal symbolic limit on \texorpdfstring{\(\R^2\)}{R2}}

This subsection is interpretive and is not used in the convergence proof.
In the formal whole-space limit, as the local stencil symbol is replaced
by its continuous counterpart on \(\R^2\) with constant \(\bfbeta\), modal
rectification corresponds to the operator with symbol
\[
	\mu^*(\bfxi)
	= \sqrt{\eps^2|\bfxi|^4+(\bfbeta\cdot\bfxi)^2},
	\quad \bfxi\in\R^2.
\]
Indeed, for $e^{\mathrm{i}\bfxi\cdot x}$, the symbol of $-\eps\Delta$ is
$\eps|\bfxi|^2$, and the symbol of $\bfbeta\cdot\nabla$ is
$\mathrm{i}\bfbeta\cdot\bfxi$. Replacing the complex symbol
$\eps|\bfxi|^2+\mathrm{i}\bfbeta\cdot\bfxi$ by its modulus gives
$\mu^*(\bfxi)$. The function \(\mu^*(\bfxi)\) is therefore the symbol of
the formal pseudo-differential square root of
\(\eps^2\Delta^2-(\bfbeta\cdot\nabla)^2\), i.e., at the symbolic level,
$L^* = \bigl[\eps^2(-\Delta)^2+(\bfbeta\cdot\nabla)^*(\bfbeta\cdot\nabla)\bigr]^{1/2}$,
where $(\bfbeta\cdot\nabla)^*=-\bfbeta\cdot\nabla$ is the formal
$L^2(\R^2)$-adjoint (valid since $\bfbeta$ is constant and divergence-free).

This is only a whole-space symbolic statement. On the bounded domain with
Dirichlet conditions, no functional calculus for $\Kh$ is invoked; the
expression is used only to interpret modal rectification geometrically at the
continuous level. In particular, $\mu^*(\bfxi)$ depends on the projection of
$\bfxi$ onto $\bfbeta$, reflecting the directional nature of the
rectification. No analytical functional calculus for the finite operator
\(L_h^\ast\) is inferred from this formal limit. The displayed whole-space
expression should not be read as an operator-theoretic statement on
\(H^1_0(\Omega)\) for the bounded Dirichlet problem; only its Fourier-symbol
interpretation is used in the sequel.

\subsection{Nonlocality of exact modal rectification}

Let $\Phi$ be the orthogonal matrix whose columns are the ordered modes
$\phi_{pq}$. By Definition~\ref{def:lstar}, the reference operator is
\(\Lstar=\Phi D^*\Phi^T\), where
\(D^*=\operatorname{diag}(\abs{\lhpq})\). This identity is a definition of
\(\Lstar\) in the chosen sine basis, not a diagonalization of the full
Dirichlet matrix \(\Kh\). Because \(D^*\) is real positive diagonal and
\(\Phi\) is real orthogonal, \(\Lstar\) is a real symmetric positive
definite matrix.

\begin{proposition}[Nonlocality of exact modal rectification]
\label{prop:nonlocality}
Assume that \(\bfbeta\neq 0\). The rectified symbol
\[
m_h(\theta_1,\theta_2)
=
|\lambda_h(\theta_1,\theta_2)|
=
\sqrt{a_h(\theta_1,\theta_2)^2+b_h(\theta_1,\theta_2)^2}
\]
viewed as a \(2\pi\)-periodic continuous function on the
two-dimensional torus, is not a trigonometric polynomial.
Consequently, the reference modal operator \(L_h^\ast\), defined by
\[
L_h^\ast \phi_{pq}
=
m_h(\theta_p,\theta_q)\phi_{pq},
\]
cannot be represented exactly by a Cartesian stencil of fixed bandwidth
independent of \(N\). In the pure diffusion case \(\bfbeta=0\), one has
\(m_h=a_h\), and the rectified operator reduces to the standard local
Laplacian stencil.
\end{proposition}

\begin{proof}
A Cartesian stencil with fixed bandwidth has, at the level of the interior
Fourier representation on the torus, a symbol that is a trigonometric
polynomial in \((\theta_1,\theta_2)\). Hence an
exact fixed-stencil realization of \(L_h^\ast\), valid independently of the
grid resolution, would require the periodic rectified symbol
\(m_h=|\lambda_h|\) to coincide with such a trigonometric polynomial on the
torus.

We show that this is impossible when \(\bfbeta\neq0\). Near
\(\theta=0\), one has
\[
a_h(\theta)
=
\frac{2\varepsilon}{h^2}(2-\cos\theta_1-\cos\theta_2)
=
\frac{\varepsilon}{h^2}|\theta|^2+O(|\theta|^4),
\]
and
\[
b_h(\theta)
=
\frac{1}{h}(\beta_1\sin\theta_1+\beta_2\sin\theta_2)
=
\frac{\bfbeta\cdot\theta}{h}+O(|\theta|^3).
\]
Choose a direction \(e\in\mathbb R^2\) such that \(\bfbeta\cdot e\neq0\).
Then, as \(t\to0\),
\[
m_h(te)
=
\sqrt{a_h(te)^2+b_h(te)^2}
=
\frac{|\bfbeta\cdot e|}{h}|t|+O(t^2).
\]
Extend \(m_h\) periodically to \((-\pi,\pi)^2\). The point \(\theta=0\)
is then an interior point of the periodic chart of the torus, so differentiability
can be tested along lines \(te\) crossing the origin. If \(m_h\) coincided on
the torus with a trigonometric polynomial, then that polynomial would provide
a \(C^\infty\), indeed real analytic, periodic representative of \(m_h\) in a
neighborhood of the origin. However,
\(m_h(te)\) has a corner singularity (a non-differentiable point) at
\(t=0\). Indeed,
\[
	\frac{d}{dt}\Big|_{t=0^+}m_h(te)
	=\frac{|\bfbeta\cdot e|}{h}
	\ne
	-\frac{|\bfbeta\cdot e|}{h}
	=\frac{d}{dt}\Big|_{t=0^-}m_h(te).
\]
Every trigonometric polynomial is real analytic, hence \(C^1\), on the
torus~\citep{Zygmund2002}. Therefore \(m_h\) cannot be a trigonometric
polynomial.

It follows that no fixed-bandwidth Cartesian stencil can have \(m_h\) as
its exact Fourier symbol. Hence exact modal rectification is local in
frequency but nonlocal in the nodal representation. When \(\bfbeta=0\),
\(b_h=0\) and \(m_h=a_h\), which is precisely the standard degree-one
Laplacian symbol.
\end{proof}
The obstruction is structural: exact modal rectification is
	frequency-local (acting independently on each pair $(p,q)$) but
	physically nonlocal (generating long-range nodal couplings after basis
	change). This is the fundamental reason why practical stabilizations of
	fixed stencil width can only approximate the dissipative effects of
	ideal modal rectification. The observation is an instance, in this
	discrete-symbol setting, of the standard pseudo-differential fact that
	non-polynomial functions of differential symbols generally produce
	nonlocal operators~\citep{Taylor1981,Shubin2001}; the novelty claimed
	here is the consequence for the present modal rectification design, not
	the general pseudo-differential principle itself.

\subsection{Structural meaning of sparsification}

The Frobenius viewpoint is retained only as a structural warning. It says
that forcing the dense constructed rectified reference operator into a
prescribed sparse pattern loses nonlocal modal information. It is not used as
a stabilization rule for ADSC: a Frobenius projection would not by itself
preserve monotonicity, positivity, or energy stability. ADSC is instead
designed separately as a symmetric positive semidefinite nearest-neighbor edge
correction.

\subsection{Effect of the right-hand side}

Let $f=\sum_{p,q}\hat{f}_{pq}\phi_{pq}$. The solution of
$\Lstar u^*=f$ satisfies
\begin{equation}\label{eq:lstar_rhs_modal}
	\hat{u}^*_{pq}=\frac{\hat{f}_{pq}}{|\lhpq|}.
\end{equation}
Since
\[
	|\lhpq|=\sqrt{a_{pq}^{2}+b_{pq}^{2}}\ge a_{pq},
\]
modal rectification replaces the complex Galerkin symbol
$a_{pq}+\mathrm{i} b_{pq}$ by a real positive symbol with the same modulus.
Thus it removes the modal phase associated with the skew convective part while
retaining the magnitude of the complex symbol. Relative to the purely
symmetric diffusive part $a_{pq}$, it increases the effective real dissipation
on modes for which $b_{pq}\ne0$.

\section{Modal interpretation of classical stabilizations}
\label{sec:stabilizations}

\subsection{Dissipative correction and modal response}

Let $K^{\mathrm{stab}}_h=\Kh+\Sh$ with $\Sh=\Sh^T\ge 0$.
For each modal pair, define
$\delta_{pq}=\langle\Sh\phi_{pq},\phi_{pq}\rangle/\norm{\phi_{pq}}^2\ge 0$.
The stabilized modal response in the interior-stencil interpretation is
\[
	\hat{u}^{\mathrm{stab}}_{pq}
	\approx \frac{\hat{f}_{pq}}{\lhpq+\delta_{pq}},
\]
where ``$\approx$'' denotes the Rayleigh-quotient approximation of the
modal action of $\Sh$ via the sine modes $\phi_{pq}$ -- an approximation
valid when boundary effects are negligible, i.e., when the interior
region dominates (large $N$). This display is not used as an error-bound
statement; it is a modal scaling model for interpreting the dissipative
increment \(\delta_{pq}\). A symmetric positive correction increases
the effective real part of the modal symbol while leaving the convective
imaginary part $b_h$ unchanged.

\subsection{Coordinate upwinding and SUPG}

\paragraph{Coordinate upwinding.}
This method adds componentwise artificial diffusion aligned with the
coordinate axes, represented by the diffusion tensor
$D^{\mathrm{up}}=(h/2)\operatorname{diag}(|\beta_1|,|\beta_2|)$.
For the modal wave vector
$\bfxi_{pq}=(\theta_p/h,\theta_q/h)^T$, the associated modal
dissipative increment is proportional to
$\bfxi_{pq}^T D^{\mathrm{up}}\bfxi_{pq}
	=(h/2)(|\beta_1|\xi_p^2+|\beta_2|\xi_q^2)$.

\paragraph{SUPG.}
SUPG adds streamline diffusion with tensor
$D^{\mathrm{SUPG}}=\tau\bfbeta\otimes\bfbeta$,
$\tau\approx h/(2|\bfbeta|)$,
giving modal increment proportional to
$\bfxi_{pq}^T D^{\mathrm{SUPG}}\bfxi_{pq}=\tau(\bfbeta\cdot\bfxi_{pq})^2$.

\paragraph{Key distinction.}
Coordinate upwinding damps coordinate-aligned oscillations according
to the individual components of $\bfbeta$; SUPG damps preferentially
along the streamline direction. This follows directly from the classical
upwind and SUPG diffusion tensors~\citep{Brooks1982,Roos2008,John2018};
we do not claim the tensor distinction itself as new. The contribution
here is to insert this distinction into the same interior-symbol footprint
used for the ADSC design, so that upwinding, SUPG, and ADSC are compared
mode by mode in the two-dimensional Cartesian setting. This distinction is invisible in the
reduced one-dimensional interior-symbol setting because the coordinate
axis coincides with the streamline, so both corrections add diffusion in
the same direction and reduce to
$|\beta|h/2$ for the same local stencil normalization. Indeed, for
\(\bfbeta=(\beta,0)\) and \(\tau=h/(2|\beta|)\),
\[
	\tau(\bfbeta\cdot\bfxi)^2
	=\frac{h}{2|\beta|}\beta^2\xi^2
	=\frac{|\beta|h}{2}\xi^2,
\]
which is the same modal diffusion increment as coordinate upwinding in
one dimension. The distinction becomes
essential in two dimensions whenever $\bfbeta$ is not aligned with a
coordinate axis. This statement concerns the local interior symbol; it is
specific to the finite-difference streamline correction used here. Full
finite-element SUPG matrices, especially with consistent mass matrices and
Dirichlet boundary conditions, need not be identical to upwind matrices on
bounded domains.

\paragraph{Why exact isotropic modal correction is not practical.}
As a reference case, suppose the ideal correction is restricted to isotropic
scalar artificial diffusion. Then
the scalar coefficient for mode $(p,q)$ satisfying
$\apq+\alpha^*_{pq}\,(2/h^2)(2-\cos\theta_p-\cos\theta_q)
	=\abs{\lhpq}$
evaluates to
\begin{equation}\label{eq:alpha_star}
	\alpha^*_{pq}
	= \eps\left(\frac{1}{\cos\varphi_{pq}}-1\right),
	\quad\varphi_{pq}=\arctan\rpq.
\end{equation}
This follows from
\(\apq+\alpha^*_{pq}\,2(2-\cos\theta_p-\cos\theta_q)/h^2=|\lhpq|\)
and \(\apq=|\lhpq|\cos\varphi_{pq}\).
This coefficient depends on $(p,q)$ and diverges as $\rpq\to\infty$
(modes with large $\rho_{pq}$ would require very large isotropic
diffusion to be fully rectified). A uniform artificial diffusion
corresponds to a single scalar applied to all modes, inevitably
over-damping low-frequency modes and under-damping high-frequency ones.
Thus the ideal isotropic correction is not a practical prescription; its
role is only to expose the nonlocal and mode-dependent character of exact
rectification, which ADSC approximates through bounded local edge diffusion.
The coefficient \(\alpha^*_{pq}\) in~\eqref{eq:alpha_star} is not used as a numerical parameter in
ADSC.

\subsection{Schematic CIP, LPS, and AFC modal positioning}

Table~\ref{tab:method_positioning} summarizes these distinctions after the symbolic descriptions below. The following descriptions are symbolic positioning statements. They locate
the characteristic dissipative mechanisms of CIP, LPS, and AFC relative to
ADSC in the present structured-grid modal framework; they should not be read
as complete analyses of all calibrated production implementations of these
methods.

\paragraph{CIP~\citep{Burman2004,Burman2007}.}
The interior-edge penalty adds contributions proportional to
$(1-\cos\theta_r)^2/h^2$ per coordinate direction, targeting
high-frequency oscillations more selectively than coordinate upwinding:
for small \(\theta_r\), the CIP factor scales like \(\theta_r^4/h^2\),
whereas the usual diffusion factor scales like \(\theta_r^2/h^2\), so the
penalty is concentrated on short-wave modes. At the modal level, CIP acts
as a high-pass filter biased toward fourth-order suppression of short-wave
modes.

\paragraph{LPS~\citep{Braack2006}.}
Local projection stabilization projects the streamline derivative onto
a coarser space, introducing dissipation only for the component not
representable on the coarse projection space. Modally, LPS acts as a
band-pass filter above the coarse-space resolution threshold.

\paragraph{AFC~\citep{Kuzmin2005}.}
Algebraic flux correction applies a nonlinear flux-limiting procedure
to enforce local monotonicity. Its modal footprint is solution-dependent
and cannot be described by a fixed symbol.

Recent developments confirm that AFC and LPS remain active research
directions beyond the schematic structured-grid representatives considered
here. Examples include high-order AFC schemes based on Bernstein finite
elements with skew-symmetrized discrete gradient operators
\citep{Hajduk2025}, element-based convex limiting with WENO-type smoothness
sensors for continuous Galerkin discretizations
\citep{KuzminHajdukVedral2025}, and LPS extensions to
\(\boldsymbol H(\mathrm{curl})\) and \(\boldsymbol H(\mathrm{div})\)
advection problems \citep{LuoWangWu2025}. These references are cited to
position the present work with respect to recent activity in AFC and LPS;
the numerical representatives used below are not intended to reflect the
state of the art of these families.

These positions motivate ADSC as a direction-aware local surrogate that
selectively targets convectively dominant modes identified by $\rpq$,
without attempting to replace calibrated flux-limiting or projection-based
stabilizations.

\section{ADSC as a local surrogate for directional modal rectification}
\label{sec:adsc}

Proposition~\ref{prop:nonlocality} and the Frobenius diagnostic discussion
show that exact modal rectification is intrinsically nonlocal in the nodal
basis, while any fixed local sparse pattern necessarily ignores part of the
rectified reference operator. A practical correction must therefore be a
local surrogate rather than an exact realization of \(\Lstar\). The
construction of ADSC is based on four structural requirements inherited
from this obstruction and from the modal interpretation: positivity,
sparsity, directionality, and adaptive activation.

\paragraph{ADSC design class.}
We do not claim that the following requirements characterize all possible
stabilizations. Other local graph-diffusion choices are possible; this paper
analyzes one member of that broader class. The requirements below define the
particular local surrogate class studied here. Within the class of Cartesian nearest-neighbor
coordinate-edge corrections, we impose:
\begin{enumerate}[label=(P\arabic*)]
\item symmetric positive semidefiniteness;
\item fixed stencil bandwidth independent of \(N\);
\item dependence on the convection direction through the coordinate
      decomposition \(\bfbeta\cdot\nabla=\beta_1\partial_x+
      \beta_2\partial_y\), with edge weights normalized at the mesh-level
      coordinate-upwind scale \(|\beta_r|h\) in each active coordinate
      direction;
\item adaptive reinforcement confined to edges adjacent to nodes where the
      computed solution exhibits directional non-monotonicity;
\end{enumerate}
These requirements lead to the edge-diffusion form
\[
 S_h = D_x^T W_x D_x + D_y^T W_y D_y,
 \qquad W_x\ge0,\quad W_y\ge0,
\]
with diagonal nonnegative weights of the form
\[
 W_r = W_r^{0}+W_r^{1}(\chi), \qquad r\in\{x,y\},
\]
where the baseline part is present everywhere and the reinforced part is
activated only through the selected directional activation field. The ADSC choice
\eqref{eq:adsc_correction}--\eqref{eq:edge_coeffs} is the particular
realization of this normalized structure; its dimensionless
amplitudes are the parameters \(\gamma_0\) and \(\gamma_1\).

The paragraph above is a construction, not a uniqueness theorem. The
conditions intentionally encode the modeling choices made for ADSC:
coordinate-edge sparsity, positivity, directional weighting, and
activation reinforcement. The scales \(|\beta_r|h\) are a normalization of
this chosen class, while the dimensionless amplitudes are set by the
modal-balance parameter class below.

\begin{table}[htbp]
\centering
\scriptsize
\caption{Structural positioning of ADSC relative to classical stabilizations. The entries are schematic and refer to the structured-grid setting used in this work; AFC-inspired edge-diffusion comparator denotes a diagnostic representative, not a calibrated production AFC implementation.}
\label{tab:method_positioning}
\resizebox{\textwidth}{!}{%
\begin{tabular}{llllll}
\toprule
Property & Galerkin & Coord.\ upwind & SUPG & AFC-inspired & ADSC \\
\midrule
Operator form & \(K_h\) & \(K_h+D_h^{up}\) & \(K_h+\tau D_\beta^TD_\beta\) & \(K_h+S_h^{AFC}(U)\) & \(K_h+S_h^{ADSC}(\chi)\) \\
PSD correction & No & Yes & Yes & nonlinear PSD surrogate & Yes \\
Fixed bandwidth & Yes & Yes & Yes & Yes in this representative & Yes \\
Directional weighting & none & coordinate axes & streamline & solution-dependent & \(\bfbeta\)-decomposed \\
Adaptive activation & none & none & none & limiter-based & directional detector \\
Nonsingularity & diffusion coercivity & diffusion plus PSD & diffusion plus PSD & limiter-dependent & fixed-activation nonsingularity by coercivity \\
Parameter principle & none & monotone upwind scale & streamline scale \(\tau\) & flux limiting & modal-balance scaling \\
\bottomrule
\end{tabular}
}
\end{table}

\subsection{Sparse directional dissipative correction}

Let $D_x$, $D_y$ denote the forward edge-difference operators with factor
$1/h$:
$(D_x U)_{i+\frac{1}{2},j}=(U_{i+1,j}-U_{ij})/h$,
$(D_y U)_{i,j+\frac{1}{2}}=(U_{i,j+1}-U_{ij})/h$.
Their transposes \(D_x^T\) and \(D_y^T\) are the adjoint backward
edge-difference operators, so that \(D_x^TW_xD_x\) and \(D_y^TW_yD_y\)
are standard one-dimensional diffusion stencils with edge weights.
The ADSC correction is
\begin{equation}\label{eq:adsc_correction}
	S^{\mathrm{ADSC}}_h=D_x^T W_x D_x+D_y^T W_y D_y,
	\quad W_x\ge 0,\; W_y\ge 0,
\end{equation}
where $W_x$, $W_y$ are diagonal matrices of nonnegative edge-diffusion
coefficients. Since $W_x,W_y\ge 0$, $S^{\mathrm{ADSC}}_h$ is symmetric
positive semidefinite: it increases the effective real part of the modal
symbol while leaving $b_h$ unchanged.

\subsection{Directional monotonicity principle and regularized detector}
\label{subsec:regularized_detector}

In physical space, convectively dominant modes appear through local changes
of monotonicity along the transport direction. Set
\(\beta_r^+=\max(\beta_r,0)\), \(\beta_r^-=\min(\beta_r,0)\), \(r=1,2\).
The backward and forward directional differences are
\begin{align*}
	D^{-}_{\bfbeta} U
	&=\beta_1^+D^-_x U+\beta_1^-D^+_x U
	+\beta_2^+D^-_y U+\beta_2^-D^+_y U,\\
	D^{+}_{\bfbeta} U
	&=\beta_1^+D^+_x U+\beta_1^-D^-_x U
	+\beta_2^+D^+_y U+\beta_2^-D^-_y U.
\end{align*}
Here \(D^-_{\bfbeta}U\) is the backward (upwind) directional difference
and \(D^+_{\bfbeta}U\) is the forward (downwind) directional difference
along \(\bfbeta\).

The sharp active-set detector used in the binary implementation is
\[
\chi^{\rm sh}_{ij}
=
\mathbf{1}_{\{(D^{-}_{\bfbeta} U)_{ij}(D^{+}_{\bfbeta} U)_{ij}<0\}}.
\]
It equals one when the discrete solution exhibits a sign change of
successive directional increments along \(\bfbeta\), that is, when a local
directional maximum or minimum is detected. The detector is not a modal
projector; it is a physical-space proxy for directional oscillatory changes
used to transfer the modal rectification principle into a local edge-based
activation rule.

For the consistency analysis, however, a discontinuous binary detector is
too sharp: it can generate grid-scale jumps in the artificial-diffusion
coefficients. We therefore also introduce a regularized directional
detector. Define the normalized directional non-monotonicity score
\begin{equation}\label{eq:theta_detector}
\Theta_{ij}(U)
=
\frac{
2\left[-(D^-_{\bfbeta} U)_{ij}(D^+_{\bfbeta} U)_{ij}\right]_+
}{
(D^-_{\bfbeta} U)_{ij}^{2}+(D^+_{\bfbeta} U)_{ij}^{2}+\delta_h
},
\qquad \delta_h>0,
\end{equation}
where \([z]_+=\max(z,0)\). The quantity \(\Theta_{ij}\) vanishes when the
two directional increments have the same sign and becomes positive when a
directional change of monotonicity occurs. The denominator normalizes the
indicator and prevents singular behavior near flat regions. In the
analysis below we assume \(\delta_h\ge\delta_0>0\) and a uniformly
Lipschitz activation, which gives a clean coefficient-variation estimate
with mesh-independent proof constants. The computations use a very small
regularization parameter and should therefore be read as numerical evidence
for the sharp regularized regime, not as a direct verification of those
uniform-Lipschitz constants. In particular, the proofs apply formally for
a fixed positive lower bound \(\delta_0\), whereas the computations use
\(\delta_h=10^{-12}\) as a sharp-regularized numerical regime. The numerical
values of \(\delta_h\) and \(\eta_{\rm det}\) are reported in
Section~\ref{sec:numerics}. Thus the numerical setting approximates the
sharp-detector limit and lies outside the mesh-uniform Lipschitz constants
assumed in the proof.

Let \(s_h:[0,\infty)\to[0,1]\) be a Lipschitz activation function such
that \(s_h(0)=0\) and \(s_h(t)\) increases toward one. We set
\begin{equation}\label{eq:regularized_chi}
\chi^h_{ij}=s_h(\Theta_{ij}(U)).
\end{equation}
A simple admissible choice is
\[
s_h(t)=\frac{t}{t+\eta_{\rm det}},\qquad \eta_{\rm det}>0.
\]
The sharp detector is thus the limiting selective active-set realization
of the same directional monotonicity principle, whereas the regularized
detector provides the analytically controlled version used below to obtain
a discrete variation estimate. This distinction prevents ADSC from being
an ad hoc artificial viscosity: the detector is built from a directional
monotonicity criterion, normalized by local directional gradients, and
regularized only to make the induced edge coefficients compatible with a
consistency argument.
\subsection{Baseline parameter and modal balance}

\paragraph{Activation and baseline law.}
The activation function $\eta_{\Pe}(\Peh)=\max(0,1-1/\Peh)$ vanishes below
$\Peh=1$ and increases monotonically toward 1. The baseline coefficient is
\begin{equation}\label{eq:gamma0}
	\gamma_0(\Peh)=\begin{cases}
		0, & \Peh\le 1, \\
		\gamma_{\min}+(\gamma_{\max}-\gamma_{\min})\dfrac{\Peh-1}{\Peh+1},
		   & \Peh>1,
	\end{cases}
	\qquad
	\gamma_1(\Peh)=\kappa\,\gamma_0(\Peh),\quad\kappa\ge 1.
\end{equation}
The saturating form $(\Peh-1)/(\Peh+1)$ is smooth for \(\Peh>1\),
bounded in \([0,1)\), monotone, and vanishes continuously at
\(\Peh=1\); the analysis below uses only these qualitative properties of
the saturating function.
In the numerical section the constants are selected reproducibly as
follows. First, the Galerkin modal indicator at the target mesh P\'eclet
number is used in the modal-balance estimate below to determine the
order of magnitude of \(\gamma_0\). This indicator is computed from the
interior stencil symbol and the prescribed coefficients, not from the
computed solution, the fine-grid reference solution, or the error tables.
Second, \(\gamma_{\min}\) and \(\gamma_{\max}\) are chosen as conservative
fractions of the raw modal-balance range observed over the tested mesh
family. In Table~\ref{tab:gamma_balance_validation}, that range is
\([0.378,0.456]\); the implemented interval \([0.08,0.25]\) places the
saturating law roughly between one fifth and two thirds of those raw values.
This deliberately keeps the implemented law below full modal balance, so that
the baseline correction does not become a nearly uniform upwind replacement.
The factor \(\kappa\) controls the strength of activation reinforcement. This
choice is heuristic and not error-optimized; systematic parameter
optimization is left open. The sensitivity table in Section~\ref{sec:numerics}
then reports the effect of varying these dimensionless constants. The same
triplet \((\gamma_{\min},\gamma_{\max},\kappa)=(0.08,0.25,2)\) is kept for all
reported tests, including the manufactured and NIST-type benchmarks.

\paragraph{Modal-balance scaling for the baseline parameter.}
The following modal-balance argument provides a scaling rule for the
baseline parameter. When the detector is inactive ($\chi_{ij}\equiv 0$),
the baseline correction has modal increment
\[
	\delta^0_{pq}
	= \frac{2\gamma_0}{h}
	\bigl[|\beta_1|(1-\cos\theta_p)+|\beta_2|(1-\cos\theta_q)\bigr].
\]
Writing $r_p=1-\cos\theta_p$, $r_q=1-\cos\theta_q$,
and using $\apq=2\eps(r_p+r_q)/h^2$, one obtains
\[
	\frac{\delta^0_{pq}}{\apq}=2\gamma_0\Peh\,B_{pq},
	\qquad
	B_{pq}=\frac{|\beta_1|r_p+|\beta_2|r_q}{|\bfbeta|(r_p+r_q)}\in[0,1].
\]
The Rayleigh-quotient corrected modal indicator is therefore
\[
	\tilde{\rho}_{pq}
	=
	\frac{\rho_{pq}}{1+2\gamma_0\Peh B_{pq}}
	+O(\eta_{pq}),
\]
where \(\eta_{pq}\) denotes the relative contribution of the off-diagonal
modal couplings neglected by the one-mode Rayleigh approximation. For the
translation-invariant periodic stencil, \(\eta_{pq}=0\) and the displayed
ratio is the exact symbol-level indicator. For the finite Dirichlet matrix,
the formula is used only as a first-order modal scaling rule for the ADSC
parameters; no uniform error bound is claimed for individual modes.
Requiring $\tilde{\rho}_{pq}\le\bar{\rho}_{\mathrm{target}}$ for a representative
most dominant mode, and averaging over $\Icd$, motivates the modal-balance
scaling
\begin{equation}\label{eq:gamma0_balance}
	\gamma_0\sim\frac{\bar\rho_{\mathrm{Gal}}-1}{2\Peh\,\bar B},
	\qquad
	\bar B=\frac{1}{|\Icd|}\sum_{(p,q)\in\Icd}B_{pq},
\end{equation}
showing that the modal-balance target belongs to the class
$\gamma_0^{\mathrm{bal}}=O(\Peh^{-1})$ when the mean modal indicator is
significantly above the threshold. A reproducible bounded choice used in
practice is
\begin{equation}\label{eq:gamma_projected_choice}
\gamma_0^{\mathrm{bal}}
=
\frac{\max\{\bar\rho_{\mathrm{Gal}}-\rho_{\mathrm{target}},0\}}
{2\Peh\,\bar B},
\qquad
\gamma_0
=
\min\{\gamma_{\max},\max\{\gamma_{\min},\gamma_0^{\mathrm{bal}}\}\},
\end{equation}
with \(\rho_{\mathrm{target}}\ge1\). The lower and upper bounds
\(\gamma_{\min}\) and \(\gamma_{\max}\) are not fitted to the numerical
error; they prevent the modal-balance estimate from producing either a
vanishing correction or an excessively large artificial diffusion.

\begin{proposition}[Conditional modal-balance scaling]
\label{lem:modal_balance_class}
Assume that, on the considered mesh family, the averaged modal-balance factor
\(\bar B\) is bounded away from zero on the Galerkin-dominant modal set and
that \(\bar\rho_{\mathrm{Gal}}\) remains bounded above. Then the balance
relation~\eqref{eq:gamma0_balance} selects a baseline parameter of order
\[
 \gamma_0^{\mathrm{bal}}=O(\Peh^{-1}).
\]
Consequently, the baseline artificial diffusion scale
\(|\bfbeta|h\gamma_0^{\mathrm{bal}}\) is of order
\(|\bfbeta|h/\Peh=2\eps\), and the added baseline correction remains a
bounded local diffusion surrogate rather than a full uniform upwind
replacement.
\end{proposition}

\begin{proof}
The estimate follows directly from~\eqref{eq:gamma0_balance}. The projected
choice~\eqref{eq:gamma_projected_choice} has the same order whenever the
projection bounds are inactive; when the bounds are active, it remains bounded
by construction. If $\bar B\ge B_0>0$ and $\bar\rho_{\mathrm{Gal}}=O(1)$, then
\[
 \gamma_0^{\mathrm{bal}}
 =\frac{\bar\rho_{\mathrm{Gal}}-1}{2\Peh\bar B}
 =O(\Peh^{-1}).
\]
Multiplying by the mesh-level scale $|\bfbeta|h$ gives
$|\bfbeta|h\gamma_0^{\mathrm{bal}}=O(|\bfbeta|h/\Peh)=O(\eps)$, because
$\Peh=|\bfbeta|h/(2\eps)$. This proves the stated scaling of the baseline
edge diffusion. The reinforced term is then supplied selectively through
$\gamma_1\eta_{\Pe}(\Peh)\chi$ on detected directional non-monotonicity
regions.
\end{proof}
These hypotheses are checked a posteriori in Table~\ref{tab:mesh_diag} for
the tested configurations; they are not consequences of ADSC and are not used
to fit the errors. Their general validity is an open question, and the
modal-balance scaling should not be invoked on mesh families for which
\(\bar B\) degenerates.

Proposition~\ref{lem:modal_balance_class} identifies the parameter class
suggested by the modal-balance argument under explicit boundedness
assumptions on \(\bar B\) and \(\bar\rho_{\mathrm{Gal}}\). These assumptions
are observed in the tested configurations
(Table~\ref{tab:mesh_diag}); their general validity requires further analysis.
They should therefore be read as a conditional numerical hypothesis for
the present modal-design rule, not as an unconditional theorem about all
convection--diffusion data.
They are expected to hold when
the Galerkin-dominant modal set is not concentrated on frequencies for
which the coordinate-weighted numerator in \(B_{pq}\) is nearly zero. The
projected rule~\eqref{eq:gamma_projected_choice} gives a reproducible
implementation of this principle, and the saturating law~\eqref{eq:gamma0}
is a smooth bounded parametrization over the tested P\'eclet range; the constants
$(\gamma_{\min},\gamma_{\max},\kappa)$ are not claimed to be optimal. Their
role is analogous to practical parameter choices in SUPG, CIP, or AFC
implementations: they instantiate a structural scaling rule while the final
performance is assessed numerically.

\paragraph{Edge diffusion coefficients.}
The nodal activation is transferred to edges. In the sharp active-set
implementation one may use the conservative maximum transfer
\[
	\chi^x_{\ell j}=\max(\chi_{\ell j},\chi_{\ell+1,j}),
	\qquad
	\chi^y_{i\ell}=\max(\chi_{i\ell},\chi_{i,\ell+1}).
\]
For the regularized detector, the smoother averaged transfer
\begin{equation}\label{eq:edge_average_detector}
	\chi^x_{i+\frac12,j}
	=\frac12(\chi^h_{ij}+\chi^h_{i+1,j}),
	\qquad
	\chi^y_{i,j+\frac12}
	=\frac12(\chi^h_{ij}+\chi^h_{i,j+1})
\end{equation}
is used in the analysis. A maximum transfer would generally preserve the
sharp active-set character but would not give the same immediate discrete
variation estimate. The associated edge diffusion coefficients are
\begin{equation}\label{eq:edge_coeffs}
	\begin{aligned}
	\alpha^x_{i+\frac12,j}=|\beta_1|h
	\bigl[\gamma_0(\Peh)
	+\gamma_1(\Peh)\eta_{\Pe}(\Peh)\chi^x_{i+\frac12,j}\bigr],\\
	\alpha^y_{i,j+\frac12}=|\beta_2|h
	\bigl[\gamma_0(\Peh)
	+\gamma_1(\Peh)\eta_{\Pe}(\Peh)\chi^y_{i,j+\frac12}\bigr].
	\end{aligned}
\end{equation}
The weights $|\beta_1|$, $|\beta_2|$ distribute the reinforcement along
the Cartesian directions proportionally to the components of $\bfbeta$;
if \(\beta_r=0\), no ADSC edge diffusion is added in that coordinate
direction. This reproduces the anisotropic structure of coordinate
upwinding while confining the reinforcement to detected oscillatory regions. The maximum
transfer gives the sharp binary limiting variant, whereas the computations
and the coefficient-variation estimate below use the averaged regularized
transfer.

\subsection{Monotone activation update and final fixed solve}

In the computations reported below, ADSC is implemented with a monotone
regularized activation field. The Galerkin solution is used as a warm start;
other initializations are reported in Section~\ref{sec:numerics}. Starting
from $U^{(0)}=U^{\mathrm{Gal}}$ and
$\chi^{(0)}\equiv0$, the detector is updated by
\[
	\chi^{(m+1)}
	=
	\chi^{(m)}
	\vee
	\chi(U^{(m)}),
\]
where $\vee$ denotes the componentwise maximum. Thus, once a node has
received a given activation level, that level is not decreased. At step
$m$, one assembles
$W^{(m)}_x$, $W^{(m)}_y$, solves
\[
	\bigl(\Kh+D_x^T W^{(m)}_x D_x+D_y^T W^{(m)}_y D_y\bigr)
	\tilde{U}^{(m+1)}=\mathbf f,
\]
and sets $U^{(m+1)}=(1-\omega)U^{(m)}+\omega\tilde{U}^{(m+1)}$.
The iteration is stopped when the relative change of the activation field is
below the prescribed tolerance; a final linear solve is then performed with
the fixed matrix corresponding to the final activation field.

\begin{proposition}[Nonsingularity and relaxed final solve for a fixed activation]
	\label{prop:adsc_convergence}
	For every fixed activation field \(0\le\chi^\ast\le1\), the associated
	ADSC stabilized matrix is nonsingular. Once the activation field is fixed, the post-processing relaxation toward
	the corresponding fixed stabilized solution is linearly contractive for
	every \(0<\omega\le1\).
\end{proposition}

\begin{proof}
	For a fixed activation field, set
	\[
		A^\ast
		=
		\Kh+D_x^T W_x(\chi^\ast)D_x
		+D_y^T W_y(\chi^\ast)D_y.
	\]
	The centered convection matrix is skew-symmetric, while the Dirichlet
	diffusion matrix is symmetric positive definite. Moreover,
	$D_x^T W_x(\chi^\ast)D_x+D_y^T W_y(\chi^\ast)D_y$ is symmetric positive
	semidefinite because the weights are nonnegative. Hence, for any nonzero
	real vector $U$,
	\[
		U^T A^\ast U
		=
		U^T D_h U
		+U^T\bigl(D_x^T W_x(\chi^\ast)D_x+D_y^T W_y(\chi^\ast)D_y\bigr)U
		>0,
	\]
	where $D_h$ denotes the symmetric diffusive part of $\Kh$. Although
	$A^\ast$ is not symmetric because of the centered convective part, the
	positivity of its quadratic form is sufficient for nonsingularity.
	Therefore $A^\ast$ is nonsingular.

	Once the activation field is fixed, let
	$U^\ast=(A^\ast)^{-1}\mathbf f$. The fixed-activation solve gives
	$\tilde U^{(m+1)}=U^\ast$, hence
	\[
		U^{(m+1)}-U^\ast
		=
		(1-\omega)(U^{(m)}-U^\ast).
	\]
	Thus $U^{(m)}\to U^\ast$ linearly for $0<\omega\le1$.
\end{proof}

The proposition concerns the fixed sparse system selected by the
	monotone regularized update in the numerical experiments. For the
	regularized detector, the activation values are real numbers in
	\([0,1]\), so finite stationarity in at most \(N^2\) updates is no longer
	a mathematical consequence of binary monotonicity. The computations
	therefore use a relative activation-change tolerance, taken as
	\(10^{-8}\) in the reported tests, and then report the final
	fixed-activation solve. Thus Proposition~\ref{prop:adsc_convergence}
	applies to the exactly frozen matrix used in the final solve, while the
	nonzero final activation variation reported in the tables is only the
	stopping indicator used before this freeze. This is a fixed-activation convergence
	statement. The next subsection gives a separate consistency--stability
	estimate for the regularized fixed-activation ADSC discretization; no
	optimality of the selected activation field or closeness to the dense
	ideal rectified operator is claimed.

\subsection{Regularized fixed-activation consistency and convergence}

The previous proposition gives the well-posed fixed sparse system used by
the monotone activation algorithm for a fixed mesh. We now record the
consistency--stability theory for the regularized fixed-activation ADSC
operator. Throughout this subsection, the activation field is frozen after
being generated by an auxiliary grid sequence \(V_h\) satisfying the discrete
smoothness assumptions below. Thus the result concerns the selected
fixed-activation operator, not the fully coupled nonlinear iteration in which
the activation is generated by the unknown stabilized solution itself. The
coupled iteration is assessed numerically in Section~\ref{sec:numerics},
including the fixed-reference experiment of Section~\ref{subsec:fixed_reference_activation}.
The role of the regularized detector is precisely to provide a verifiable
mechanism for the discrete variation bound that was only assumed for an
arbitrary binary active set.

We use the standard discrete norms associated with the uniform Cartesian
mesh. For an interior grid vector \(V=(V_{ij})\),
\[
 \|V\|_{0,h}^{2}=h^2\sum_{i,j}|V_{ij}|^2,
\]
is the discrete \(L^2\)-norm, while
\[
 |V|_{1,h}^{2}
 =\|D_xV\|_{0,h}^{2}+\|D_yV\|_{0,h}^{2}
\]
is the discrete \(H^1\)-seminorm induced by the nearest-neighbor edge
differences used in the construction of the stiffness and stabilization
matrices. Thus the subscript \(0,h\) refers to the mesh-dependent
\(L^2\)-level norm, and \(1,h\) refers to the mesh-dependent first-order
seminorm. With homogeneous Dirichlet boundary values, \(|\cdot|_{1,h}\) is
a norm on the interior grid space, up to the usual discrete Poincar\'e
inequality.
The constants below are independent of \(h\), but may depend on
\(\eps\), \(\bfbeta\), the regularization parameters, and the continuous
solution.

\begin{lemma}[Discrete variation of the regularized detector]
\label{lem:regularized_detector_variation}
Assume that the input grid sequence \(V_h\) satisfies the discrete
smoothness bounds
\[
\|D^\pm_{\bfbeta}V_h\|_{\ell^\infty}\le C_V,
\qquad
| (D^\pm_{\bfbeta}V_h)_{i+1,j}-(D^\pm_{\bfbeta}V_h)_{ij}|
+
| (D^\pm_{\bfbeta}V_h)_{i,j+1}-(D^\pm_{\bfbeta}V_h)_{ij}|
\le C_V h.
\]
Assume also that \(\delta_h\ge \delta_0>0\) and that \(s_h\) is Lipschitz
with a constant bounded independently of \(h\). The constant below depends
on this uniform Lipschitz bound. These are assumptions on the auxiliary
sequence \(V_h\); the lemma does not assert that they hold when \(V_h\) is
the fully coupled ADSC solution. Then the regularized
detector defined by \eqref{eq:theta_detector}--\eqref{eq:regularized_chi}
satisfies
\[
|\chi^h_{i+1,j}-\chi^h_{ij}|
+
|\chi^h_{i,j+1}-\chi^h_{ij}|
\le C h.
\]
Consequently, the averaged edge activations \eqref{eq:edge_average_detector}
satisfy
\[
|\chi^x_{i+\frac12,j}-\chi^x_{i-\frac12,j}|
+
|\chi^y_{i,j+\frac12}-\chi^y_{i,j-\frac12}|
\le C h,
\]
and the corresponding ADSC edge coefficients satisfy
\[
|\alpha^x_{i+\frac12,j}-\alpha^x_{i-\frac12,j}|
\le C|\beta_1|h^2,\qquad
|\alpha^y_{i,j+\frac12}-\alpha^y_{i,j-\frac12}|
\le C|\beta_2|h^2.
\]
\end{lemma}

\begin{proof}
Let \(a_{ij}=(D^-_{\bfbeta}V_h)_{ij}\) and
\(b_{ij}=(D^+_{\bfbeta}V_h)_{ij}\). By assumption, \(a_{ij}\) and
\(b_{ij}\) are uniformly bounded and vary by \(O(h)\) between neighboring
nodes. The map
\[
F(a,b)=
\frac{2[-ab]_+}{a^2+b^2+\delta_h}
\]
is Lipschitz on bounded subsets of \(\mathbb R^2\) uniformly in \(h\),
because \(\delta_h\ge\delta_0>0\). Hence
\[
|\Theta_{i+1,j}(V_h)-\Theta_{ij}(V_h)|
+
|\Theta_{i,j+1}(V_h)-\Theta_{ij}(V_h)|
\le C h.
\]
Composing with the uniformly Lipschitz function \(s_h\) gives the same
bound for \(\chi^h\). The edge activations are local averages of adjacent
nodal activations, so their first differences are also \(O(h)\). Finally,
\[
\alpha^x=|\beta_1|h\bigl(\gamma_0+\gamma_1\eta_{\Pe}\chi^x\bigr),
\qquad
\alpha^y=|\beta_2|h\bigl(\gamma_0+\gamma_1\eta_{\Pe}\chi^y\bigr),
\]
with bounded dimensionless factors. Multiplication by the mesh factor
\(h\) converts the \(O(h)\) variation of \(\chi^x,\chi^y\) into the
\(O(h^2)\) variation of the edge diffusion coefficients.
\end{proof}

\paragraph{A directly verifiable smoothness case.}
If \(V_h\) is obtained by sampling a function \(v\in C^2(\overline\Omega)\),
then the discrete smoothness assumptions in
Lemma~\ref{lem:regularized_detector_variation} follow from Taylor
expansion. Thus the lemma applies directly to activation fields built from
smooth reference grid functions.

\paragraph{Warning for the fully coupled activation.}
If \(V_h\) is taken to be the computed ADSC solution itself, an additional
discrete regularity estimate would be needed. This fully coupled nonlinear
question is not used in the fixed-operator convergence theorem below. In the
numerical tests, the computed detector counts and the total-variation values
reported in Section~\ref{sec:numerics} provide diagnostic evidence that the
selected activations remain localized and do not spread globally on the tested
meshes, but they are not used as a proof of the present smoothness hypothesis.

\begin{proposition}[Consistency of the fixed regularized ADSC operator]
\label{prop:adsc_consistency_fixed}
Assume that \(u\in C^4(\overline\Omega)\), and let \(u_h\) denote its
grid sampling. Suppose that the regularized ADSC edge diffusion
coefficients satisfy
\[
0\le \alpha^x_{i+\frac12,j}\le C_\alpha|\beta_1|h,
\qquad
0\le \alpha^y_{i,j+\frac12}\le C_\alpha|\beta_2|h,
\]
and the variation bounds supplied, for instance, by
Lemma~\ref{lem:regularized_detector_variation}:
\[
|\alpha^x_{i+\frac12,j}-\alpha^x_{i-\frac12,j}|
\le C_\alpha|\beta_1|h^2,
\qquad
|\alpha^y_{i,j+\frac12}-\alpha^y_{i,j-\frac12}|
\le C_\alpha|\beta_2|h^2.
\]
Let \(f_h=(-\eps\Delta u+\bfbeta\cdot\nabla u)|_{\rm grid}\). Then
\[
\|\Kh u_h+S_h^{\rm ADSC}u_h-f_h\|_{0,h}
\le
C\Bigl[
h^2\bigl(\eps\|u\|_{C^4}+|\bfbeta|\|u\|_{C^3}\bigr)
+h|\bfbeta|\|u\|_{C^2}
\Bigr].
\]
Consequently, in the active convection-dominated regime the regularized
fixed-activation ADSC discretization is first-order consistent. If
\(\Peh\le 1\) and the activation law gives \(\gamma_0=\gamma_1=0\), then
\(S_h^{\rm ADSC}=0\) and the centered second-order truncation error is
recovered.
\end{proposition}

\begin{proof}
The centered diffusion and convection differences are second-order
consistent for smooth functions. Hence
\[
\|\Kh u_h-f_h\|_{0,h}
\le
Ch^2\bigl(\eps\|u\|_{C^4}+|\bfbeta|\|u\|_{C^3}\bigr).
\]
It remains to estimate the ADSC term. We give the details in the
\(x\)-direction; the \(y\)-direction is identical. Set
\[
g^x_{i+\frac12,j}
=
(D_xu_h)_{i+\frac12,j}
=
\frac{u(x_{i+1},y_j)-u(x_i,y_j)}{h}.
\]
For \(u\in C^2(\overline\Omega)\), Taylor expansion gives
\[
|g^x_{i+\frac12,j}|\le C\|u\|_{C^1},
\qquad
|g^x_{i+\frac12,j}-g^x_{i-\frac12,j}|
\le Ch\|u\|_{C^2}.
\]
At an interior node, up to the harmless sign determined by the convention
for \(D_x^T\), the \(x\)-part of the ADSC correction has the conservative
flux-difference form
\[
(D_x^TW_xD_xu_h)_{ij}
=
\frac1h
\left(
\alpha^x_{i-\frac12,j}g^x_{i-\frac12,j}
-
\alpha^x_{i+\frac12,j}g^x_{i+\frac12,j}
\right).
\]
Therefore
\[
\begin{aligned}
|(D_x^TW_xD_xu_h)_{ij}|
&\le
\frac1h\alpha^x_{i+\frac12,j}
|g^x_{i+\frac12,j}-g^x_{i-\frac12,j}|\\
&\quad+
\frac1h
|\alpha^x_{i+\frac12,j}-\alpha^x_{i-\frac12,j}|
|g^x_{i-\frac12,j}|\\
&\le C|\beta_1|h\|u\|_{C^2}.
\end{aligned}
\]
The same argument in the \(y\)-direction gives
\[
|(D_y^TW_yD_yu_h)_{ij}|
\le C|\beta_2|h\|u\|_{C^2}.
\]
Thus, pointwise on the grid and hence in the discrete \(L^2\) norm,
\[
\|S_h^{\rm ADSC}u_h\|_{0,h}
\le Ch|\bfbeta|\|u\|_{C^2}.
\]
Combining this estimate with the centered truncation estimate proves the
claim.
\end{proof}

\begin{proposition}[Fixed-activation energy stability]
\label{prop:adsc_energy_stability}
For any fixed activation field and any nonnegative ADSC weights, the
stabilized bilinear form associated with
\[
A_h^{\rm ADSC}=\Kh+S_h^{\rm ADSC}
\]
satisfies
\[
(A_h^{\rm ADSC}U,U)_{0,h}
\ge \eps |U|_{1,h}^2
\qquad \forall U\in\mathbb R^{N^2}.
\]
In particular, for fixed \(\eps>0\), the fixed-activation ADSC system is
coercive in the discrete \(H^1\) seminorm.
\end{proposition}

\begin{proof}
The centered convection part is skew-symmetric under homogeneous
Dirichlet conditions, so its quadratic contribution vanishes. With the discrete norm \(\|U\|_{0,h}^2=h^2\sum_{i,j}|U_{ij}|^2\) and the
edge-difference seminorm used in this paper, the diffusive part gives
\(\eps |U|_{1,h}^{2}\). Moreover,
\[
(S_h^{\rm ADSC}U,U)_{0,h}
=
\|W_x^{1/2}D_xU\|_{0,h}^{2}
+
\|W_y^{1/2}D_yU\|_{0,h}^{2}
\ge0,
\]
because the edge weights are nonnegative. The stated estimate follows.
\end{proof}

\paragraph{No \(\eps\)-uniform stability claim.}
The coercivity constant in Proposition~\ref{prop:adsc_energy_stability}
is proportional to \(\eps\). Thus the estimate is suitable for fixed
positive diffusion but is not a robust \(\eps\)-uniform stability bound in
the singularly perturbed limit. Establishing an \(\eps\)-uniform norm,
for instance by adding a streamline or graph-diffusion contribution to the
energy norm, would require a separate analysis and is left open.

\begin{theorem}[Conditional discrete \(H^1\)-seminorm convergence of the fixed-activation regularized ADSC operator]
\label{thm:adsc_pde_convergence}
Let \(u\in C^4(\overline\Omega)\) be the solution of~\eqref{eq:strong}
and let \(U_h^{\rm ADSC}\) be the solution of a fixed regularized ADSC
linear system whose activation field has been generated from an auxiliary
grid sequence \(V_h\) satisfying Lemma~\ref{lem:regularized_detector_variation}.
Here \(V_h\) is not assumed to be the unknown computed solution; the fully
coupled case \(V_h=U_h^{\rm ADSC}\) would require an additional discrete
regularity argument and is not part of this theorem. Assume that the
resulting ADSC coefficients satisfy the size and discrete variation estimates
of Proposition~\ref{prop:adsc_consistency_fixed} uniformly in \(h\). Then
\[
 |U_h^{\rm ADSC}-u_h|_{1,h}
 \le
 C_\eps
 \Bigl[
 h^2\bigl(\eps\|u\|_{C^4}+|\bfbeta|\|u\|_{C^3}\bigr)
 +h|\bfbeta|\|u\|_{C^2}
 \Bigr].
\]
Thus \( |U_h^{\rm ADSC}-u_h|_{1,h}=O(h)\) in the active regularized
regime. If \(\Peh\le1\) and ADSC is inactive, the estimate reduces to the
corresponding centered second-order consistency regime.
\end{theorem}

\begin{proof}
Set \(e_h=U_h^{\rm ADSC}-u_h\) and
\[
\tau_h=\Kh u_h+S_h^{\rm ADSC}u_h-f_h.
\]
Subtracting the fixed-activation ADSC system from the equation obtained
by applying the same stabilized operator to \(u_h\) gives
\[
A_h^{\rm ADSC}e_h=-\tau_h.
\]
Testing with \(e_h\) and using
Proposition~\ref{prop:adsc_energy_stability} yields
\[
\eps |e_h|_{1,h}^{2}\le |(\tau_h,e_h)_{0,h}|.
\]
By the discrete Poincar\'e inequality,
\[
\|e_h\|_{0,h}\le C_P |e_h|_{1,h},
\]
hence
\[
\eps |e_h|_{1,h}^{2}
\le C_P\|\tau_h\|_{0,h}|e_h|_{1,h}.
\]
After cancellation,
Proposition~\ref{prop:adsc_consistency_fixed} gives the announced bound.
\end{proof}

The estimates above show why the regularized detector is analytically
useful: it gives a concrete route to the coefficient variation bound
needed in the consistency and convergence argument. The proven order is
first order in the discrete \(H^1\) seminorm in the active regularized
regime, due to the leading artificial-diffusion contribution
\(O(h|\bfbeta|)\). When the activation vanishes, the artificial-diffusion contribution vanishes
and the underlying centered second-order truncation regime is recovered at the
consistency level. In the numerical section
below, unless explicitly stated otherwise, ADSC denotes the regularized
activation version with averaged edge transfer. The sharp binary detector
is retained only as the limiting selective active-set realization of the
same directional monotonicity principle. No discrete Aubin--Nitsche
argument is used here; the reported \(L^2\) rates in Section~\ref{sec:numerics}
are therefore numerical diagnostics rather than consequences of
Theorem~\ref{thm:adsc_pde_convergence}.

\begin{proposition}[Existence of a regularized fully coupled ADSC solution]
\label{prop:adsc_coupled_existence}
Fix a mesh size \(h>0\). Assume \(\eps>0\), homogeneous Dirichlet boundary
conditions, nonnegative ADSC edge weights of the form~\eqref{eq:edge_coeffs},
and the regularized detector~\eqref{eq:theta_detector}--\eqref{eq:regularized_chi}
with \(\delta_h>0\) and a continuous activation function
\(s_h:[0,\infty)\to[0,1]\). Then the fully coupled regularized ADSC problem
\begin{equation}\label{eq:fully_coupled_adsc_problem}
\bigl(\Kh+S_h^{\rm ADSC}(\chi(U_h))\bigr)U_h=f_h
\end{equation}
admits at least one solution \(U_h\in V_h\).
\end{proposition}

\begin{proof}
The proof is finite-dimensional. For each \(V\in V_h\), freeze the
activation field generated by the regularized detector and define \(T_h(V)\)
as the unique solution of
\[
\bigl(\Kh+S_h^{\rm ADSC}(\chi(V))\bigr)T_h(V)=f_h.
\]
This fixed-activation system is nonsingular by
Proposition~\ref{prop:adsc_energy_stability}: the centered convective part
is skew-symmetric, the diffusion part is coercive for \(\eps>0\), and the
ADSC correction is positive semidefinite because all edge weights are
nonnegative. Therefore \(T_h\) is well defined.

Moreover, testing the fixed-activation equation with \(T_h(V)\) gives,
uniformly with respect to \(V\),
\[
\eps |T_h(V)|_{1,h}^2
\le (f_h,T_h(V))_{0,h}
\le C_P\|f_h\|_{0,h}|T_h(V)|_{1,h}.
\]
Thus \(T_h\) maps the closed convex ball
\[
B_h=\{W\in V_h:\ |W|_{1,h}\le C_P\|f_h\|_{0,h}/\eps\}
\]
into itself. Since \(V_h\) is finite-dimensional, \(B_h\) is compact and
convex.

It remains to note that \(T_h\) is continuous on \(B_h\). Since
\(\delta_h>0\), the map \(V\mapsto\Theta(V)\) is continuous in the finite
dimensional space \(V_h\). The activation function \(s_h\) and the averaged
edge transfer are continuous, hence \(V\mapsto\chi(V)\) and the associated
matrix \(\Kh+S_h^{\rm ADSC}(\chi(V))\) are continuous. Uniform coercivity
prevents singular loss along this map, so the solution operator
\(V\mapsto T_h(V)\) is continuous. Brouwer's fixed point theorem applied to
\(T_h:B_h\to B_h\) gives a point \(U_h=T_h(U_h)\), which is a solution of
\eqref{eq:fully_coupled_adsc_problem}.
\end{proof}

\paragraph{Uniqueness of coupled solutions for fixed \(h\).}
Proposition~\ref{prop:adsc_coupled_existence} gives existence, but uniqueness
of a fully coupled solution is not established. A sufficient contraction-type
condition would be that the Lipschitz constant of the map
\(V\mapsto S_h^{\rm ADSC}(\chi(V))\) in the discrete \(\|\cdot\|_{0,h}\) norm,
combined with the induced solution-map bound, be smaller than the coercivity
constant supplied by the diffusive part. The numerical experiments do not
indicate multiplicity, but no analytical bound of this type is derived here.

\begin{proposition}[Qualitative convergence of regularized fully coupled ADSC solutions]
\label{prop:adsc_coupled_qualitative_convergence}
Assume that \(\eps>0\), \(\bfbeta\) is constant, and \(f_h\) is a consistent
grid representation of a fixed \(f\in L^2(\Omega)\): \(f_h\) is uniformly
bounded in \(\ell_h^2\) and
\[
(f_h,\varphi_h)_{0,h}\to (f,\varphi)_{L^2(\Omega)}
\qquad\text{for every }\varphi\in C_c^\infty(\Omega),
\]
where \(\varphi_h\) is the grid sampling of \(\varphi\). Assume also that the
regularized ADSC weights are those defined in~\eqref{eq:edge_coeffs}, with
\(0\le \chi^x,\chi^y\le 1\) and a bounded P\'eclet transfer factor. In
particular, there exists a constant
\(C_\alpha\), independent of \(h\) and of the grid function \(U\), such that
\[
0\le \alpha^x_{i+\frac12,j}(U)\le C_\alpha h,
\qquad
0\le \alpha^y_{i,j+\frac12}(U)\le C_\alpha h.
\]
For each \(h\), let \(U_h^c\) be one of the fully coupled regularized ADSC
solutions whose existence is given by
Proposition~\ref{prop:adsc_coupled_existence}, namely
\[
\bigl(\Kh+S_h^{\rm ADSC}(\chi(U_h^c))\bigr)U_h^c=f_h.
\]
Let \(\mathcal I_hU_h^c\) denote the standard continuous piecewise affine
reconstruction on the Cartesian triangulation, obtained by extending the
nodal values on each rectangular cell after splitting it into two triangles
along the diagonal from \((x_i,y_j)\) to \((x_{i+1},y_{j+1})\).
Then \(\{\mathcal I_hU_h^c\}_h\) is bounded uniformly in \(H_0^1(\Omega)\).
Consequently, every sequence \(h_n\to0\) contains a subsequence, not
relabeled, such that
\[
\mathcal I_{h_n}U_{h_n}^c\rightharpoonup u
\quad\hbox{weakly in }H_0^1(\Omega),
\qquad
\mathcal I_{h_n}U_{h_n}^c\to u
\quad\hbox{strongly in }L^2(\Omega).
\]
Moreover, the limit \(u\) is the unique weak solution of the continuous
convection--diffusion problem~\eqref{eq:strong}. Consequently, any sequence
of selected coupled solutions converges to \(u\) in \(L^2(\Omega)\), up to
the chosen reconstruction.
\end{proposition}

\begin{proof}
Testing the coupled equation with \(U_h^c\) gives
\[
\bigl(\Kh U_h^c,U_h^c\bigr)_{0,h}
+
\bigl(S_h^{\rm ADSC}(\chi(U_h^c))U_h^c,U_h^c\bigr)_{0,h}
=
(f_h,U_h^c)_{0,h}.
\]
The centered convection part of \(\Kh\) is skew-symmetric under the
homogeneous Dirichlet boundary condition, the diffusion part gives
\(\eps |U_h^c|_{1,h}^2\), and the ADSC contribution is nonnegative because
the edge weights are nonnegative. Hence
\[
\eps |U_h^c|_{1,h}^2
\le
(f_h,U_h^c)_{0,h}
\le
\|f_h\|_{0,h}\,\|U_h^c\|_{0,h}
\le
C_P\|f_h\|_{0,h}|U_h^c|_{1,h}.
\]
Thus \(|U_h^c|_{1,h}\le C\), uniformly in \(h\). By the equivalence between
the discrete \(H^1\) seminorm and the \(H^1\)-seminorm of the piecewise
affine reconstruction, there exists a mesh-independent constant \(C\) such
that
\[
\|\mathcal I_h U_h^c\|_{H_0^1(\Omega)}\le C |U_h^c|_{1,h}\le C.
\]
Rellich--Kondrachov compactness therefore yields a subsequence converging
weakly in \(H_0^1(\Omega)\) and strongly in \(L^2(\Omega)\).
The Rellich--Kondrachov theorem thus gives strong \(L^2\)-convergence of a
subsequence to some limit \(u\in H_0^1(\Omega)\); this strong convergence is
then used to pass to the limit in the convective term before \(u\) is
identified as the unique weak solution.

It remains to identify the limit. Let \(\varphi\in C_c^\infty(\Omega)\) and
let \(\varphi_h\) be its grid sampling. The diffusive part converges by the
weak convergence of the reconstructed discrete gradients and the consistency
of the centered differences applied to the smooth test function. More
precisely, the reconstructed gradients of \(U_h^c\) are bounded in
\(L^2(\Omega)^2\), converge weakly to \(\nabla u\), and the sampled gradients
of \(\varphi\) converge strongly to \(\nabla\varphi\).

For the convective part, use the discrete summation-by-parts identity for
the centered first difference. In the \(x\)-direction, for instance,
\[
\left(
\frac{\beta_1}{2h}(U_{i+1,j}^c-U_{i-1,j}^c),\varphi_{ij}
\right)_{0,h}
=
-
\left(
U_{ij}^c,
\frac{\beta_1}{2h}(\varphi_{i+1,j}-\varphi_{i-1,j})
\right)_{0,h},
\]
where the boundary contribution vanishes because of the homogeneous
Dirichlet boundary condition and because \(\varphi\) is compactly supported.
The same identity holds in the \(y\)-direction. Since
\[
\frac{\varphi_{i+1,j}-\varphi_{i-1,j}}{2h}
=\partial_x\varphi(x_i,y_j)+O(h^2),
\qquad
\frac{\varphi_{i,j+1}-\varphi_{i,j-1}}{2h}
=\partial_y\varphi(x_i,y_j)+O(h^2),
\]
uniformly on the support of \(\varphi\), the strong \(L^2\)-convergence of
\(\mathcal I_hU_h^c\) gives
\[
-\int_\Omega u\,\bfbeta\cdot\nabla\varphi\,dx.
\]
Because \(\nabla\cdot\bfbeta=0\), \(u\in H_0^1(\Omega)\), and
\(\varphi\in C_c^\infty(\Omega)\), this expression is equal to
\[
\int_\Omega (\bfbeta\cdot\nabla u)\,\varphi\,dx,
\]
which is the continuous convective contribution in the weak formulation.

It remains to show that the ADSC correction vanishes in the limit. From
\eqref{eq:edge_coeffs}, since \(0\le \chi^x,\chi^y\le1\) and the transfer
factor is bounded, for instance \(0\le \eta_{\rm Pe}\le1\),
\[
\alpha^x_{i+\frac12,j}(U)
=|\beta_1|h\bigl(\gamma_0+\gamma_1\eta_{\rm Pe}\chi^x_{i+\frac12,j}(U)\bigr)
\le |\beta_1|h(\gamma_0+\gamma_1),
\]
and similarly
\(\alpha^y_{i,j+\frac12}(U)\le |\beta_2|h(\gamma_0+\gamma_1)\). Hence the
constant \(C_\alpha\) is independent of \(U\) and \(h\). With the discrete
edge inner product associated with \(\|\cdot\|_{0,h}\), the edge
representation reads
\[
\bigl(S_h^{\rm ADSC}(\chi(U))U,\varphi\bigr)_{0,h}
=
\sum_{r=x,y}(W_r(\chi(U))D_rU,D_r\varphi)_{0,h}.
\]
Since the diagonal entries of \(W_r\) are the edge coefficients
\(\alpha^r\), the bound \(0\le\alpha^r\le C_\alpha h\) from
\eqref{eq:edge_coeffs} gives
\[
\|W_r^{1/2}D_rZ_h\|_{0,h}
\le C h^{1/2}|Z_h|_{1,h},
\qquad Z_h\in V_h .
\]
Applying this bound once with \(Z_h=U_h^c\) and once with
\(Z_h=\varphi_h\) supplies the extra factor \(h\). Thus, by
Cauchy--Schwarz,
\[
\begin{aligned}
\left|
\bigl(S_h^{\rm ADSC}(\chi(U_h^c))U_h^c,\varphi_h\bigr)_{0,h}
\right|
&\le
\sum_{r=x,y}
\|W_r(\chi(U_h^c))^{1/2}D_rU_h^c\|_{0,h}
\|W_r(\chi(U_h^c))^{1/2}D_r\varphi_h\|_{0,h}  \\
&\le
C h\, |U_h^c|_{1,h}\,|\varphi_h|_{1,h}
\longrightarrow 0.
\end{aligned}
\]
The right-hand side converges by the assumed consistency of \(f_h\):
\[
(f_h,\varphi_h)_{0,h}\to (f,\varphi)_{L^2(\Omega)}.
\]
Therefore the limit satisfies the weak formulation of
\eqref{eq:strong}. The continuous weak problem is coercive for fixed
\(\eps>0\): the diffusion term controls \(\|\nabla u\|_{L^2}\), while the
constant-coefficient transport contribution is skew-symmetric on
\(H_0^1(\Omega)\). Thus the continuous weak solution is unique. Consequently,
all convergent subsequences have the same limit, and the whole reconstructed
family converges to \(u\) in \(L^2(\Omega)\).
\end{proof}

\paragraph{Scope of the coupled compactness statement.}
Propositions~\ref{prop:adsc_coupled_existence} and~\ref{prop:adsc_coupled_qualitative_convergence} provide an existence and
compactness statement for the regularized fully coupled ADSC formulation.
The existence argument uses only the continuity of the regularized detector
for \(\delta_h>0\), the positivity of the ADSC correction, and the coercivity
of the diffusive part for fixed \(\eps>0\). The convergence statement is
qualitative: it states that the family supplied by the preceding energy bound
admits subsequences converging to the weak solution of the continuous
convection--diffusion problem, and hence that any sequence of selected coupled solutions converges in
\(L^2\), since the continuous weak solution is unique in the present setting. It does not prove
uniqueness of the nonlinear discrete solution for fixed \(h\), convergence
of the activation update algorithm used in the computations, or an
energy-norm convergence rate. Such stronger conclusions would require an
additional stability estimate for the activation map \(U\mapsto\chi(U)\),
which is not assumed here.

\subsection{Modal interpretation}

For a fixed detector, ADSC shifts the modal footprint toward larger
effective real parts,
\[
	\lhpq
	\longmapsto
	\lhpq+\delta^{\mathrm{ADSC}}_{pq},
	\qquad
	\delta^{\mathrm{ADSC}}_{pq}\ge 0.
\]
ADSC is not an exact realization of $\Lstar$ but a sparse directional
surrogate preserving two key qualitative effects of the rectified
reference operator: a dissipative shift of the modal footprint, and a
directional selection of the regions where reinforcement is activated.

\section{Numerical experiments}
\label{sec:numerics}

\subsection{Setup}
\label{subsec:numerical_setup}

All experiments solve~\eqref{eq:strong} on $\Omega=(0,1)^2$ with
homogeneous Dirichlet conditions. The LFA symbol is used only to design and
diagnose interior dissipative mechanisms; all errors, extrema violations, and
reference comparisons below are computed for the finite Dirichlet problems.
This separation is intentional: the modal quantities explain the interior
mechanism, while the numerical tests assess its effect in the bounded-domain
Dirichlet setting. The numerical experiments use the coupled regularized
activation generated from the computed ADSC solution. They should therefore be
read as numerical evidence for the practical ADSC iteration, not as a direct
verification of Theorem~\ref{thm:adsc_pde_convergence}, which concerns the
fixed regularized operator selected from an auxiliary activation field
satisfying the stated discrete regularity assumptions. Moreover, the
computations use \(\delta_h=10^{-12}\), a sharp-regularized detector regime.
The mesh-uniform Lipschitz constants required in
Lemma~\ref{lem:regularized_detector_variation} are therefore not verified by
these runs; the numerical results illustrate the practical sharp-detector
limit rather than the mesh-uniform analytical regime of the lemma.

We compare seven methods: centered Galerkin, coordinate upwinding,
SUPG~\citep{Brooks1982}, CIP-type, LPS-type, AFC-inspired, and ADSC. Unless
explicitly stated otherwise, ADSC denotes the regularized detector version
defined in Section~\ref{subsec:regularized_detector}, with averaged edge
transfer as in~\eqref{eq:edge_average_detector}. The CIP-type, LPS-type,
and AFC-inspired entries are structured-grid representatives of the
corresponding stabilization mechanisms. They are included to compare
dissipative behavior, not to reproduce full finite-element production
implementations. The comparator methods use fixed standard parameter
choices. The values \(\gamma_{\rm CIP}=0.030\), \(\gamma_{\rm LPS}=1\),
and \(\theta_{\rm AFC}=1\) are selected once on the main benchmark to give
comparable dissipative strength, measured operationally by a reduction of
the same common-set modal indicator \(\bar\rho_{\mathrm{stab}}\) without
collapsing all methods to full upwind-like damping. They are not optimized
for error and are not refitted across tests. The
tables therefore compare representative dissipative footprints, not fully
optimized versions of every stabilization family. They should not be read
as a performance ranking against calibrated CIP, LPS, or AFC codes. Claims
about method-level performance are restricted to the fully specified
Cartesian methods: Galerkin, coordinate upwinding, SUPG, and ADSC. The
CIP-, LPS-, and AFC-inspired entries are controlled same-grid diagnostics of
characteristic dissipative mechanisms.

For reproducibility, the structured-grid representatives are defined as
follows. Let \(L_{xx}\), \(L_{yy}\) denote the centered second-difference
matrices and \(D_\beta=\beta_1G_x+\beta_2G_y\) the centered directional
derivative matrix. The CIP-type representative is
\[
	A_{\rm CIP}=A_{\rm Gal}
	+\gamma_{\rm CIP}|\bfbeta|h^3
	\bigl(L_{xx}^TL_{xx}+L_{yy}^TL_{yy}\bigr),
	\qquad \gamma_{\rm CIP}=0.030.
\]
For LPS-type, set \(P=T\otimes T\), where
\(T=\operatorname{tridiag}(1/4,1/2,1/4)\), \(H=I-P\), and
\(\tau=h/(2|\bfbeta|)\). The reported representative is
\[
	A_{\rm LPS}=A_{\rm Gal}
	+\gamma_{\rm LPS}\tau(HD_\beta)^T(HD_\beta),
	\qquad
	F_{\rm LPS}=F+\gamma_{\rm LPS}\tau(HD_\beta)^T(HF),
	\qquad
	\gamma_{\rm LPS}=1.
\]
The AFC-inspired edge-diffusion comparator uses the same directional detector only to
define a reproducible nonlinear edge-diffusion comparator:
\[
	A_{\rm AFC}=A_{\rm Gal}+D_x^TW_x^{\rm AFC}D_x
	+D_y^TW_y^{\rm AFC}D_y,\qquad
	\alpha_x^{\rm AFC}=\theta_{\rm AFC}\frac{|\beta_1|h}{2}\chi^x,
	\quad
	\alpha_y^{\rm AFC}=\theta_{\rm AFC}\frac{|\beta_2|h}{2}\chi^y,
\]
with \(\theta_{\rm AFC}=1\) and 80 limiting iterations. These
representatives are not tuned to provide production-level implementations
of CIP, LPS, or AFC methods. In particular, the AFC-inspired entry is not a
flux-corrected transport limiter; it is a detector-driven, positive
semidefinite edge-diffusion comparator included only to compare
characteristic dissipative behaviors within the same structured-grid
modal framework.
In SUPG computations, the linear system includes both the
streamline operator correction $\tau D_{\bfbeta}^{T} D_{\bfbeta}$ and the
right-hand side correction $\tau D_{\bfbeta}^{T}\mathbf f$.

\paragraph{Error and diagnostic quantities.}
Errors $\norm{e}_{L^2}$ and $\norm{e}_{L^\infty}$ are computed with
respect to a fine-grid reference solution (interpolated onto the coarse
grid). The discrete total variation used in the tables is
\begin{equation}\label{eq:discrete_tv}
\mathrm{TV}(U)=\sum_{i,j}\left(|U_{i+1,j}-U_{ij}|+|U_{i,j+1}-U_{ij}|\right).
\end{equation}
The quantity in~\eqref{eq:discrete_tv} is used only as a reported diagnostic, not as a stability norm. The main amplitude-based oscillation diagnostic is the extrema
violation
\[
	E_{\mathrm{ext}}
	=
	\max(0,\min U_{\rm ref}-\min U)
	+
	\max(0,\max U-\max U_{\rm ref}).
\]
The first term measures undershoot below the reference minimum, and the
second measures overshoot above the reference maximum. When an exact solution
is available, the reference extrema are the exact extrema.
When a directional detector count is reported, it counts sign changes
of directional increments above a shared absolute amplitude threshold.
This count is denoted by ``Det.'' and is used only as a local
non-monotonicity diagnostic along $\bfbeta$, not as a universal
oscillation norm. For a stabilized method, the reported modal indicator is
computed as
\[
	\rho^{\mathrm{stab}}_{pq}
	=
	\frac{|b_{pq}|}{a_{pq}+\delta^{\mathrm{stab}}_{pq}},
	\qquad
	\delta^{\mathrm{stab}}_{pq}
	=
	\frac{\langle S_h^{\mathrm{stab}}\phi_{pq},\phi_{pq}\rangle}
	{\|\phi_{pq}\|^2}.
\]
The reported mean value is taken over the Galerkin-dominant modal set
$\Icd$, so that all methods are compared on the same set of convectively
dominant modes. In table headings, \(\bar\rho_{\mathrm{stab}}\) denotes this
common-set diagnostic; for Galerkin it coincides with
\(\bar\rho_{\mathrm{Gal}}\), while for stabilized methods it is evaluated
from the final selected operator. For nonlinear or adaptive corrections such
as AFC-inspired and ADSC, $\delta^{\mathrm{stab}}_{pq}$ is evaluated from the
final stabilized operator selected by the nonlinear or activation procedure.
This modal-indicator value is an interior-symbol diagnostic of modal damping, not
an error estimator and not a replacement for the \(L^2\), \(L^\infty\),
extrema, and total-variation diagnostics. The tables below therefore do
not assume a monotone relation between smaller \(\bar\rho_{\mathrm{stab}}\)
and smaller solution error.

\paragraph{Benchmark scope.}
The numerical section is designed to test the modal mechanism under
controlled conditions: a one-dimensional reduction, a localized
two-dimensional source, mesh refinement, manufactured exact solutions in
the active and inactive ADSC regimes, parameter sensitivity, right-hand
side variation, and convection-direction variation. To complement these
controlled tests, Sections~\ref{subsec:nist_uniform}--\ref{subsec:nist_shishkin}
add a NIST-type exponential boundary-layer benchmark with exact data,
motivated by the NIST adaptive-test collection~\citep{Mitchell2010} and by
standard layer-adapted mesh methodology~\citep{Roos2008}. This directly
addresses the sharp-gradient objection on both a uniform Cartesian mesh and
a tensor-product Shishkin mesh. More extensive benchmark suites with
internal layers, fully anisotropic meshes, and non-Cartesian geometry
remain outside the proof framework of this paper. In particular, no calibrated
AFC implementation and no oblique internal-layer benchmark are used as
performance claims in this version; these are listed in the conclusion as
priority extensions for a broader benchmark campaign.

\paragraph{Numerical roadmap.}
Table~\ref{tab:numerical_roadmap} summarizes how the numerical section is
used. The goal is not to rank all stabilized methods uniformly, but to check
whether the regularized ADSC construction gives the effects predicted by the
modal design: extrema control, reduction of the interior modal indicator,
and the expected active/inactive consistency behavior.

\begin{table}[htbp]
\centering
\small
\caption{Numerical roadmap and interpretation of the reported tests.}
\label{tab:numerical_roadmap}
\begin{tabular}{L{0.24\textwidth}L{0.34\textwidth}L{0.34\textwidth}}
\toprule
Test block & What is checked & Interpretation \\
\midrule
Main Gaussian test & Extrema violation, detector count, and mean modal indicator & Tests the modal-directional mechanism on a bounded Dirichlet problem. \\
Fixed-reference activation & ADSC with an activation generated from the fine reference and then frozen & Numerically illustrates the fixed-activation framework used in Theorem~\ref{thm:adsc_pde_convergence}, without claiming an energy-norm rate or algorithmic convergence for the fully coupled activation. \\
Mesh refinement & Observed error decrease and inter-level slopes & Reports monotone error decrease while the active correction remains first-order dominated. \\
Manufactured active/inactive tests & Exact-solution rates & Checks the active-regime first-order behavior and recovery of the centered second-order regime when ADSC is inactive. \\
NIST-inspired exponential boundary-layer tests and small-diffusion sweep & Robustness on exact exponential layers & Probes under-resolved layers and varying diffusion; results are robustness diagnostics, not \(\eps\)-uniform stability claims. \\
Timing, sensitivity, and few-shot activation & Cost, parameter dependence, and reduced activation updates & Reports the computational price of the reference activation and indicates that a few-shot variant preserves the main diagnostics at much lower cost. \\
\bottomrule
\end{tabular}
\end{table}

\paragraph{Fine-grid reference.}
No closed-form analytical solution is available for the localized
Gaussian tests used in the main comparison.
For the main test, the reference solution is computed at
$N_{\mathrm{ref}}=264$ ($\Peh^{\mathrm{ref}}<1$); for
\(\eps=2\times10^{-3}\) and \(|\bfbeta|=1\), this gives a
diffusion-resolved centered reference. A check at
$N_{\mathrm{ref}}=360$ gives reference-to-reference differences below
$10^{-5}$ in $L^2$ and $5.5\times 10^{-5}$ in $L^\infty$; the resulting
changes in reported coarse-grid $L^2$ errors are below $2.5\times 10^{-6}$
for all methods. Both reference systems are solved with the same centered
Galerkin finite-difference operator on the refined grid using MATLAB's
sparse direct linear solver.

\subsection{Reduced one-dimensional benchmark}

Parameters: $\eps=10^{-3}$, $\beta=1$, $N_e=80$, $\Peh=6.25$,
Gaussian source centered at $x_0=0.5$.
In this configuration, 71 modes out of 79 satisfy $\rho_k>1$ (89.9\%).

\begin{table}[htbp]
	\centering
	\caption{Reduced one-dimensional benchmark, $\Peh=6.25$.}
	\label{tab:1d}
	\begin{tabular}{lcccc}
		\toprule
		Method      & $\norm{e}_{L^2}$         & $\norm{e}_{L^\infty}$    & Osc. & $\bar\rho_{\mathrm{stab}}$              \\
		\midrule
		Galerkin    & $6.245\!\times\!10^{-3}$ & $3.851\!\times\!10^{-2}$ & 21   & $1.125\!\times\!10^{1}$ \\
		Upwind/SUPG & $1.282\!\times\!10^{-3}$ & $6.066\!\times\!10^{-3}$ & 0    & $1.552$                 \\
		ADSC        & $9.378\!\times\!10^{-4}$ & $4.836\!\times\!10^{-3}$ & 2    & $1.659$                 \\
		\bottomrule
	\end{tabular}
\end{table}

Table~\ref{tab:1d} reports the reduced one-dimensional benchmark. At the level of this reduced one-dimensional discrete interior model,
the coincidence of upwinding and SUPG is verified to machine precision.
This identity is specific to the finite-difference formulation used here;
finite-element SUPG variants with consistent mass matrices can differ.
ADSC achieves the smallest $L^2$ and $L^\infty$ errors with two residual
detected oscillations and a mean modal indicator slightly above
upwinding/SUPG (\(1.659\) versus \(1.552\)). Thus, in this reduced 1D test,
the partial ADSC activation does not completely remove the dominant modal
peak even though it improves the error norms.
CIP-type, LPS-type, and AFC-inspired edge-diffusion comparators are not reported in this
one-dimensional benchmark, because their structured-grid definitions in
Section~\ref{sec:numerics} are included for the two-dimensional comparison
and do not reduce to a unique one-dimensional representative.

\subsection{Two-dimensional directional convection test}

Parameters: $\eps=2\times 10^{-3}$,
$\bfbeta=(1,0.6)/\sqrt{1^2+0.6^2}$, $N_e=45$, $\Peh=5.56$.
Gaussian source with $\sigma=0.07$ centered at $(0.5,0.5)$.
ADSC parameters: $\gamma_{\min}=0.08$, $\gamma_{\max}=0.25$,
$\kappa=2$, $\omega=0.35$, \(\delta_h=10^{-12}\), and
\(\eta_{\rm det}=5\times10^{-2}\); $\gamma_0(5.56)=0.198$.
For the same test, \(\bar\rho_{\mathrm{Gal}}=3.878\),
\(\Peh=5.56\), and \(\bar B\approx0.686\). The modal-balance formula gives
\(\gamma_0^{\mathrm{bal}}\approx0.378\), while the saturating law used in
the simulations gives the more conservative value \(\gamma_0=0.198\).
This is consistent with the predicted scaling order but also shows that the implemented
parameter choice is deliberately below the full balance target.
The monotone regularized activation iteration reaches a relative variation
$4.34\times10^{-10}$ in 47 iterations, and the residual of the final
fixed-activation linear system is $1.77\times10^{-15}$. The final
activation mass is \(7.538\times10^2\), with 936 nodes above the
\(10^{-3}\) activation threshold out of 1936 interior nodes.
At this $\Peh$, 1775/1936 modal pairs are convectively dominant (91.7\%).

\begin{table}[htbp]
\centering
\small
\caption{Modal-balance calibration diagnostic over the mesh-refinement family.
The quantity \(\gamma_0^{\mathrm{bal}}\) is computed from
\eqref{eq:gamma_projected_choice} with \(\rho_{\mathrm{target}}=1\), using
the Galerkin-dominant set. The implemented \(\gamma_0\) is the bounded
saturating value used in all simulations.}
\label{tab:gamma_balance_validation}
\begin{tabular}{cccccc}
\toprule
\(N_e\) & \(\Peh\) & \(\bar\rho_{\mathrm{Gal}}\) & \(\bar B\) & \(\gamma_0^{\mathrm{bal}}\) & implemented \(\gamma_0\) \\
\midrule
30  & 8.33 & 5.359 & 0.686 & 0.381 & 0.214 \\
45  & 5.56 & 3.878 & 0.686 & 0.378 & 0.198 \\
60  & 4.17 & 3.181 & 0.684 & 0.383 & 0.184 \\
90  & 2.78 & 2.575 & 0.685 & 0.414 & 0.160 \\
120 & 2.08 & 2.311 & 0.690 & 0.456 & 0.140 \\
\bottomrule
\end{tabular}
\end{table}

The implemented baseline values are deliberately below the raw modal-balance
values in Table~\ref{tab:gamma_balance_validation}. The chosen interval
\([0.08,0.25]\) is a conservative fraction of the observed raw-balance range
\([0.378,0.456]\), roughly one fifth to two thirds of those values. This
choice avoids turning the baseline correction into a nearly uniform upwind-like
diffusion, limits over-smoothing of the reference layer, and leaves the
reinforced term \(\gamma_1\eta_{\Pe}\chi\) to act only where the directional
detector identifies local non-monotonicity. It is a reproducible heuristic,
not an optimized parameter law. Figures~\ref{fig:solutions_2d},
\ref{fig:errors_2d_common_scale}, and~\ref{fig:rho_2d} provide the corresponding
solution, error, and modal-indicator diagnostics.

\begin{figure}[htbp]
	\centering
	\includegraphics[width=0.92\textwidth]{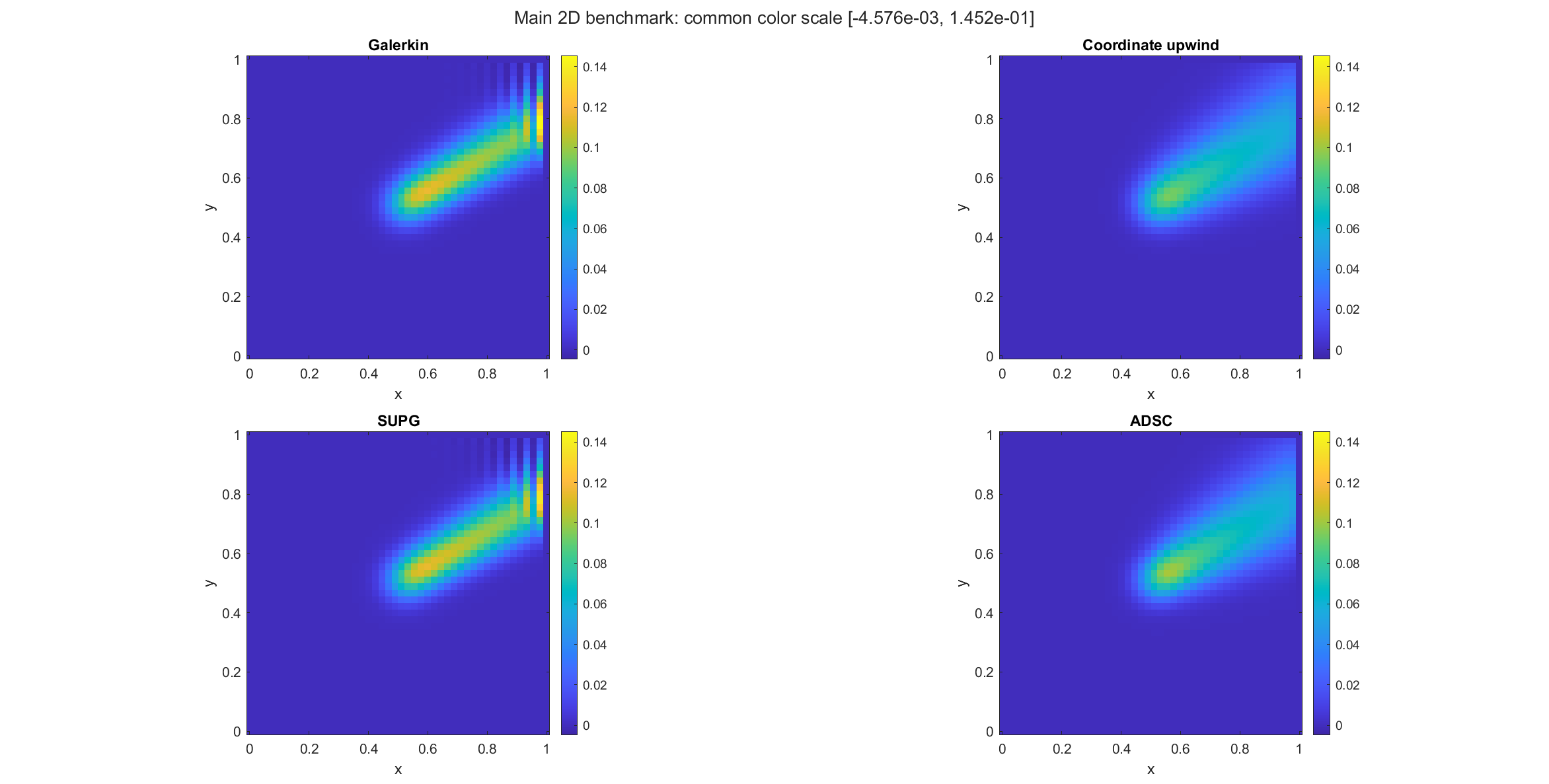}
	\caption{Two-dimensional numerical solutions for the main test case associated with~\eqref{eq:strong}, plotted with a common color scale. The common scale makes the amplitude comparison visually consistent with the quantitative diagnostics in Table~\ref{tab:2d}; the color range is approximately \([-4.58\times10^{-3},1.45\times10^{-1}]\). Galerkin displays visible directional oscillations, coordinate upwinding gives the strongest smoothing, SUPG acts through streamline dissipation, and ADSC provides a localized directional correction. CIP-type, LPS-type, and AFC-inspired solution maps are not shown; their quantitative diagnostics appear in Table~\ref{tab:2d}.}
	\label{fig:solutions_2d}
\end{figure}

\begin{figure}[htbp]
	\centering
	\includegraphics[width=0.92\textwidth]{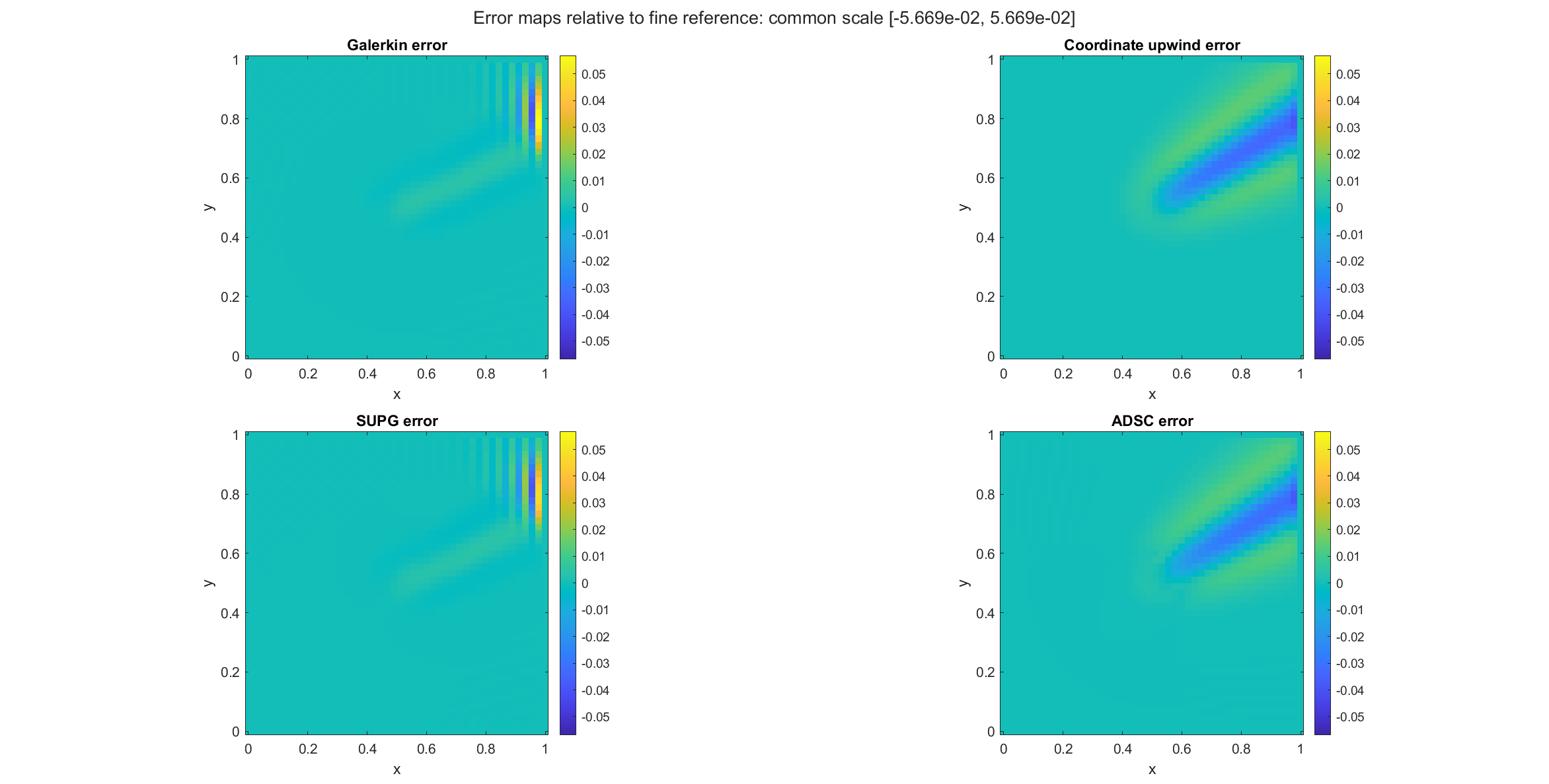}
	\caption{Error maps relative to the fine-grid reference solution for the main two-dimensional benchmark, plotted with a common symmetric color scale. The maps complement Table~\ref{tab:2d} by showing where the different stabilizations introduce oscillatory errors, streamline smearing, or localized directional damping.}
	\label{fig:errors_2d_common_scale}
\end{figure}

\begin{figure}[htbp]
	\centering
	\includegraphics[width=0.92\textwidth]{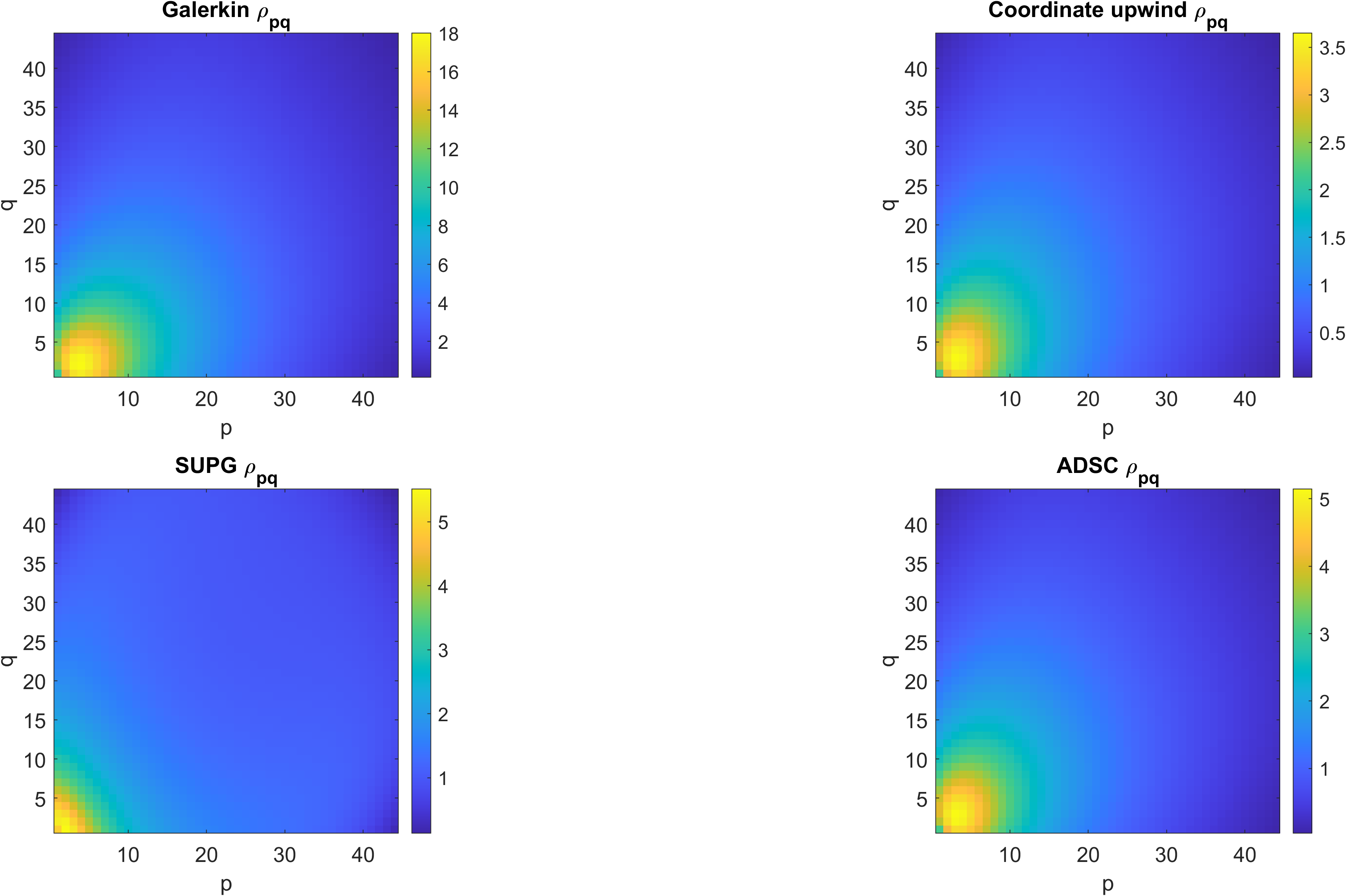}
	\caption{Maps of the modal convection-dominance indicator $\rho_{pq}$ defined in~\eqref{eq:indicator}.
		Coordinate upwinding produces the strongest uniform reduction, SUPG
		acts anisotropically through streamline diffusion, and ADSC gives a
		selective reduction consistent with its sparse directional design.
		The mean indicator over the Galerkin-dominant set is reduced from
		$3.878$ (Galerkin) to $1.123$ (ADSC), compared with $0.811$ for
		coordinate upwinding and $1.298$ for SUPG.}
	\label{fig:rho_2d}
\end{figure}

\begin{table}[htbp]
	\centering
	\caption{Quantitative synthesis of two-dimensional stabilization effects,
		$N_e=45$. Here Det. denotes the number of grid nodes with a directional
		sign change above the common absolute threshold defined in
		Section~\ref{subsec:numerical_setup}.}
	\label{tab:2d}
	\begin{tabular}{lcccccc}
		\toprule
		Method        & $\norm{e}_{L^2}$         & $\norm{e}_{L^\infty}$    & TV
		              & Det. & $E_{\mathrm{ext}}$       & $\bar\rho_{\mathrm{stab}}$                                                        \\
		\midrule
		Galerkin      & $4.394\!\times\!10^{-3}$ & $5.669\!\times\!10^{-2}$ & $2.443\!\times\!10^{-1}$ & 262 & $3.700\!\times\!10^{-2}$ & 3.878 \\
		Coord.~upwind & $6.813\!\times\!10^{-3}$ & $3.877\!\times\!10^{-2}$ & $1.423\!\times\!10^{-1}$ & 9   & 0                        & 0.811 \\
		SUPG          & $4.037\!\times\!10^{-3}$ & $4.772\!\times\!10^{-2}$ & $2.425\!\times\!10^{-1}$ & 273 & $2.869\!\times\!10^{-2}$ & 1.298 \\
		CIP-type      & $2.856\!\times\!10^{-3}$ & $4.353\!\times\!10^{-2}$ & $2.099\!\times\!10^{-1}$ & 192 & $2.020\!\times\!10^{-2}$ & 2.560 \\
		LPS-type      & $4.129\!\times\!10^{-3}$ & $5.212\!\times\!10^{-2}$ & $2.434\!\times\!10^{-1}$ & 276 & $3.296\!\times\!10^{-2}$ & 2.536 \\
		AFC-inspired      & $5.869\!\times\!10^{-3}$ & $3.696\!\times\!10^{-2}$ & $1.501\!\times\!10^{-1}$ & 14  & $1.080\!\times\!10^{-5}$ & 1.720 \\
		ADSC          & $6.123\!\times\!10^{-3}$ & $3.775\!\times\!10^{-2}$ & $1.474\!\times\!10^{-1}$ & 9   & $6.100\!\times\!10^{-10}$ & 1.123 \\
		\bottomrule
	\end{tabular}
\end{table}

Recall from the setup that the CIP-, LPS-, and AFC-inspired entries are
structured-grid representatives of the corresponding stabilization
mechanisms and are not production implementations; the comparisons below
concern only the dissipative behavior within this specific structured-grid
modal framework.
Table~\ref{tab:2d} highlights major differences from the one-dimensional
setting. SUPG reduces both \(\bar\rho_{\mathrm{stab}}\) and \(E_{\mathrm{ext}}\) relative to
	Galerkin while retaining a low $L^2$ error, consistent with the expected
streamline-diffusion effect of SUPG. Within the structured-grid
representatives used here, the CIP-type entry gives the smallest \(L^2\)
error but leaves a larger modal indicator, whereas the AFC-inspired entry and
coordinate upwinding strongly reduce extrema violations at the cost of
stronger smoothing. ADSC should not be read as uniformly superior to
coordinate upwinding: in Table~\ref{tab:2d}, upwinding has exactly zero
\(E_{\mathrm{ext}}\), a slightly smaller \(L^\infty\) error, and a much
lower one-shot computational cost. The ADSC advantage in this test is
instead more specific: it gives near-upwind extrema control while retaining
a less uniformly damped directional modal footprint than full coordinate
upwinding and a lower \(\bar\rho_{\mathrm{stab}}\) than SUPG and the structured-grid
CIP-, LPS-, and AFC-inspired edge-diffusion comparators. This statement is limited to the
representatives defined in Section~\ref{sec:numerics} and is not a claim
of superiority over calibrated production implementations.

\begin{table}[htbp]
	\centering
	\caption{Extrema diagnostics, $N_e=45$.}
	\label{tab:extrema}
	\begin{tabular}{lccc}
		\toprule
		Method        & Undershoot               & Overshoot                & $E_{\mathrm{ext}}$       \\
		\midrule
		Galerkin      & $4.182\!\times\!10^{-3}$ & $3.281\!\times\!10^{-2}$ & $3.700\!\times\!10^{-2}$ \\
		Coord.~upwind & 0                        & 0                        & 0                        \\
		SUPG          & $4.576\!\times\!10^{-3}$ & $2.412\!\times\!10^{-2}$ & $2.869\!\times\!10^{-2}$ \\
		CIP-type      & $3.194\!\times\!10^{-4}$ & $1.988\!\times\!10^{-2}$ & $2.020\!\times\!10^{-2}$ \\
		LPS-type      & $4.483\!\times\!10^{-3}$ & $2.847\!\times\!10^{-2}$ & $3.296\!\times\!10^{-2}$ \\
		AFC-inspired      & $1.080\!\times\!10^{-5}$ & 0                        & $1.080\!\times\!10^{-5}$ \\
		ADSC          & $6.100\!\times\!10^{-10}$ & 0                       & $6.100\!\times\!10^{-10}$ \\
		\bottomrule
	\end{tabular}
\end{table}

Figures~\ref{fig:footprint_2d} and~\ref{fig:detector_2d} then display the modal footprint and the spatial activation pattern for the same test. Table~\ref{tab:extrema} separates amplitude preservation from modal
reduction. SUPG and CIP-type stabilization reduce the extrema violation
but do not remove it. Coordinate upwinding, AFC-inspired correction, and ADSC
remove the overshoot; ADSC reduces the total extrema violation by more
than seven orders of magnitude relative to Galerkin
($3.700\times10^{-2}$ to $6.100\times10^{-10}$) while keeping the modal
reduction tied to the directional indicator.

\begin{figure}[htbp]
	\centering
	\includegraphics[width=0.72\textwidth]{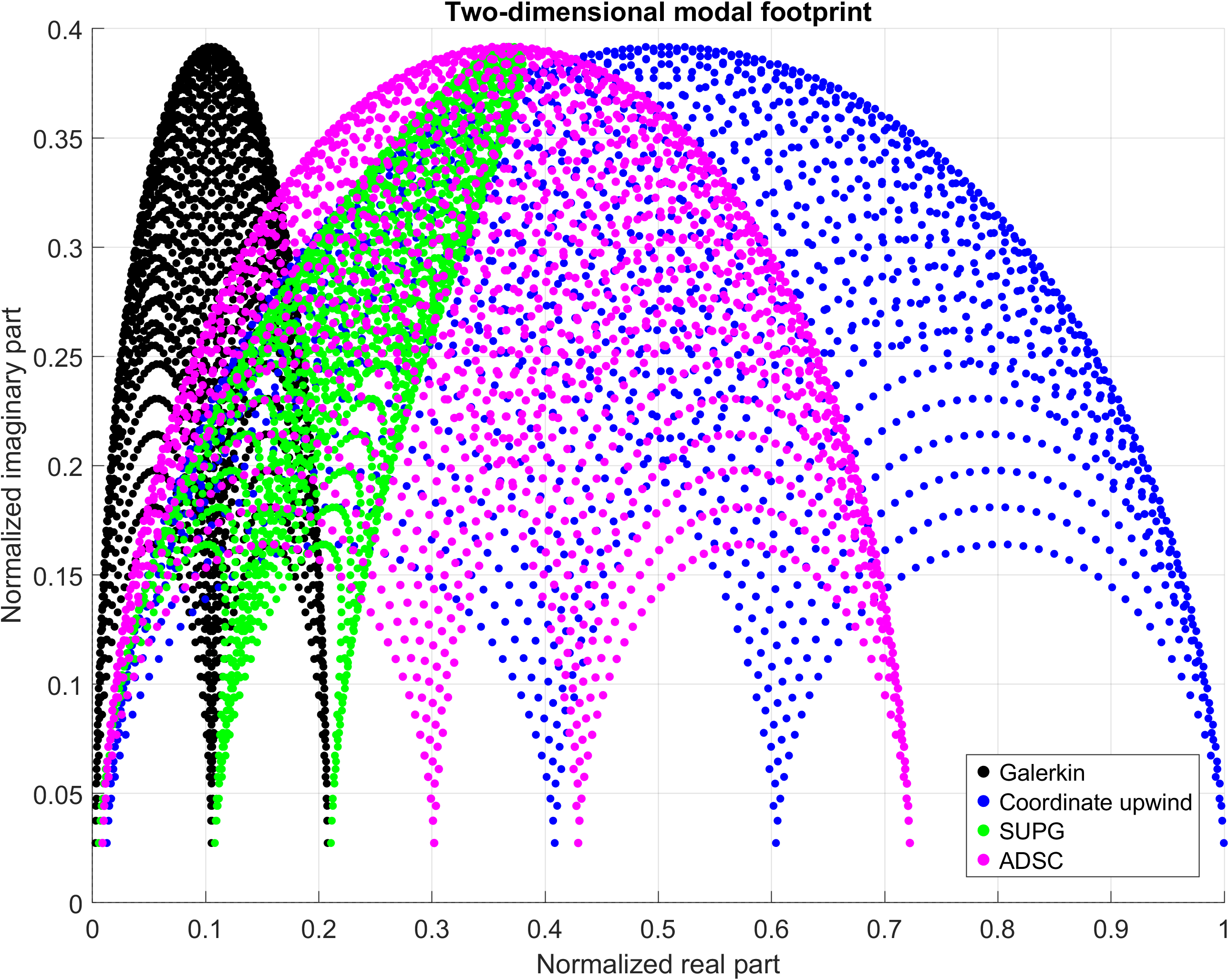}
	\caption{Normalized modal footprint \(\mathcal F_h\) from Section~\ref{sec:footprint}
		in the complex plane for the representative fixed-symbol methods used in
		the modal-design argument:
		Galerkin, coordinate upwinding, SUPG, and ADSC. Stabilized methods shift
		the modal contributions toward larger effective real parts. ADSC does
		not realize the ideal modal rectification exactly; it provides a sparse
		local surrogate for its dissipative effect. Both axes are normalized by
		the maximum modulus of the Galerkin symbol over the displayed modal
		set. CIP-type and LPS-type are
		omitted for readability, while AFC-inspired is not shown because its
		effective footprint is solution-dependent. Full quantitative
		comparisons are reported in Tables~\ref{tab:2d},
		\ref{tab:rhs}, and~\ref{tab:direction}.}
	\label{fig:footprint_2d}
\end{figure}

\begin{figure}[htbp]
	\centering
	\includegraphics[width=0.62\textwidth]{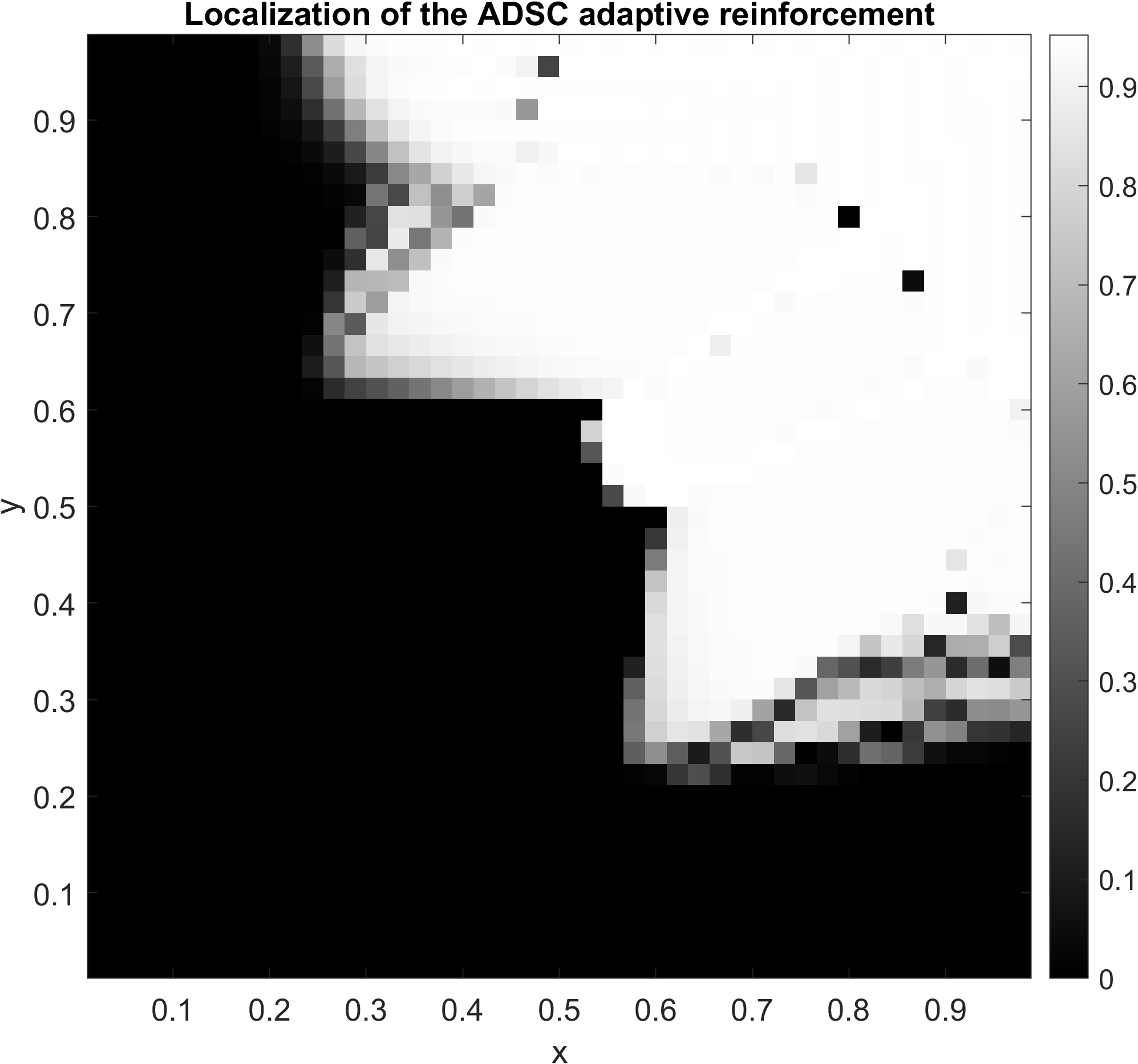}
		\caption{Spatial localization of the monotone regularized ADSC activation in the
		main test. The adaptive reinforcement is activated on a structured
		subset of the domain, unlike uniform dissipative corrections. During construction
		of the stabilized operator, the regularized detector reaches an
		activation mass of \(7.538\times10^2\), with 936 nodes above
		\(10^{-3}\) activation out of 1936, about 48\%. The activation pattern
		is consistent with the convection direction \(\bfbeta=(1,0.6)/\sqrt{1.36}\):
		nodes are activated downstream of the Gaussian source, where the
		transported front creates directional non-monotonicity. After the final fixed-activation
		solve, the separate
		post-processed directional non-monotonicity diagnostic reported in
		Table~\ref{tab:mesh_diag} is 9 for ADSC, compared with 262 for Galerkin.}
	\label{fig:detector_2d}
\end{figure}

\subsection{Fixed-reference activation test}
\label{subsec:fixed_reference_activation}

The conditional convergence result in Theorem~\ref{thm:adsc_pde_convergence}
concerns a fixed regularized ADSC operator selected from an auxiliary
activation field. By contrast, the practical ADSC iteration used in the
preceding tables generates the activation from the computed stabilized
solution. To make this distinction numerically explicit, we add a
fixed-reference activation test. The fine-grid reference solution is
interpolated onto the coarse grid, the regularized detector is applied to
this interpolated field, and the resulting activation is then frozen before
solving a single stabilized linear system.

This experiment is not a proof of convergence for the fully coupled
nonlinear activation. Its purpose is narrower: it verifies that the
fixed-activation mechanism appearing in the theorem already reproduces the
main qualitative effects expected from ADSC, namely extrema control and
reduction of the mean modal indicator.

\begin{table}[htbp]
\centering
\small
\caption{Fixed-reference activation test in the main two-dimensional benchmark.
The fixed-reference ADSC activation is generated from the fine reference
solution interpolated on the coarse grid and then held fixed. The column
\(\|U-U_{\mathrm{cpl}}\|_{L^2}\) measures the distance to the fully coupled
ADSC solution.}
\label{tab:fixed_reference_activation}
\begin{tabular}{lcccccc}
\toprule
Method & \(\|e\|_{L^2}\) & \(\|e\|_{L^\infty}\) & \(E_{\mathrm{ext}}\) & Det. & \(\overline{\rho}_{\mathrm{stab}}\) & \(\|U-U_{\mathrm{cpl}}\|_{L^2}\) \\
\midrule
Galerkin       & \(4.394\!\times\!10^{-3}\) & \(5.669\!\times\!10^{-2}\) & \(3.700\!\times\!10^{-2}\) & 262 & 3.878 & -- \\
ADSC fixed-ref & \(5.768\!\times\!10^{-3}\) & \(3.152\!\times\!10^{-2}\) & \(3.010\!\times\!10^{-8}\) & 32  & 1.188 & \(1.094\!\times\!10^{-3}\) \\
ADSC coupled   & \(6.123\!\times\!10^{-3}\) & \(3.775\!\times\!10^{-2}\) & \(6.100\!\times\!10^{-10}\) & 9 & 1.123 & 0 \\
\bottomrule
\end{tabular}
\end{table}

Table~\ref{tab:fixed_reference_activation} indicates that the fixed-reference
operator is not merely a formal construction. Relative to Galerkin, it
reduces the total extrema violation from \(3.700\times10^{-2}\) to
\(3.010\times10^{-8}\), while decreasing the mean modal indicator from
\(3.878\) to \(1.188\). The coupled ADSC iteration is more selective in
this test, with a smaller final detector count and a slightly lower mean
modal indicator. Nevertheless, the fixed-reference and coupled ADSC
solutions are close at the level of the reported diagnostics, and their
\(L^2\)-distance is \(1.094\times10^{-3}\). This supports the use of the
fixed-activation analysis as a relevant theoretical proxy, while keeping
separate the still-open questions of nonlinear discrete uniqueness,
energy-norm rates, and algorithmic convergence for the fully coupled
activation.

\subsection{Sensitivity analysis for ADSC parameters}

\begin{table}[htbp]
	\centering
	\caption{Sensitivity to parameter bounds, main 2D test.}
	\label{tab:sensitivity}
	\begin{tabular}{ccccccc}
		\toprule
		$\gamma_{\min}$ & $\gamma_{\max}$  & $\kappa$
		                & $\gamma_0(\Peh)$ & $\norm{e}_{L^\infty}$ & Det.  & $\bar\rho_{\mathrm{stab}}$                             \\
		\midrule
		0.08            & 0.25             & 2.00                  & 0.198 & $3.775\!\times\!10^{-2}$ & 9   & 1.123 \\
		0.10            & 0.30             & 1.50                  & 0.239 & $3.875\!\times\!10^{-2}$ & 9   & 1.059 \\
		0.10            & 0.30             & 2.00                  & 0.239 & $4.464\!\times\!10^{-2}$ & 9   & 0.981 \\
		0.12            & 0.35             & 2.00                  & 0.280 & $5.008\!\times\!10^{-2}$ & 10  & 0.871 \\
		\bottomrule
	\end{tabular}
\end{table}

Table~\ref{tab:sensitivity} shows the expected trade-off. Larger
parameter bounds reduce \(\bar\rho_{\mathrm{stab}}\) more strongly, but increase the
$L^\infty$ error through stronger smoothing. The selected parameter
triple $(0.08,0.25,2)$ is a conservative, reproducible choice below the raw
modal-balance target, not an error-optimized setting. It keeps
\(\bar\rho_{\mathrm{stab}}\) near
the intended modal-reduction range
while keeping the \(L^\infty\) error at \(3.775\times10^{-2}\), compared
with \(5.008\times10^{-2}\) for the most aggressive setting
\((\gamma_{\min},\gamma_{\max})=(0.12,0.35)\).

\subsection{Iteration sensitivity}

\begin{table}[htbp]
	\centering
	\caption{Sensitivity to maximum number of fixed-point iterations.}
	\label{tab:iter}
	\begin{tabular}{ccccc}
		\toprule
		Max.~iter. & $\norm{e}_{L^\infty}$    & Det. & $\bar\rho_{\mathrm{stab}}$ & Final variation          \\
		\midrule
		100        & $3.775\!\times\!10^{-2}$ & 9    & 1.123      & $4.343\!\times\!10^{-10}$ \\
		300        & $3.775\!\times\!10^{-2}$ & 9    & 1.123      & $4.343\!\times\!10^{-10}$ \\
		500        & $3.775\!\times\!10^{-2}$ & 9    & 1.123      & $4.343\!\times\!10^{-10}$ \\
		1000       & $3.775\!\times\!10^{-2}$ & 9    & 1.123      & $4.343\!\times\!10^{-10}$ \\
		\bottomrule
	\end{tabular}
\end{table}

Table~\ref{tab:iter} illustrates the convergence mechanism of
Proposition~\ref{prop:adsc_convergence}. The activation field reaches the
prescribed relative-change tolerance well before 100 iterations; after the
final fixed-activation solve, the reported
quantities are identical across the tested iteration caps up to the
displayed precision.

\begin{table}[htbp]
	\centering
	\small
	\caption{Sensitivity of ADSC to relaxation and initialization in the main
		test.}
	\label{tab:omega_init}
	\begin{tabular}{lccccc}
		\toprule
		Case & $\norm{e}_{L^\infty}$ & Det. & $\bar\rho_{\mathrm{stab}}$ & Iter. & Final variation \\
		\midrule
		$\omega=0.20$ & $3.773\!\times10^{-2}$ & 9 & 1.124 & 86 & $9.806\!\times10^{-10}$ \\
		$\omega=0.35$ & $3.775\!\times10^{-2}$ & 9 & 1.123 & 47 & $4.343\!\times10^{-10}$ \\
		$\omega=0.50$ & $3.778\!\times10^{-2}$ & 9 & 1.122 & 31 & $1.729\!\times10^{-10}$ \\
		$\omega=0.75$ & $3.780\!\times10^{-2}$ & 9 & 1.121 & 18 & $2.186\!\times10^{-11}$ \\
		$\omega=1.00$ & $3.781\!\times10^{-2}$ & 9 & 1.120 & 10 & $0$ \\
		\midrule
		Zero start     & $3.646\!\times10^{-2}$ & 9  & 1.161 & 48 & $3.950\!\times10^{-10}$ \\
		Coarse-grid start & $3.306\!\times10^{-2}$ & 10 & 0.980 & 78 & $2.539\!\times10^{-11}$ \\
		Galerkin start & $3.775\!\times10^{-2}$ & 9  & 1.123 & 47 & $4.343\!\times10^{-10}$ \\
		Upwind start   & $3.139\!\times10^{-2}$ & 16 & 1.248 & 48 & $3.292\!\times10^{-10}$ \\
		\bottomrule
	\end{tabular}
\end{table}

Table~\ref{tab:omega_init} addresses the practical question of whether
the 47 activation solves are an artifact of the relaxation or of the
Galerkin warm start. Increasing \(\omega\) mainly reduces the number of
activation iterations, while the final diagnostics remain close at the
displayed precision. Replacing the Galerkin warm start by a zero,
coarse-grid Galerkin, or upwind initialization has a more visible effect
for the regularized detector: the final detector count ranges from 9 to 16,
\(\bar\rho_{\mathrm{stab}}\) ranges from 0.980 to 1.248 and the
\(L^\infty\) error ranges from \(3.139\times10^{-2}\) to
\(3.775\times10^{-2}\). Thus the reported cost is not a consequence of a
fragile relaxation choice, but the selected regularized activation can
retain a measurable dependence on the initial profile in this nonlinear
test. The cost reported below is the cost of resolving that monotone
activation stabilization with sparse nearest-neighbor linear systems.

\paragraph{Computational cost.}
In the main test, Galerkin, coordinate upwinding, SUPG, CIP-type, and
LPS-type each require one sparse linear solve. AFC-inspired uses 80 nonlinear
limiting iterations in the implementation used here. ADSC requires 47
activation linear solves, followed by one final fixed-activation sparse solve.
Thus ADSC is more expensive than the one-shot linear stabilizations, but
its overhead is controlled by the monotone activation stabilization reported in
Table~\ref{tab:iter}. The important structural point is that every ADSC
linear system has the same sparse nearest-neighbor form; no dense
rectified operator is assembled. Consequently, the per-solve cost remains
of the same sparse nearest-neighbor order throughout the activation loop,
even though several such solves are required.

The following MATLAB wall-clock timings were measured with
\texttt{tic}/\texttt{toc} in one representative run of the same
implementation, excluding plotting and reference-grid generation.
Method-specific timings include the method-specific matrix assembly and
sparse solves after the common operator and source construction. Since
single-run timings can vary with caching and background system activity,
these values are reported only as implementation-level indicators for the
present MATLAB code, not as hardware-independent benchmarks and not as the
basis for ranking the methods. Algorithmic comparisons should primarily
rely on the number and sparse structure of the linear systems.

\begin{table}[htbp]
	\centering
	\small
	\caption{Representative MATLAB wall-clock timings over the mesh-refinement
	run for the fully specified Cartesian methods. Timings are single-run
	implementation indicators, not hardware-independent benchmarks.}
	\label{tab:timing_scaling}
	\begin{tabular}{ccccccc}
		\toprule
		\(N_e\) & Gal. (s) & Upwind (s) & SUPG (s) & ADSC (s)
		& ADSC/Gal. & ADSC iter. \\
		\midrule
		30  & \(1.58\!\times10^{-3}\) & \(1.15\!\times10^{-3}\) & \(3.02\!\times10^{-3}\) & \(7.68\!\times10^{-2}\) & 48.56 & 38 \\
		45  & \(3.85\!\times10^{-3}\) & \(3.83\!\times10^{-3}\) & \(7.91\!\times10^{-3}\) & \(2.05\!\times10^{-1}\) & 53.19 & 47 \\
		60  & \(7.28\!\times10^{-3}\) & \(6.14\!\times10^{-3}\) & \(1.61\!\times10^{-2}\) & \(3.66\!\times10^{-1}\) & 50.28 & 45 \\
		90  & \(1.51\!\times10^{-2}\) & \(1.40\!\times10^{-2}\) & \(2.83\!\times10^{-2}\) & \(9.22\!\times10^{-1}\) & 60.91 & 49 \\
		120 & \(2.38\!\times10^{-2}\) & \(2.63\!\times10^{-2}\) & \(6.28\!\times10^{-2}\) & \(1.47\)                 & 61.65 & 51 \\
		\bottomrule
	\end{tabular}
\end{table}

Table~\ref{tab:timing_scaling} reports that the present MATLAB
activation implementation is significantly more expensive than one-shot
linear stabilizations, and that the relative overhead grows over the
tested range. The activation iteration count remains moderate over this mesh range
(38, 47, 45, 49, and 51 iterations for \(N_e=30,45,60,90,120\)),
so the cost increase is mainly the repeated
sparse solves. This is a practical
limitation of the current implementation and should be weighed against the
modal and extrema diagnostics. Decomposing the ratio gives a more precise picture: \(t_{\rm ADSC}/[(n_{\rm iter}+1)t_{\rm Gal}]\) is approximately
1.25, 1.11, 1.09, 1.22, and 1.19 for
\(N_e=30,45,60,90,120\), respectively. Thus the per-solve cost of the
ADSC matrix remains of the same sparse order as the Galerkin solve in
this implementation; the main overhead is the number of activation solves,
not a dense or higher-complexity matrix. The growth of the total ratio is
therefore an implementation-level warning about repeated sparse solves,
not evidence that each ADSC solve has a higher asymptotic order. A
reduced-activation or one-shot ADSC variant can lower the cost
substantially. This is quantified in the few-shot study reported in
Section~\ref{subsec:few_shot}.

\subsection{Few-shot activation and practical cost}
\label{subsec:few_shot}

The timings in Table~\ref{tab:timing_scaling} correspond to the fully
converged monotone activation procedure. This is useful as a reference
implementation, but it overestimates the practical cost of ADSC when the
activation process is stopped after a small prescribed number of updates.
To check whether this reduced-activation strategy is stable under mesh
refinement, we repeat the main two-dimensional benchmark for
\(N_e=30,45,60,90,120\) with 5, 10, and 1000 allowed activation updates.
The 1000-iteration cap represents the reference coupled ADSC output obtained with the prescribed stopping criterion or the iteration cap
on each mesh; it is denoted by \(U_{\mathrm{conv}}\) in the last column.

\begin{table}[htbp]
\centering
\scriptsize
\setlength{\tabcolsep}{3pt}
\caption{Few-shot ADSC over the mesh-refinement family. The 1000-iteration cap represents the reference coupled ADSC iterate obtained with the prescribed stopping criterion or the iteration cap for each mesh.}
\label{tab:adsc_few_shot}
\begin{tabular}{cccccccc}
\toprule
\(N_e\) & \(Pe_h\) & Max.\ iter. & Used iter. & \(\|e\|_{L^2}\) & \(E_{\mathrm{ext}}\) & \(\bar\rho_{\mathrm{ADSC}}\) & \(\|U-U_{\mathrm{conv}}\|_{0,h}\) \\
\midrule
30  & 8.33 & 5    & 5  & \(8.414\times10^{-3}\) & \(6.428\times10^{-11}\) & 0.834 & \(6.404\times10^{-6}\) \\
30  & 8.33 & 10   & 10 & \(8.416\times10^{-3}\) & \(6.293\times10^{-11}\) & 0.834 & \(5.950\times10^{-7}\) \\
30  & 8.33 & 1000 & 38 & \(8.416\times10^{-3}\) & \(6.278\times10^{-11}\) & 0.834 & 0 \\
\midrule
45  & 5.56 & 5    & 5  & \(6.086\times10^{-3}\) & \(6.523\times10^{-10}\) & 1.124 & \(1.017\times10^{-4}\) \\
45  & 5.56 & 10   & 10 & \(6.122\times10^{-3}\) & \(6.104\times10^{-10}\) & 1.123 & \(9.425\times10^{-6}\) \\
45  & 5.56 & 1000 & 47 & \(6.125\times10^{-3}\) & \(6.100\times10^{-10}\) & 1.123 & 0 \\
\midrule
60  & 4.17 & 5    & 5  & \(4.610\times10^{-3}\) & \(1.389\times10^{-10}\) & 1.227 & \(1.628\times10^{-5}\) \\
60  & 4.17 & 10   & 10 & \(4.616\times10^{-3}\) & \(1.216\times10^{-10}\) & 1.226 & \(1.541\times10^{-6}\) \\
60  & 4.17 & 1000 & 45 & \(4.616\times10^{-3}\) & \(1.205\times10^{-10}\) & 1.226 & 0 \\
\midrule
90  & 2.78 & 5    & 5  & \(2.760\times10^{-3}\) & \(5.430\times10^{-23}\) & 1.397 & \(9.620\times10^{-6}\) \\
90  & 2.78 & 10   & 10 & \(2.764\times10^{-3}\) & \(5.844\times10^{-23}\) & 1.397 & \(9.202\times10^{-7}\) \\
90  & 2.78 & 1000 & 49 & \(2.765\times10^{-3}\) & \(5.840\times10^{-23}\) & 1.397 & 0 \\
\midrule
120 & 2.08 & 5    & 5  & \(1.704\times10^{-3}\) & \(8.047\times10^{-40}\) & 1.534 & \(1.789\times10^{-5}\) \\
120 & 2.08 & 10   & 10 & \(1.713\times10^{-3}\) & \(8.047\times10^{-40}\) & 1.534 & \(3.044\times10^{-6}\) \\
120 & 2.08 & 1000 & 51 & \(1.714\times10^{-3}\) & \(8.047\times10^{-40}\) & 1.533 & 0 \\
\bottomrule
\end{tabular}
\end{table}

\begin{figure}[htbp]
\centering
\includegraphics[width=1\textwidth]{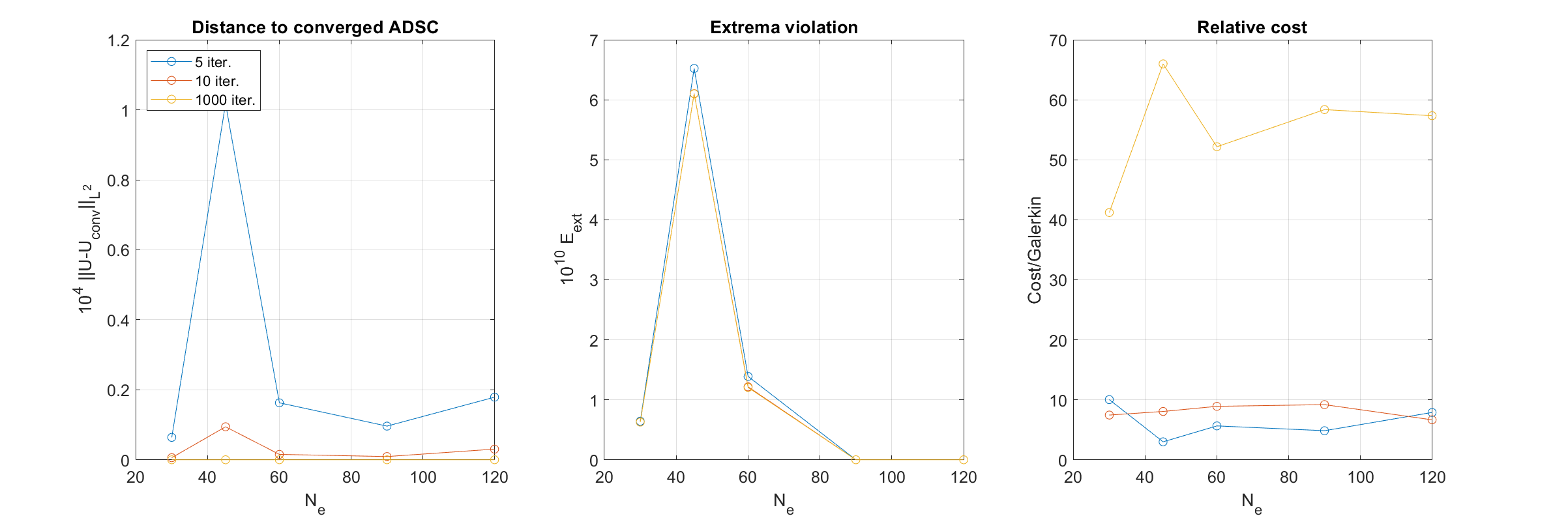}
\caption{Few-shot ADSC behavior over the mesh-refinement family. The left panel shows the distance to the reference coupled ADSC iterate \(U_{\mathrm{conv}}\) obtained with the prescribed stopping criterion or the 1000-iteration cap, the middle panel reports the extrema violation, and the right panel reports the relative cost with respect to the Galerkin solve.}
\label{fig:adsc_few_shot}
\end{figure}

Table~\ref{tab:adsc_few_shot} and Figure~\ref{fig:adsc_few_shot} show that the reduced-activation strategy is stable across the tested mesh-refinement family. Five activation updates already recover the extrema control and modal damping of the reference coupled ADSC iterate, while ten updates are nearly indistinguishable from the converged active-set output. For all tested meshes, five activation updates recover more than 99\% of the extrema-violation reduction at less than 15\% of the full coupled cost (Figure~\ref{fig:adsc_few_shot}, right panel). Thus the cost of the reference active-set procedure should be understood as a reference cost, not as the intrinsic practical cost of ADSC.

\subsection{Mesh refinement and observed error decrease}

\begin{table}[htbp]
	\centering
	\caption{Mesh refinement modal diagnostics in the main two-dimensional test.}
	\label{tab:mesh_diag}
	\scriptsize
	\setlength{\tabcolsep}{3pt}
	\begin{tabular}{ccccccc}
		\toprule
		$N_e$ & $h$                       & $\Peh$                     & Gal.~Det. & ADSC~Det.
		      & $\bar\rho_{\mathrm{Gal}}$ & $\bar\rho_{\mathrm{ADSC}}$ \\
		\midrule
		30    & $3.333\!\times10^{-2}$ & 8.33 & 181 & 8  & 5.359 & 0.834 \\
		45    & $2.222\!\times10^{-2}$ & 5.56 & 262 & 9  & 3.878 & 1.123 \\
		60    & $1.667\!\times10^{-2}$ & 4.17 & 221 & 14 & 3.181 & 1.226 \\
		90    & $1.111\!\times10^{-2}$ & 2.78 & 143 & 5  & 2.575 & 1.397 \\
		120   & $8.333\!\times10^{-3}$ & 2.08 & 111 & 0  & 2.311 & 1.533 \\
		\bottomrule
	\end{tabular}
\end{table}

\begin{table}[htbp]
	\centering
	\caption{Pre-asymptotic SUPG/ADSC inter-level slopes in the main two-dimensional test
	(not convergence rates).
These slopes are computed between consecutive mesh levels; they are numerical
diagnostics for the coupled ADSC iteration, not asymptotic convergence rates
for the fixed-activation operator of Theorem~\ref{thm:adsc_pde_convergence}.}
	\label{tab:mesh}
	\scriptsize
	\setlength{\tabcolsep}{3pt}
	\begin{tabular}{ccccccccc}
		\toprule
		$N_e$ & \multicolumn{2}{c}{SUPG $L^2$}
		      & \multicolumn{2}{c}{ADSC $L^2$}
		      & \multicolumn{2}{c}{SUPG $L^\infty$}
		      & \multicolumn{2}{c}{ADSC $L^\infty$} \\
		\cmidrule(lr){2-3}\cmidrule(lr){4-5}\cmidrule(lr){6-7}\cmidrule(lr){8-9}
		      & error & slope & error & slope & error & slope & error & slope \\
		\midrule
		30    & $6.905\!\times10^{-3}$ & --   & $8.416\!\times10^{-3}$ & --   & $5.534\!\times10^{-2}$ & --   & $4.659\!\times10^{-2}$ & --   \\
		45    & $4.039\!\times10^{-3}$ & 1.32 & $6.125\!\times10^{-3}$ & 0.78 & $4.770\!\times10^{-2}$ & 0.37 & $3.775\!\times10^{-2}$ & 0.52 \\
		60    & $2.662\!\times10^{-3}$ & 1.45 & $4.616\!\times10^{-3}$ & 0.98 & $3.962\!\times10^{-2}$ & 0.65 & $3.070\!\times10^{-2}$ & 0.72 \\
		90    & $1.341\!\times10^{-3}$ & 1.69 & $2.765\!\times10^{-3}$ & 1.26 & $2.605\!\times10^{-2}$ & 1.03 & $2.020\!\times10^{-2}$ & 1.03 \\
		120   & $7.364\!\times10^{-4}$ & 2.08 & $1.714\!\times10^{-3}$ & 1.66 & $1.617\!\times10^{-2}$ & 1.66 & $1.299\!\times10^{-2}$ & 1.54 \\
		\bottomrule
	\end{tabular}
	\par\smallskip
	\footnotesize\emph{Note.} The \(N_e=45\) ADSC \(L^2\) value differs
	slightly from Table~\ref{tab:2d} because Table~\ref{tab:2d} uses the
	\(N_{\rm ref}=264\) reference, whereas this refinement table uses the
	refined-reference post-check described in the text.
\end{table}

Tables~\ref{tab:mesh_diag} and~\ref{tab:mesh} show that the Galerkin modal
indicator decreases as the mesh is refined, as expected from the reduction
of $\Peh$. Across all meshes, ADSC keeps the detector count low and
substantially reduces \(\bar\rho_{\mathrm{stab}}\) relative to
\(\bar\rho_{\mathrm{Gal}}\). The reported rates are observed
rates against the fine-grid reference, not true errors with respect to a
closed-form solution and not a convergence proof. The exact-solution
checks in Sections~\ref{subsec:active_exact}--\ref{subsec:inactive_verification}
are the appropriate convergence diagnostics. SUPG is included as
a linear stabilized comparator: it gives smaller $L^2$ errors and higher
observed $L^2$ rates on this smooth Gaussian test, whereas ADSC is designed
primarily to reduce directional modal dominance and extrema violations.
The lower ADSC rates are consistent with the consistency estimate in
Theorem~\ref{thm:adsc_pde_convergence}: in the active regime the edge-diffusion
term contributes a first-order localized artificial-diffusion residual of
size controlled by the active scale \(|\bfbeta|h\gamma_0\). Thus ADSC is not
expected to display the same smooth-test $L^2$ rates as SUPG while the
detector remains active. As the mesh is refined, \(\Peh\) decreases, the
activation weakens, and the method approaches the centered diffusion-resolved
regime, which explains the increasing observed ADSC rates.
Both SUPG and ADSC show monotone error decrease over the tested meshes,
with improved rates as the mesh P\'eclet number approaches the
diffusion-resolved regime. The
increase of $\bar\rho_{\mathrm{ADSC}}$ under refinement should not be read
as a loss of stability. As $\Peh$ decreases, the activation factor
\(\eta_{\Pe}\) and the activation level both decrease, so less artificial
dissipation is injected and \(\bar\rho_{\mathrm{ADSC}}\) naturally drifts
toward the Galerkin value while errors decrease. The detector count is
grid-dependent and should be read as a local diagnostic, not as a norm.
The small difference between the \(N_e=45\) ADSC \(L^2\) value in
Table~\ref{tab:2d} and Table~\ref{tab:mesh} is due to the reference-grid
check described in the setup: Table~\ref{tab:2d} uses the \(N_{\rm ref}=264\)
reference value, while Table~\ref{tab:mesh} uses the refined-reference
post-check value; the difference is \(O(10^{-6})\).

\subsection{Active-regime verification with an exact solution}
\label{subsec:active_exact}

To separate reference-grid error from the consistency behavior in the
active regime, we also use the exact smooth solution
\[
	u(x,y)=\sin(\pi x)\sin(\pi y),
\]
with \(\eps=2\times10^{-3}\),
\(\bfbeta=(1,0.6)/\sqrt{1+0.6^2}\), and
\(f=-\eps\Delta u+\bfbeta\cdot\nabla u\). For the mesh sizes in
Table~\ref{tab:active_exact}, \(\Peh>1\), so the ADSC baseline is active.
This manufactured test is not meant to reproduce an internal layer; it is a
controlled exact-solution check of the active-regime consistency estimate.

\begin{table}[htbp]
	\centering
	\caption{Manufactured exact-solution verification in the active ADSC
		regime. The ADSC rates are consistent with the conditional
		first-order active-regime estimate, while the centered Galerkin
		discretization remains second order for this smooth solution.}
	\label{tab:active_exact}
	\begin{tabular}{ccccccc}
		\toprule
		\(N_e\) & \(h\) & \(\Peh\) & Gal.~\(L^2\) & rate & ADSC~\(L^2\) & rate \\
		\midrule
		30  & \(3.333\!\times10^{-2}\) & 8.33 & \(9.076\!\times10^{-4}\) & --   & \(2.848\!\times10^{-2}\) & --   \\
		45  & \(2.222\!\times10^{-2}\) & 5.56 & \(4.031\!\times10^{-4}\) & 2.00 & \(1.758\!\times10^{-2}\) & 1.19 \\
		60  & \(1.667\!\times10^{-2}\) & 4.17 & \(2.267\!\times10^{-4}\) & 2.00 & \(1.222\!\times10^{-2}\) & 1.26 \\
		90  & \(1.111\!\times10^{-2}\) & 2.78 & \(1.007\!\times10^{-4}\) & 2.00 & \(7.025\!\times10^{-3}\) & 1.36 \\
		120 & \(8.333\!\times10^{-3}\) & 2.08 & \(5.666\!\times10^{-5}\) & 2.00 & \(4.593\!\times10^{-3}\) & 1.48 \\
		\bottomrule
	\end{tabular}
\end{table}

Table~\ref{tab:active_exact} supports the interpretation of
Theorem~\ref{thm:adsc_pde_convergence}: when the ADSC correction is active,
the added edge diffusion produces a leading \(O(h)\)-type contribution, and
the observed rates approach the first-order-to-intermediate range as
\(\Peh\) decreases. The second-order Galerkin rates in this smooth
manufactured test do not contradict the main Gaussian experiments; they
only show that the exact smooth solution itself is well resolved by the
centered stencil, whereas ADSC deliberately injects artificial diffusion in
the active regime.

\subsection{Inactive-regime verification with an exact solution}
\label{subsec:inactive_verification}

Proposition~\ref{prop:adsc_consistency_fixed} states
that the centered second-order regime is recovered when ADSC is inactive.
We verify this separately on a manufactured smooth solution. Let
\[
	u(x,y)=\sin(\pi x)\sin(\pi y),
\]
with \(\eps=1\), \(\bfbeta=(1,0.6)/\sqrt{1+0.6^2}\), and
\(f=-\eps\Delta u+\bfbeta\cdot\nabla u\). For all meshes in
Table~\ref{tab:inactive_second_order}, one has \(\Peh<1\), hence
\(\gamma_0=\gamma_1=0\) and ADSC coincides exactly with the centered
Galerkin discretization. This test verifies that ADSC exactly recovers the
second-order centered accuracy in the diffusion-resolved regime.

\begin{table}[htbp]
	\centering
	\caption{Manufactured-solution verification in the inactive ADSC regime
		(\(\Peh<1\)). Since \(\gamma_0=\gamma_1=0\), ADSC reduces to the centered
		Galerkin scheme and second-order accuracy is recovered.}
	\label{tab:inactive_second_order}
	\begin{tabular}{ccccccc}
		\toprule
		\(N_e\) & \(h\) & \(\Peh\) & \(\norm{e}_{L^2}\) & rate & \(\norm{e}_{L^\infty}\) & rate \\
		\midrule
		20  & \(5.000\!\times10^{-2}\) & 0.0250  & \(1.039\!\times10^{-3}\) & --   & \(2.078\!\times10^{-3}\) & --   \\
		40  & \(2.500\!\times10^{-2}\) & 0.0125  & \(2.594\!\times10^{-4}\) & 2.00 & \(5.207\!\times10^{-4}\) & 2.00 \\
		80  & \(1.250\!\times10^{-2}\) & 0.00625 & \(6.484\!\times10^{-5}\) & 2.00 & \(1.302\!\times10^{-4}\) & 2.00 \\
		160 & \(6.250\!\times10^{-3}\) & 0.003125& \(1.621\!\times10^{-5}\) & 2.00 & \(3.255\!\times10^{-5}\) & 2.00 \\
		\bottomrule
	\end{tabular}
\end{table}

Table~\ref{tab:inactive_second_order} is consistent with the theoretical inactive-regime
statement: once \(\Peh\le1\), the ADSC correction vanishes and the method
recovers the second-order behavior of the centered stencil on a smooth exact
solution. This also helps interpret Table~\ref{tab:mesh}: the lower ADSC rates
there are not caused by a loss of consistency in the inactive regime, but by
the active localized artificial diffusion used to suppress modal dominance and
extrema violations in the convection-dominated regime.

\subsection{Sensitivity with respect to the right-hand side}

\begin{table}[htbp]
	\centering
	\caption{Sensitivity with respect to the right-hand side, $\Peh=5.56$.}
	\label{tab:rhs}
	\begin{tabular}{llccc}
		\toprule
		Right-hand side   & Method        & $\norm{e}_{L^2}$         & $E_{\mathrm{ext}}$ & $\bar\rho_{\mathrm{stab}}$ \\
		\midrule
		Centered Gaussian & Galerkin      & $4.394\!\times\!10^{-3}$ & $3.700\!\times\!10^{-2}$ & 3.878      \\
		                  & Coord.~upwind & $6.813\!\times\!10^{-3}$ & 0                        & 0.811      \\
		                  & SUPG          & $4.037\!\times\!10^{-3}$ & $2.869\!\times\!10^{-2}$ & 1.298      \\
		                  & CIP-type      & $2.856\!\times\!10^{-3}$ & $2.020\!\times\!10^{-2}$ & 2.560      \\
		                  & LPS-type      & $4.129\!\times\!10^{-3}$ & $3.296\!\times\!10^{-2}$ & 2.536      \\
		                  & AFC-inspired      & $5.869\!\times\!10^{-3}$ & $1.080\!\times\!10^{-5}$ & 1.720      \\
		                  & ADSC          & $6.123\!\times\!10^{-3}$ & $6.100\!\times\!10^{-10}$ & 1.123      \\
		\midrule
		Narrow Gaussian   & Galerkin      & $1.660\!\times\!10^{-3}$ & $5.385\!\times\!10^{-3}$ & 3.878      \\
		                  & Coord.~upwind & $3.437\!\times\!10^{-3}$ & 0                        & 0.811      \\
		                  & SUPG          & $1.543\!\times\!10^{-3}$ & $4.870\!\times\!10^{-3}$ & 1.298      \\
		                  & CIP-type      & $1.163\!\times\!10^{-3}$ & $2.451\!\times\!10^{-3}$ & 2.560      \\
		                  & LPS-type      & $1.580\!\times\!10^{-3}$ & $5.536\!\times\!10^{-3}$ & 2.536      \\
		                  & AFC-inspired      & $3.231\!\times\!10^{-3}$ & $6.860\!\times\!10^{-6}$ & 1.818      \\
		                  & ADSC          & $3.272\!\times\!10^{-3}$ & $4.746\!\times\!10^{-12}$ & 1.007      \\
		\midrule
		Double source     & Galerkin      & $8.420\!\times\!10^{-3}$ & $9.854\!\times\!10^{-2}$ & 3.878      \\
		                  & Coord.~upwind & $1.026\!\times\!10^{-2}$ & 0                        & 0.811      \\
		                  & SUPG          & $7.731\!\times\!10^{-3}$ & $8.288\!\times\!10^{-2}$ & 1.298      \\
		                  & CIP-type      & $5.408\!\times\!10^{-3}$ & $7.050\!\times\!10^{-2}$ & 2.560      \\
		                  & LPS-type      & $7.898\!\times\!10^{-3}$ & $9.094\!\times\!10^{-2}$ & 2.536      \\
		                  & AFC-inspired      & $8.143\!\times\!10^{-3}$ & $9.675\!\times\!10^{-6}$ & 1.428      \\
		                  & ADSC          & $8.640\!\times\!10^{-3}$ & $6.538\!\times\!10^{-10}$ & 1.033      \\
		\bottomrule
	\end{tabular}
\end{table}

Table~\ref{tab:rhs} reports that the modal indicators of the linear
operator-based methods remain unchanged when the source changes, whereas
the ADSC and AFC-inspired indicators vary because their selected activations
depend on the computed solution. SUPG and CIP-type corrections improve the extrema
diagnostic relative to Galerkin, but ADSC and AFC-inspired corrections reduce
it to the $10^{-11}$--$10^{-5}$ level; in the centered Gaussian test, ADSC reduces
$E_{\mathrm{ext}}$ from $3.700\times10^{-2}$ to $6.100\times10^{-10}$.
ADSC remains the most effective of the adaptive methods in
reducing \(\bar\rho_{\mathrm{stab}}\) across the three right-hand sides. The table reports
\(L^2\), \(E_{\mathrm{ext}}\), and \(\bar\rho_{\mathrm{stab}}\) because this subsection
focuses on source-dependent modal and extrema behavior; the main
\(L^\infty\) comparison for the centered Gaussian case is given in
Table~\ref{tab:2d}.

\subsection{Sensitivity with respect to the convection direction}

\begin{table}[htbp]
	\centering
	\caption{Sensitivity with respect to the convection direction.}
	\label{tab:direction}
	\begin{tabular}{llcc}
		\toprule
		Direction        & Method        & Det. & $\bar\rho_{\mathrm{stab}}$ \\
		\midrule
		$(1,0)$          & Galerkin      & 160  & 3.230      \\
		                 & Coord.~upwind & 0    & 0.916      \\
		                 & SUPG          & 176  & 1.352      \\
		                 & CIP-type      & 56   & 2.149      \\
		                 & LPS-type      & 176  & 2.343      \\
		                 & AFC-inspired      & 0    & 2.192      \\
		                 & ADSC          & 0    & 1.136      \\
		\midrule
		$(1,1)/\sqrt{2}$ & Galerkin      & 275  & 3.960      \\
		                 & Coord.~upwind & 9    & 0.803      \\
		                 & SUPG          & 256  & 1.289      \\
		                 & CIP-type      & 258  & 2.611      \\
		                 & LPS-type      & 266  & 2.556      \\
		                 & AFC-inspired      & 12   & 1.701      \\
		                 & ADSC          & 9    & 1.126      \\
		\midrule
		$(2,1)/\sqrt{5}$ & Galerkin      & 255  & 3.833      \\
		                 & Coord.~upwind & 9    & 0.819      \\
		                 & SUPG          & 274  & 1.307      \\
		                 & CIP-type      & 161  & 2.532      \\
		                 & LPS-type      & 273  & 2.528      \\
		                 & AFC-inspired      & 10   & 1.709      \\
		                 & ADSC          & 10   & 1.126      \\
		\bottomrule
	\end{tabular}
\end{table}

\begin{figure}[htbp]
	\centering
	\includegraphics[width=0.98\textwidth]{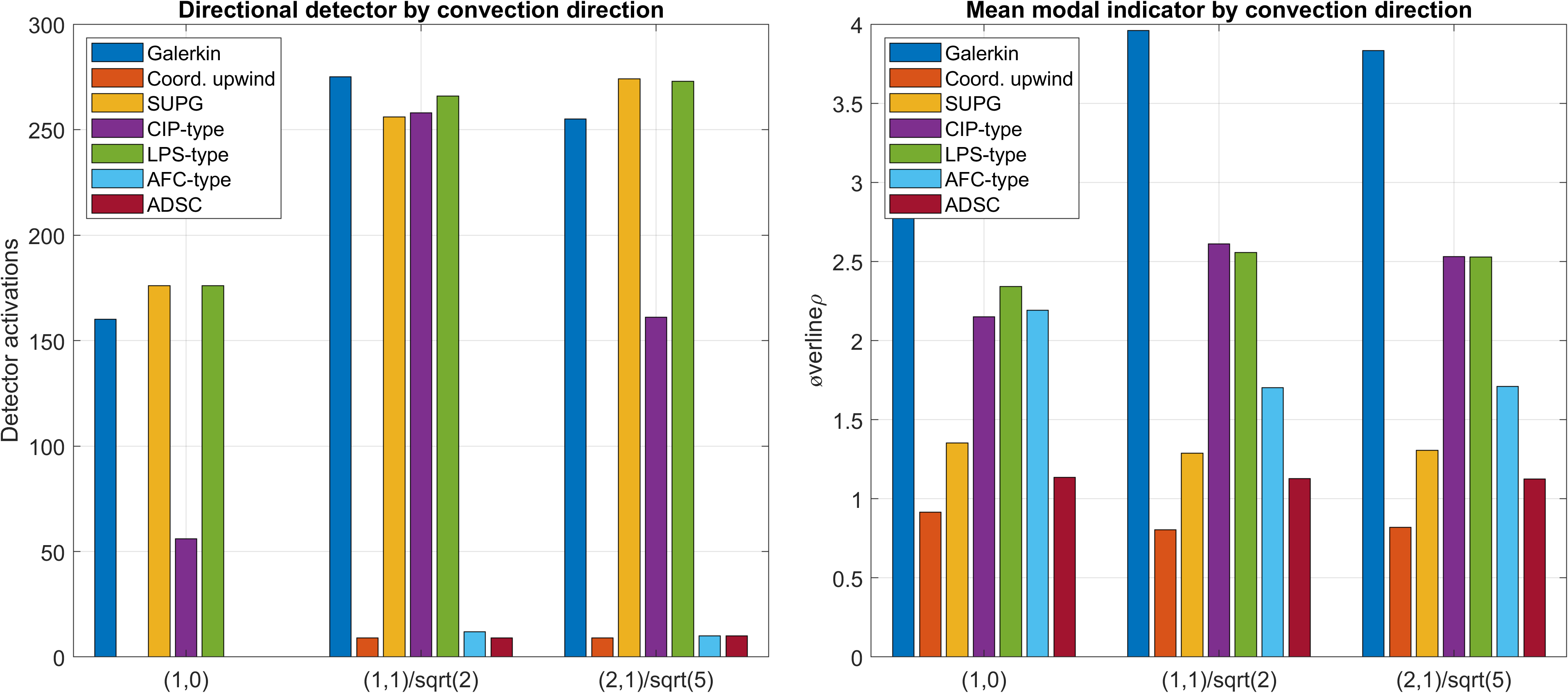}
	\caption{Sensitivity with respect to the convection direction. The
		directional tests show that coordinate upwinding remains the most
		dissipative correction, whereas ADSC provides a selective reduction of
		directional detector activations, especially for oblique directions.
		For $(2,1)/\sqrt5$, the detector count decreases from 255
		(Galerkin) to 10 (ADSC), while \(\bar\rho_{\mathrm{stab}}\) decreases from 3.833 to
		1.126.}
	\label{fig:direction_2d}
\end{figure}

Table~\ref{tab:direction} and Figure~\ref{fig:direction_2d} indicate that the effect of stabilization is
direction dependent. The aligned case \((1,0)\) is closest to the
one-dimensional situation: coordinate upwinding and the monotonicity-driven
corrections eliminate the detector activations. The oblique directions are
the genuinely two-dimensional cases. Coordinate upwinding gives the smallest modal
indicator because it is the most uniformly dissipative correction.
SUPG reduces \(\bar\rho_{\mathrm{stab}}\) but does not necessarily reduce the directional
detector count. CIP-type and LPS-type corrections mainly damp
high-frequency or unresolved components and are less effective on the
present directional detector. ADSC keeps the detector count close to the
upwind and AFC-inspired values while giving a lower \(\bar\rho_{\mathrm{stab}}\) than AFC-inspired
correction in the oblique tests.
These observations are consistent with the directional modal
interpretation developed in the theoretical sections and show how it
translates into the genuinely two-dimensional tests considered here.

\subsection{Additional benchmark: NIST-type layer problem and Shishkin mesh}
\label{subsec:nist_uniform}

We now add a sharp-gradient benchmark in the spirit of the NIST adaptive
elliptic test collection~\citep{Mitchell2010}. The problem is posed on
\(\Omega=(0,1)^2\), with
\[
	\bfbeta=(1/2,\sqrt3/2),\qquad
	u(x,y)=\ell(x)\ell(y),
\]
where
\[
	\ell(t)
	=
	t-\frac{\exp((t-1)/\eps)-\exp(-1/\eps)}
	{1-\exp(-1/\eps)}.
\]
The source term is \(f=-\eps\Delta u+\bfbeta\cdot\nabla u\), and the
boundary condition is homogeneous because \(\ell(0)=\ell(1)=0\). The
solution has exponential outflow layers at \(x=1\) and \(y=1\), with a
corner layer at \((1,1)\). The ADSC parameters are exactly those of the
main test: \(\gamma_{\min}=0.08\), \(\gamma_{\max}=0.25\), \(\kappa=2\),
\(\omega=0.35\), \(\delta_h=10^{-12}\), and
\(\eta_{\rm det}=5\times10^{-2}\). No benchmark-specific retuning is used.
All errors in this subsection are measured directly against this closed
form solution, not against a mesh-dependent numerical reference.

The benchmark tables report Galerkin, coordinate upwinding, SUPG, and ADSC.
These are the methods whose definitions transfer to the nonuniform
Shishkin mesh without introducing extra method-specific calibration. The
CIP-, LPS-, and AFC-inspired edge-diffusion comparators remain part of the main uniform
Cartesian comparison; extending them to a layer-adapted nonuniform mesh
would require additional choices that are separate from the ADSC
validation considered here. The SUPG entries use the same high-P\'eclet
streamline scale as in the main test,
\(\tau=h_{\max}/(2|\bfbeta|)\), together with the residual right-hand side
correction. A one-dimensional exponentially fitted choice would replace
this by
\[
	\tau_{\rm exp}
	=\frac{h}{2|\bfbeta|}
	\left(\coth(\Pe_h)-\frac{1}{\Pe_h}\right),
\]
and a layer-optimized SUPG study would require a separate calibration.
Thus the SUPG rows in Tables~\ref{tab:nist_uniform_2}--\ref{tab:nist_uniform_3}
should be read as the fixed-parameter streamline-diffusion comparator used
throughout this paper, not as an optimized SUPG benchmark for exponential
layers.

\paragraph{Small-diffusion regime.}
The coercivity estimate proved in Section~\ref{sec:adsc} is not uniform
with respect to \(\eps\). Consequently, the tests below with
\(\eps=10^{-3}\), the fixed-grid sweep toward smaller diffusion values,
and the Shishkin experiments should be interpreted as robustness diagnostics
in convection-dominated regimes, not as numerical verification of an
\(\eps\)-uniform stability theorem. In this part of the paper, the main
ADSC quantities of interest are extrema control and reduction of directional
non-monotonicity; the reported errors and inter-level slopes are
pre-asymptotic diagnostics on under-resolved layers.

\begin{table}[htbp]
	\centering
	\scriptsize
	\caption{Uniform NIST-inspired exponential boundary-layer benchmark, \(\eps=10^{-2}\).}
	\label{tab:nist_uniform_2}
	\begin{tabular}{ccclcccc}
		\toprule
		\(N_e\) & \(h_{\max}\) & \(\Pe_{\max}\) & Method &
		\(\norm{e}_{L^2}\) & \(\norm{e}_{L^\infty}\) & \(E_{\mathrm{ext}}\) & Det. \\
		\midrule
		30 & \(3.333\!\times10^{-2}\) & 1.67 & Galerkin & \(2.514\!\times10^{-2}\) & \(2.699\!\times10^{-1}\) & \(2.679\!\times10^{-1}\) & 56 \\
		30 & \(3.333\!\times10^{-2}\) & 1.67 & Upwind   & \(3.578\!\times10^{-2}\) & \(3.580\!\times10^{-1}\) & 0 & 53 \\
		30 & \(3.333\!\times10^{-2}\) & 1.67 & SUPG     & \(1.891\!\times10^{-2}\) & \(1.763\!\times10^{-1}\) & \(1.743\!\times10^{-1}\) & 100 \\
		30 & \(3.333\!\times10^{-2}\) & 1.67 & ADSC     & \(9.686\!\times10^{-3}\) & \(8.557\!\times10^{-2}\) & 0 & 54 \\
		45 & \(2.222\!\times10^{-2}\) & 1.11 & Galerkin & \(1.091\!\times10^{-2}\) & \(1.240\!\times10^{-1}\) & \(5.382\!\times10^{-2}\) & 83 \\
		45 & \(2.222\!\times10^{-2}\) & 1.11 & Upwind   & \(2.912\!\times10^{-2}\) & \(3.020\!\times10^{-1}\) & 0 & 77 \\
		45 & \(2.222\!\times10^{-2}\) & 1.11 & SUPG     & \(1.214\!\times10^{-2}\) & \(1.468\!\times10^{-1}\) & 0 & 153 \\
		45 & \(2.222\!\times10^{-2}\) & 1.11 & ADSC     & \(4.419\!\times10^{-3}\) & \(3.336\!\times10^{-2}\) & 0 & 81 \\
		60 & \(1.667\!\times10^{-2}\) & 0.83 & Galerkin & \(5.911\!\times10^{-3}\) & \(7.254\!\times10^{-2}\) & \(1.454\!\times10^{-2}\) & 110 \\
		60 & \(1.667\!\times10^{-2}\) & 0.83 & Upwind   & \(2.372\!\times10^{-2}\) & \(2.358\!\times10^{-1}\) & 0 & 107 \\
		60 & \(1.667\!\times10^{-2}\) & 0.83 & SUPG     & \(1.134\!\times10^{-2}\) & \(1.382\!\times10^{-1}\) & 0 & 106 \\
		60 & \(1.667\!\times10^{-2}\) & 0.83 & ADSC     & \(5.911\!\times10^{-3}\) & \(7.254\!\times10^{-2}\) & \(1.454\!\times10^{-2}\) & 110 \\
		90 & \(1.111\!\times10^{-2}\) & 0.56 & Galerkin & \(2.492\!\times10^{-3}\) & \(3.128\!\times10^{-2}\) & \(5.084\!\times10^{-3}\) & 167 \\
		90 & \(1.111\!\times10^{-2}\) & 0.56 & Upwind   & \(1.698\!\times10^{-2}\) & \(1.719\!\times10^{-1}\) & 0 & 162 \\
		90 & \(1.111\!\times10^{-2}\) & 0.56 & SUPG     & \(1.022\!\times10^{-2}\) & \(1.113\!\times10^{-1}\) & 0 & 163 \\
		90 & \(1.111\!\times10^{-2}\) & 0.56 & ADSC     & \(2.492\!\times10^{-3}\) & \(3.128\!\times10^{-2}\) & \(5.084\!\times10^{-3}\) & 167 \\
		120 & \(8.333\!\times10^{-3}\) & 0.42 & Galerkin & \(1.363\!\times10^{-3}\) & \(1.623\!\times10^{-2}\) & \(2.991\!\times10^{-3}\) & 222 \\
		120 & \(8.333\!\times10^{-3}\) & 0.42 & Upwind   & \(1.316\!\times10^{-2}\) & \(1.389\!\times10^{-1}\) & 0 & 218 \\
		120 & \(8.333\!\times10^{-3}\) & 0.42 & SUPG     & \(8.996\!\times10^{-3}\) & \(9.226\!\times10^{-2}\) & 0 & 217 \\
		120 & \(8.333\!\times10^{-3}\) & 0.42 & ADSC     & \(1.363\!\times10^{-3}\) & \(1.623\!\times10^{-2}\) & \(2.991\!\times10^{-3}\) & 222 \\
		\bottomrule
	\end{tabular}
\end{table}

\begin{table}[htbp]
	\centering
	\scriptsize
	\caption{Uniform NIST-inspired exponential boundary-layer benchmark, \(\eps=10^{-3}\).}
	\label{tab:nist_uniform_3}
	\begin{tabular}{ccclcccc}
		\toprule
		\(N_e\) & \(h_{\max}\) & \(\Pe_{\max}\) & Method &
		\(\norm{e}_{L^2}\) & \(\norm{e}_{L^\infty}\) & \(E_{\mathrm{ext}}\) & Det. \\
		\midrule
		30 & \(3.333\!\times10^{-2}\) & 16.67 & Galerkin & \(2.142\!\times10^{-1}\) & \(2.200\) & \(2.200\) & 459 \\
		30 & \(3.333\!\times10^{-2}\) & 16.67 & Upwind   & \(6.566\!\times10^{-3}\) & \(8.136\!\times10^{-2}\) & 0 & 56 \\
		30 & \(3.333\!\times10^{-2}\) & 16.67 & SUPG     & \(2.064\!\times10^{-1}\) & \(2.046\) & \(2.046\) & 460 \\
		30 & \(3.333\!\times10^{-2}\) & 16.67 & ADSC     & \(2.215\!\times10^{-2}\) & \(2.668\!\times10^{-1}\) & 0 & 53 \\
		45 & \(2.222\!\times10^{-2}\) & 11.11 & Galerkin & \(1.405\!\times10^{-1}\) & \(1.987\) & \(1.987\) & 707 \\
		45 & \(2.222\!\times10^{-2}\) & 11.11 & Upwind   & \(8.007\!\times10^{-3}\) & \(1.213\!\times10^{-1}\) & 0 & 82 \\
		45 & \(2.222\!\times10^{-2}\) & 11.11 & SUPG     & \(1.337\!\times10^{-1}\) & \(1.803\) & \(1.803\) & 728 \\
		45 & \(2.222\!\times10^{-2}\) & 11.11 & ADSC     & \(1.728\!\times10^{-2}\) & \(2.533\!\times10^{-1}\) & 0 & 81 \\
		60 & \(1.667\!\times10^{-2}\) & 8.33 & Galerkin & \(1.015\!\times10^{-1}\) & \(1.777\) & \(1.777\) & 892 \\
		60 & \(1.667\!\times10^{-2}\) & 8.33 & Upwind   & \(9.127\!\times10^{-3}\) & \(1.586\!\times10^{-1}\) & 0 & 112 \\
		60 & \(1.667\!\times10^{-2}\) & 8.33 & SUPG     & \(9.530\!\times10^{-2}\) & \(1.578\) & \(1.578\) & 943 \\
		60 & \(1.667\!\times10^{-2}\) & 8.33 & ADSC     & \(1.439\!\times10^{-2}\) & \(2.481\!\times10^{-1}\) & 0 & 111 \\
		90 & \(1.111\!\times10^{-2}\) & 5.56 & Galerkin & \(6.132\!\times10^{-2}\) & \(1.406\) & \(1.406\) & 1118 \\
		90 & \(1.111\!\times10^{-2}\) & 5.56 & Upwind   & \(1.081\!\times10^{-2}\) & \(2.261\!\times10^{-1}\) & 0 & 171 \\
		90 & \(1.111\!\times10^{-2}\) & 5.56 & SUPG     & \(5.590\!\times10^{-2}\) & \(1.205\) & \(1.205\) & 1201 \\
		90 & \(1.111\!\times10^{-2}\) & 5.56 & ADSC     & \(1.090\!\times10^{-2}\) & \(2.308\!\times10^{-1}\) & 0 & 172 \\
		120 & \(8.333\!\times10^{-3}\) & 4.17 & Galerkin & \(4.103\!\times10^{-2}\) & \(1.104\) & \(1.104\) & 1174 \\
		120 & \(8.333\!\times10^{-3}\) & 4.17 & Upwind   & \(1.198\!\times10^{-2}\) & \(2.838\!\times10^{-1}\) & 0 & 229 \\
		120 & \(8.333\!\times10^{-3}\) & 4.17 & SUPG     & \(3.621\!\times10^{-2}\) & \(9.157\!\times10^{-1}\) & \(9.157\!\times10^{-1}\) & 1380 \\
		120 & \(8.333\!\times10^{-3}\) & 4.17 & ADSC     & \(8.948\!\times10^{-3}\) & \(2.151\!\times10^{-1}\) & 0 & 232 \\
		\bottomrule
	\end{tabular}
\end{table}

The large SUPG errors in Table~\ref{tab:nist_shishkin_2} should be read in
light of the deliberately fixed parameter choice above. They reflect the use
of a high-P\'eclet streamline scale on a layer-adapted mesh, not an intrinsic
failure of optimized SUPG methods for Shishkin or exponential-layer problems.

\begin{table}[htbp]
	\centering
	\scriptsize
	\caption{Pre-asymptotic \(L^2\) inter-level slopes (not convergence rates)
	on the uniform NIST-type benchmark, \(\eps=10^{-3}\), using the mesh levels
	selected from Table~\ref{tab:nist_uniform_3}.}
	\label{tab:nist_uniform_3_rates}
	\begin{tabular}{ccccc}
		\toprule
		Mesh transition & Galerkin & Upwind & SUPG & ADSC \\
		\midrule
		30--60  & 1.08 & -0.48 & 1.12 & 0.62 \\
		60--90  & 1.24 & -0.42 & 1.31 & 0.68 \\
		90--120 & 1.39 & -0.36 & 1.51 & 0.69 \\
		\bottomrule
	\end{tabular}
	\par\smallskip
	\footnotesize\emph{Note.} These slopes are inter-level pre-asymptotic
	diagnostics in a strongly under-resolved regime. They are not asymptotic
	convergence rates; negative values for upwinding indicate nonmonotone
	coarse-mesh error variation rather than divergence of the method.

\end{table}

Tables~\ref{tab:nist_uniform_2}--\ref{tab:nist_uniform_3} show that the
unretuned ADSC parameters remain effective on exponential boundary layers.
For \(\eps=10^{-2}\), ADSC gives the smallest errors on the
convection-dominated uniform meshes and then coincides with Galerkin once
\(\Pe_{\max}<1\). For \(\eps=10^{-3}\), the uniform grid is strongly
under-resolved; Galerkin and SUPG exhibit large extrema violations, whereas
ADSC removes the extrema violation and remains competitive with upwinding.
Table~\ref{tab:nist_uniform_3_rates} also shows the limitation of this
uniform-grid result: ADSC has only sublinear observed \(L^2\) rates
(about \(0.6\)--\(0.7\)) in the tested range. This should not be presented
as an asymptotic convergence regime. It indicates that, for
\(\eps=10^{-3}\), the uniform grids are primarily testing stabilization of
a severely under-resolved exponential layer. The negative upwind slopes in
Table~\ref{tab:nist_uniform_3_rates} are not caused by a changing
reference solution; the reference is the exact formula above. They are
pre-asymptotic inter-level slopes: the upwind artificial diffusion
\(|\bfbeta|h/2\) decreases as the mesh is refined, but the exact layer of
width \(O(\eps)\) remains unresolved over this range. The smeared upwind
profile can therefore move farther from the exact thin layer before the
asymptotic regime is reached. This also explains why ADSC is worse than
upwinding at \(N_e=60\) but slightly better at \(N_e=120\): ADSC decreases
monotonically in the displayed range, whereas the upwind \(L^2\) error is
still pre-asymptotic and nonmonotone. The table should therefore be read
as a robustness diagnostic for under-resolution, not as a convergence
study for upwind or ADSC.

\paragraph{Fixed-grid small-diffusion sweep.}
As a compact robustness check, we repeated the uniform NIST-type experiment
on the fixed grid \(N_e=64\) for
\(\eps\in\{10^{-1},10^{-2},10^{-3},10^{-4},10^{-5}\}\). For
\(\eps=10^{-1}\) and \(10^{-2}\), ADSC is inactive because
\(\Pe_{\max}<1\) and therefore coincides with the centered scheme. For
\(\eps=10^{-3},10^{-4},10^{-5}\), Galerkin and SUPG exhibit large extrema
violations, whereas ADSC gives \(E_{\mathrm{ext}}=0\) in the tested cases.
Coordinate upwinding remains more accurate in \(L^2\) on the most
under-resolved layers. The sweep therefore supports the same conclusion as
Tables~\ref{tab:nist_uniform_2}--\ref{tab:nist_uniform_3}: ADSC acts as a
selective directional damping and extrema-control mechanism, not as an
\(\eps\)-uniformly error-optimal replacement for upwinding.

\subsection{NIST-type layer problem on a Shishkin mesh}
\label{subsec:nist_shishkin}

To address the restriction to uniform grids without leaving the Cartesian
setting, we also test a tensor-product Shishkin mesh. For a layer at
\(x=1\), the transition point is
\[
	\tau=\min\{1/2,\,2\eps\log N_e\},
\]
with half of the intervals placed on \([0,1-\tau]\) and half on
\([1-\tau,1]\), and the same construction in the \(y\)-direction. The ADSC
weights are evaluated with a local edge P\'eclet number
\(\Pe_e=|\bfbeta|h_e/(2\eps)\), using the same saturating law for
\(\gamma_0\).

Before reading Table~\ref{tab:nist_shishkin_2}, one should keep in mind
that the SUPG rows are not produced with a Shishkin- or
exponential-layer-optimized stabilization parameter. They use the same
fixed high-P\'eclet streamline scale as in the preceding uniform-grid
benchmark, with the residual right-hand-side correction. On a
layer-adapted mesh this choice may over-stabilize the coarse subregion and
interact unfavorably with the already resolved fine layer. The table is
therefore a robustness comparison under a common fixed-parameter setup,
not a claim that this is an optimized SUPG implementation for Shishkin
meshes or exponential boundary layers.

\begin{table}[htbp]
	\centering
	\scriptsize
	\caption{Shishkin-mesh NIST-inspired exponential boundary-layer benchmark, \(\eps=10^{-2}\).}
	\label{tab:nist_shishkin_2}
	\begin{tabular}{ccclcccc}
		\toprule
		\(N_e\) & \(h_{\max}\) & \(\Pe_{\max}\) & Method &
		\(\norm{e}_{L^2}\) & \(\norm{e}_{L^\infty}\) & \(E_{\mathrm{ext}}\) & Det. \\
		\midrule
		30 & \(6.213\!\times10^{-2}\) & 3.11 & Galerkin & \(2.895\!\times10^{-3}\) & \(2.254\!\times10^{-2}\) & \(2.061\!\times10^{-2}\) & 41 \\
		30 & \(6.213\!\times10^{-2}\) & 3.11 & Upwind   & \(7.765\!\times10^{-3}\) & \(8.256\!\times10^{-2}\) & 0 & 39 \\
		30 & \(6.213\!\times10^{-2}\) & 3.11 & SUPG     & \(2.108\!\times10^{-1}\) & \(1.141\) & \(6.186\!\times10^{-1}\) & 219 \\
		30 & \(6.213\!\times10^{-2}\) & 3.11 & ADSC     & \(8.061\!\times10^{-4}\) & \(5.635\!\times10^{-3}\) & 0 & 43 \\
		60 & \(3.060\!\times10^{-2}\) & 1.53 & Galerkin & \(2.531\!\times10^{-4}\) & \(2.516\!\times10^{-3}\) & \(1.435\!\times10^{-3}\) & 86 \\
		60 & \(3.060\!\times10^{-2}\) & 1.53 & Upwind   & \(4.713\!\times10^{-3}\) & \(5.192\!\times10^{-2}\) & 0 & 86 \\
		60 & \(3.060\!\times10^{-2}\) & 1.53 & SUPG     & \(9.910\!\times10^{-2}\) & \(8.770\!\times10^{-1}\) & \(1.421\!\times10^{-1}\) & 382 \\
		60 & \(3.060\!\times10^{-2}\) & 1.53 & ADSC     & \(5.903\!\times10^{-4}\) & \(4.942\!\times10^{-3}\) & 0 & 86 \\
		90 & \(2.022\!\times10^{-2}\) & 1.01 & Galerkin & \(9.302\!\times10^{-5}\) & \(1.046\!\times10^{-3}\) & \(3.675\!\times10^{-4}\) & 137 \\
		90 & \(2.022\!\times10^{-2}\) & 1.01 & Upwind   & \(3.468\!\times10^{-3}\) & \(3.869\!\times10^{-2}\) & 0 & 132 \\
		90 & \(2.022\!\times10^{-2}\) & 1.01 & SUPG     & \(5.611\!\times10^{-2}\) & \(6.552\!\times10^{-1}\) & \(5.525\!\times10^{-2}\) & 291 \\
		90 & \(2.022\!\times10^{-2}\) & 1.01 & ADSC     & \(2.693\!\times10^{-4}\) & \(2.222\!\times10^{-3}\) & 0 & 137 \\
		120 & \(1.507\!\times10^{-2}\) & 0.75 & Galerkin & \(5.276\!\times10^{-5}\) & \(6.151\!\times10^{-4}\) & \(1.581\!\times10^{-4}\) & 189 \\
		120 & \(1.507\!\times10^{-2}\) & 0.75 & Upwind   & \(2.777\!\times10^{-3}\) & \(3.115\!\times10^{-2}\) & 0 & 183 \\
		120 & \(1.507\!\times10^{-2}\) & 0.75 & SUPG     & \(3.693\!\times10^{-2}\) & \(5.197\!\times10^{-1}\) & \(1.805\!\times10^{-2}\) & 274 \\
		120 & \(1.507\!\times10^{-2}\) & 0.75 & ADSC     & \(5.276\!\times10^{-5}\) & \(6.151\!\times10^{-4}\) & \(1.581\!\times10^{-4}\) & 189 \\
		\bottomrule
	\end{tabular}
\end{table}

\begin{table}[htbp]
	\centering
	\small
	\caption{Local P\'eclet numbers on the Shishkin meshes used in Table~\ref{tab:nist_shishkin_2}.}
	\label{tab:shishkin_pe}
	\begin{tabular}{ccccccc}
		\toprule
		\(\eps\) & \(N_e\) & \(\tau\) & \(h_c\) & \(h_f\) & \(\Pe_c\) & \(\Pe_f\) \\
		\midrule
		\(10^{-2}\) & 30  & \(6.802\!\times10^{-2}\) & \(6.213\!\times10^{-2}\) & \(4.535\!\times10^{-3}\) & 3.107 & 0.227 \\
		\(10^{-2}\) & 60  & \(8.189\!\times10^{-2}\) & \(3.060\!\times10^{-2}\) & \(2.730\!\times10^{-3}\) & 1.530 & 0.136 \\
		\(10^{-2}\) & 90  & \(9.000\!\times10^{-2}\) & \(2.022\!\times10^{-2}\) & \(2.000\!\times10^{-3}\) & 1.011 & 0.100 \\
		\(10^{-2}\) & 120 & \(9.575\!\times10^{-2}\) & \(1.507\!\times10^{-2}\) & \(1.596\!\times10^{-3}\) & 0.754 & 0.080 \\
		\bottomrule
	\end{tabular}
\end{table}

Table~\ref{tab:nist_shishkin_2} indicates that the layer-adapted mesh
improves the resolution of the exponential layer. Table~\ref{tab:shishkin_pe}
also illustrates the intended local behavior: ADSC can remain active in the
coarse subregion where \(\Pe_c>1\), while it is inactive in the refined
layer region where \(\Pe_f<1\). Thus the method does not force artificial
diffusion in the fine layer region where the Shishkin mesh already resolves
the sharp gradient. At the same time, the table exposes a limitation of the
current local-P\'eclet activation on layer-adapted meshes. For \(N_e=60\) and
\(N_e=90\), Galerkin already has small \(L^2\) and \(L^\infty\) errors on
the Shishkin grid, while ADSC removes the remaining extrema violation at
the price of extra diffusion and therefore larger norm errors. This is
consistent with the local coarse-region values \(\Pe_c>1\): the correction
can still be activated in the coarse part of the Shishkin mesh even though
the fine region has already controlled the layer error. Consequently, the
Shishkin experiment should not be read as an accuracy advantage of ADSC
over Galerkin on layer-adapted meshes. It shows a tradeoff: extrema control
can be obtained without retuning, but the current local-P\'eclet switch may
be too diffusive when an adapted mesh has already suppressed the dominant
oscillatory error.

\paragraph{Frozen-coefficient validity on Shishkin meshes.}
On a Shishkin mesh, the local mesh step varies between \(h_c\) and \(h_f\),
so the global translation-invariance used by the interior LFA symbol is
lost. The modal-balance parameter \(\gamma_0\) is therefore evaluated
locally through \(\Pe_e=|\bfbeta|h_e/(2\eps)\), and the LFA argument is used
only as a frozen-coefficient sizing rule. The nonsingularity and coercivity
arguments for ADSC remain valid for any Cartesian edge operator
\(D_x^TW_xD_x+D_y^TW_yD_y\) with nonnegative weights; what is lost is the
exact diagonalization by tensor-product Fourier or sine modes. Thus the
Shishkin test supports the robustness assessment of the sparse directional mechanism
on a structured nonuniform mesh, but it is not a proof for unstructured
simplicial finite elements.

\subsection{Summary of numerical conclusions}

The numerical evidence can be summarized in four points. First, in the main
two-dimensional Cartesian test, ADSC reduces extrema violations and the mean
modal indicator relative to the centered Galerkin scheme, while retaining a
less uniformly smoothed profile than coordinate upwinding. Second, the
fixed-reference activation test illustrates the fixed-activation framework of
the conditional theorem, while the parameter, initialization, and few-shot
studies show that the selected parameters are conservative rather than
error-optimal and that a small number of activation updates already recovers
the main extrema-control and modal-damping effects. Third, mesh-refinement
and manufactured-solution tests show observed ADSC error decrease in the
active regime and recover the centered second-order behavior when ADSC is
inactive. Fourth, the NIST-type and Shishkin experiments provide robustness diagnostics for
the local directional mechanism on sharp-gradient configurations, while
confirming the limitations stated at the beginning of
Section~\ref{subsec:nist_uniform}.


\section{Conclusion}


This paper has developed a modal-rectification inspired framework for local
directional edge diffusion on uniform Cartesian constant-coefficient
convection--diffusion discretizations with homogeneous Dirichlet conditions.
The starting point is the local Fourier symbol of the centered stencil, from
which a mode-resolved convection-dominance indicator and an anisotropic modal
footprint are obtained. The constructed rectified modal reference operator is
then shown to be nonlocal in the nodal basis, which explains why exact modal
rectification cannot be implemented by a fixed-width Cartesian stencil.

ADSC is the corresponding local surrogate studied here: a nearest-neighbor,
\(\bfbeta\)-directional, symmetric positive semidefinite edge-diffusion
correction with regularized activation. For the fixed-activation regularized
operator selected from an auxiliary smooth activation sequence, the paper
proves consistency, fixed-\(\eps\) energy stability, and conditional
discrete \(H^1\)-seminorm convergence. For the regularized fully coupled formulation, in
which the activation is generated by the unknown stabilized solution itself,
Propositions~\ref{prop:adsc_coupled_existence}--\ref{prop:adsc_coupled_qualitative_convergence}
establish existence for each fixed mesh and qualitative compactness/convergence
in \(L^2\) toward the weak solution of the continuous problem. What remains open is the
uniqueness of the nonlinear coupled discrete solution for fixed \(h\), an
energy-norm convergence rate for the fully coupled problem, and convergence
of the activation-update algorithm toward such a coupled solution. Among
these issues, the most immediate analytical step is an energy-norm estimate
for the coupled regularized problem under verifiable activation-stability
assumptions; the algorithmic convergence question is logically separate
because it concerns the practical fixed-point iteration rather than the
existence of a coupled discrete solution.

The computations support the intended role of ADSC as a selective
extrema-control and modal-damping mechanism, rather than as a uniformly
error-optimal replacement for established stabilizations. The fixed-reference
activation test provides a numerical illustration of the fixed-activation
operator used in the conditional analysis, while the few-shot study indicates
that most of the practical effect is already obtained after a small number of
activation updates. A fixed-grid sweep in \(\eps\) further indicates that ADSC
can suppress extrema violations in severely convection-dominated tests, even
though it is not an \(\eps\)-uniformly error-optimal substitute for upwinding.
The Shishkin experiments also show a practical limitation: when a
layer-adapted mesh has already controlled the dominant layer error, the
current local-P\'eclet switch can be too diffusive and may increase norm
errors while removing the remaining extrema violation.
Extensions to triangular finite elements,
variable convection fields, calibrated AFC comparisons, oblique internal
layers, and \(\eps\)-uniform analysis are natural continuations. Recent AFC
and LPS developments, including high-order Bernstein finite elements,
convex-limiting sensors, and \( \boldsymbol H(\mathrm{curl})\) /
\(\boldsymbol H(\mathrm{div})\) LPS formulations
\citep{Hajduk2025,KuzminHajdukVedral2025,LuoWangWu2025}, illustrate that
these extensions require additional discretization-specific analysis rather
than a direct transplantation of the present Cartesian symbol framework.
Overall, the contribution is a modal design route for local directional
edge diffusion on Cartesian grids, together with a conditional
fixed-activation analysis and numerical evidence of selective extrema
control.

\section*{Notation summary}
\begin{center}
\small
\begin{tabular}{L{0.24\textwidth}L{0.68\textwidth}}
\toprule
Symbol & Meaning \\
\midrule
\(\eps,\bfbeta,h,N_e,N\) & Diffusion coefficient, constant convection vector, mesh size,
number of intervals per direction, and number of interior nodes per direction. \\
\(\beta_r^+,\beta_r^-\) & Positive and negative parts of the \(r\)-th convection component,
\(\beta_r^+=\max(\beta_r,0)\) and \(\beta_r^-=\min(\beta_r,0)\). \\
\(\theta_p,\theta_q\) & Discrete modal frequencies associated with the interior sine modes. \\
\(\lhpq=a_{pq}+\mathrm{i} b_{pq}\) & Interior-stencil local Fourier symbol, split into diffusive and
convective contributions. \\
\(K_h^{\rm int}\) & Translation-invariant interior stencil operator used in local Fourier analysis, distinct from the finite Dirichlet matrix \(K_h\). \\
\(\rho_{pq}=|b_{pq}|/a_{pq}\) & Modal convection-dominance indicator. \\
\(\varphi_{pq}=\arctan\rho_{pq}\) & Modal angle used in the isotropic rectification diagnostic. \\
\(\Icd\) & Galerkin convectively dominant modal set, defined by \(\rho_{pq}>1\). \\
\(\bar\rho_{\mathrm{Gal}}\) & Mean Galerkin modal indicator over \(\Icd\). \\
\(\bar\rho_{\mathrm{stab}}\) & Mean stabilized modal indicator computed over the same Galerkin-dominant set \(\Icd\). \\
\(\phi_{pq}\) & Tensor-product sine mode satisfying the homogeneous Dirichlet boundary condition. \\
\(\psi_{\theta}\) & Interior local Fourier mode used only for the translation-invariant LFA symbol. \\
\(\Lstar\) & Ideal rectified modal reference operator with eigenvalues \(|\lhpq|\). \\
\(D_x,D_y,W_x,W_y\) & Edge-difference matrices and nonnegative ADSC edge weights. \\
\(\|\cdot\|_{0,h}\) & Discrete \(L^2\)-norm on the Cartesian grid, \(\|V\|_{0,h}^2=h^2\sum_{i,j}|V_{ij}|^2\). \\
\(|\cdot|_{1,h}\) & Discrete \(H^1\)-seminorm induced by nearest-neighbor differences, \(|V|_{1,h}^2=\|D_xV\|_{0,h}^2+\|D_yV\|_{0,h}^2\). \\
\(\|e\|_{L^2},\|e\|_{L^\infty}\) & Numerical error norms computed against the fine-grid reference solution interpolated on the tested grid. \\
\(E_{\mathrm{ext}}\) & Total violation of the reference extrema, equal to the sum of the undershoot and overshoot. \\
Det. & Number of grid nodes detected as directionally non-monotone by the post-processing detector. \\
\(TV\) & Discrete total variation diagnostic used only in the numerical comparison. \\
\(U_{\mathrm{ref}},U_{\mathrm{conv}}\) & Fine-grid reference solution and reference coupled ADSC iterate, respectively. \\
\(\chi,\chi^x,\chi^y\) & Nodal and edge activation indicators selected by the directional detector. \\
\(\gamma_0,\gamma_1,\eta_{\Pe}\) & Baseline ADSC scale, active reinforcement scale, and Pe-dependent activation factor. \\
\(\bar B\) & Averaged modal-balance factor used in the conditional scaling argument. \\
\(\gamma_0^{\mathrm{bal}}\) & Modal-balance target before projection or conservative saturation. \\
\bottomrule
\end{tabular}
\end{center}

\paragraph{Declaration of competing interest.}
Declarations of interest: none.

\paragraph{Funding.}
This research received no external funding.

\paragraph{Code and data availability.}
The MATLAB codes used to produce the results in this paper are openly available
on Zenodo at \url{https://doi.org/10.5281/zenodo.20342147}~\citep{BomissoKoumaZenodo2026}. No additional datasets were generated or analyzed during the current study. The repository contains a README file and an audit file describing the correspondence between scripts, figures, and tables. The main entry points are:
\begin{center}
\scriptsize
\begin{tabular}{L{0.43\linewidth}L{0.49\linewidth}}
\toprule
Script & Purpose \\
\midrule
\path{ADSC_Main_2D_Benchmark.m} & Main Cartesian benchmark and principal tables. \\
\path{ADSC_CommonScale_Figures.m} & Common-scale solution and error-map figures. \\
\path{ADSC_FewShot_MeshRefinement.m} & Few-shot ADSC mesh-refinement diagnostics. \\
\path{ADSC_Manufactured_ActiveRegime_Test.m} & Manufactured active-regime consistency test. \\
\path{ADSC_NIST_Shishkin_Benchmark.m} & NIST-type uniform and Shishkin benchmarks. \\
\path{ADSC_EpsilonSweep_NIST.m} & Fixed-grid epsilon-sweep robustness diagnostics. \\
\bottomrule
\end{tabular}
\end{center}
These scripts generate the numerical tables and figures reported in
Section~\ref{sec:numerics}, including the main Cartesian benchmark, the
fixed-reference and few-shot ADSC tests, the manufactured active-regime test,
and the NIST-type uniform and Shishkin boundary-layer benchmarks. The repository also contains a manifest and an MIT license.

\end{document}